\documentclass[english]{amsart}
\usepackage[T1]{fontenc}
\usepackage[latin9]{inputenc}
\usepackage{geometry}
\usepackage{color}
\usepackage{babel}
\usepackage{amstext}
\usepackage{amsthm}
\usepackage{amssymb}
\usepackage{graphicx}
\usepackage[unicode=true,pdfusetitle,
 bookmarks=true,bookmarksnumbered=false,bookmarksopen=false,
 breaklinks=false,pdfborder={0 0 0},pdfborderstyle={},backref=false,colorlinks=true]
 {hyperref}
\hypersetup{
 linkcolor=brown, citecolor=blue, pdfstartview={FitH}, hyperfootnotes=false, unicode=true}

\makeatletter

\numberwithin{equation}{section}
\numberwithin{figure}{section}
\theoremstyle{plain}
\newtheorem{thm}{\protect\theoremname}[section]
\theoremstyle{definition}
\newtheorem{defn}[thm]{\protect\definitionname}
\theoremstyle{plain}
\newtheorem{lem}[thm]{\protect\lemmaname}
\theoremstyle{definition}
\newtheorem{problem}[thm]{\protect\problemname}
\theoremstyle{plain}
\newtheorem{prop}[thm]{\protect\propositionname}
\theoremstyle{remark}
\newtheorem{rem}[thm]{\protect\remarkname}
\theoremstyle{definition}
\newtheorem{example}[thm]{\protect\examplename}
\theoremstyle{plain}
\newtheorem{fact}[thm]{\protect\factname}

\usepackage{tikz-cd}
\usepackage{bbm}

\usepackage{changepage}
\usepackage{stackengine}      

\makeatother

\providecommand{\definitionname}{Definition}
\providecommand{\examplename}{Example}
\providecommand{\factname}{Fact}
\providecommand{\lemmaname}{Lemma}
\providecommand{\problemname}{Problem}
\providecommand{\propositionname}{Proposition}
\providecommand{\remarkname}{Remark}
\providecommand{\theoremname}{Theorem}

\begin{document}
\title[New definitions for type 1 computable topological spaces]{New definitions in the theory of Type 1 computable topological spaces}
\author{Emmanuel Rauzy}
\thanks{The author is funded by an Alexander von Humboldt Research Fellowship.}
\begin{abstract}
In 1957, Lacombe initiated a systematic study of the different possible
notions of  ``computable topological spaces''. However, he interrupted
this line of research, settling for the idea that ``computably open
sets should be computable unions of basic open sets''. We explain
the limits of this approach, which in particular is not general enough
to account for all spaces that admit a computable metric. 

We give a general notion of Type 1 computable topological space that
does not rely on a notion of effective basis. Building on the work
of Spreen, we show that the use of a \emph{formal inclusion relation}
should be systematized. We give the first general definition of the
computable topology associated to a computable metric that does not
rely on effective separability. This definition can be translated
to other constructive settings, and its relevance goes beyond that
of Type 1 computability. 

Finally, we give a new version of a theorem of Moschovakis, by showing
that for an appropriate notion of effective basis, the ``computably
open sets'' reduce to ``computable unions of basic open sets''
on computably separable spaces. 
\end{abstract}

\maketitle
\tableofcontents{}

\section{\label{sec:Introduction}Introduction}

In computable mathematics, each type of structure classically considered
in standard mathematics admits an effective counterpart. The notions
of \emph{computable graph}, \emph{computable group}, \emph{computable
real number} are thus well studied and widely accepted as the correct
equivalents to the notions of \emph{graph, }of\emph{ group }and of\emph{
real number}. 

The purpose of the present article is to explain why previously existing
notions of \emph{computable topological spaces} are unsatisfactory,
and to propose, in the context of Type 1 computability, a new, more
general and more robust notion of Type 1 computable topological space,
for instance general enough to be able to guarantee that a computable
metric will always yield a computable topology. 

\bigskip

The study of the interplay between computability and topology goes
back to the 1950's. Indeed, different notions of ``computable functionals'',
or of ``computable higher order functions'' had been proposed by
Turing \cite{Turing1937,Turing1938}, Kleene \cite{Kleene1952}, Grzegorczyk
\cite{Grzegorczyk1957}, Markov \cite{Markov1963}, and others, and
it turned out that the comparison between these different notions
systematically involved the study of topology: the Rice-Shapiro Theorem
\cite{Rogers1987}, the Myhill-Shepherdson Theorem \cite{Myhill1955}
and the Kreisel-Lacombe-Shoenfield Theorem \cite{Kreisel1957} all
amount to establishing that certain ``computable functions'' are
``effectively continuous''. 

To unify these results in a common framework, one has to devise a
general setting in which it would make sense to talk about ``effectively
continuous functions'' and ``computably open sets''. That is, one
has to give a definition of a \emph{computable topological space}. 

\bigskip

In 1957, in \cite{Lacombe1957b}, Daniel Lacombe began the systematic
study of computable topological spaces, and the search for a correct
set of ground definitions and theorems. This article is unfortunately
incomplete, in that it does not contain the central definition of
a ``recursive based space'': the definition is omitted ``due to
lack of space''. 

Nevertheless, it provides a clear understanding of the issues that
Lacombe deemed crucial for the development of ``recursive topology'':
\begin{itemize}
\item The problem that is discussed most in \cite{Lacombe1957b} is the
one of defining a Cartesian closed category of topological spaces
where one would be able to study both continuity and computability:
what is the (computable) topology to be considered on sets of continuous
functions between topological spaces? 
\item A second problem considered is the problem of defining different notions
of ``complete spaces'': different kinds of ``approximating sequences''
are described, and the problem of guaranteeing that these approximating
sequences always converge to a point is discussed. 
\item And finally, Lacombe remarks (we translate) ``A similar question
(but slightly more complicated) arises when we agree to restrict functional
definitions to the sets $E^{*}$ and $E_{1}^{*}$ of recursive points
of $E$ and $E_{1}$''.\footnote{``Une question analogue (mais un peu plus compliquée) se pose lorsqu'on
convient de restreindre les définitions fonctionnelles aux ensembles
$E^{*}$ et $E_{1}^{*}$ constitués respectivement par les éléments
récursifs de $E$ et de $E_{1}$.''}
\end{itemize}
There is by now a much better understanding of the first and second
problems raised by Lacombe, and we think it is fair to argue that
these problems are now solved. 

Indeed, the first problem above is now well understood thanks to Schröder's
classification of the spaces that admit admissible representation
(the qc$\text{b}_{0}$ spaces), which form a Cartesian closed category
on which computability and continuity can be studied jointly. See
\cite{Schroeder2003,Schroeder2021}. 

The second problem can also be considered solved. In \cite[Section 6.5]{Lacombe1957b}
is sketched a notion of ``abstract basis'', which would be a structure
meant to ``abstractly reproduce the topological relations defined
on bases of open sets''. These ``abstract bases'' could then be
used to construct completion of various spaces. We think that what
Lacombe hinted at corresponds exactly to Mathew de Brecht's notion
of \emph{quasi-Polish space} \cite{Brecht2013}, and the desired completion
is the one given in \cite{Brecht2020a}. 

The present article answers the third question raised by Lacombe,
by giving the ground definitions and theorems that are required to
study Type 1 computable topology. Note that in Section \ref{subsec:Remark-on-Type-2}
of this introduction we explain why the considerable developments
that arose in the study of various forms of constructive analysis
(Type 2 computable analysis, synthetic topology, abstract stone duality)
do not directly translate to a solution of this third problem -thus
explaining that there is indeed something that is ``more complicated''
when we deal only with finite descriptions. 

\bigskip

In the article of Lacombe that followed \cite{Lacombe1957b}, which
is \cite{Lacombe1957}, discussions about the correct notion of ``computable
topological space'' did not appear anymore. Instead, Lacombe settled
for the following setting: a topological space with a basis equipped
with a total numbering, in which the ``computably open sets'' would
be \emph{c.e. unions of basic open sets}.\footnote{Note also that, in the incomplete definition that appears in \cite{Lacombe1957b},
the classical definition of a basis for a topology is expressed in
a way which would naturally lead to the notion of computable basis
that was later introduced by Nogina \cite{Nogina1966}, and thus to
a notion different from that of \cite{Lacombe1957}.} These sets, computable unions of basic open sets, were later called
\emph{Lacombe sets} by Lachlan in \cite{Lachlan1964} and Moschovakis
in \cite{Moschovakis1964}. 

We thus refer ourselves to the \emph{Lacombe approach} whenever discussing
computable topologies based on the idea that computably open sets
are c.e. unions of basic open sets. 

\bigskip

Let us briefly explain what is unsatisfying with Lacombe's final approach.

The first obvious problem is that ``computable topologies'' were
never defined: what Lacombe defined is a notion of ``computable basis''.
And such a definition cannot be used without a general definition
of a computable topological space, which would indicate \emph{what
it is the basis should generate}. 

More importantly, this approach to computable topology is not general
enough. For instance, it is not possible to define the computable
topology coming from a computable metric\emph{ }if we take ``computably
open sets'' to be c.e. unions of basic sets. (This will be discussed
in details in Section \ref{sec:Problems-with-Lacombe-bases}.) The
Lacombe approach is only general enough to deal with effectively separable
metric spaces, i.e. those that admit a dense and computable sequence.
The present article stems from the author's work on the space of marked
groups, which is a Polish space which, equipped with its natural representation,
is not computably Polish because it is not computably separable \cite{RauzyV}.
On this space, one cannot define the ``computably open sets'' to
be the Lacombe sets. 

In fact, a very important but little known theorem of Moschovakis
from \cite{Moschovakis1964} states that, on recursive Polish spaces
(i.e. on the set of computable points of a computably Polish space),
two notions of computably open sets agree: if $O$ is a semi-decidable
set which admits a computable multi-function $F$ that, given $x\in O$,
produces a computable real $r$ such that $B(x,r)\subseteq O$, then
the set $O$ is in fact a Lacombe set. (Furthermore, the equivalence
is effective, it is possible to effectively pass from one description
of $O$ to the other one.) Several other versions of this theorem
exist in the context of Type 2 computability: see \cite[Theorem 3.10]{Brattka2003}
and \cite[Proposition 3.3]{Gregoriades2016}. Note that these theorems
are easier to obtain in the context of Type 2 computability than in
Type 1 computability. 

We summarize the above in the following sentence:

\vskip 0.2cm

\begin{adjustwidth}{1cm}{1cm}

In general, there are more ``computably open sets'' than ``computable
unions of basic open sets''. And there are ``Moschovakis-like Theorems''
which give sufficient conditions for those two notions to become equivalent. 

\end{adjustwidth}

\vskip 0.2cm

And thus the notion of ``computable topological space'' considered
today in the literature, which is Weihrauch and Grubba's formalization
of Lacombe's approach \cite{Weihrauch2009ElementaryCT}, is not a
general notion: it applies precisely to those spaces for which a Moschovakis
Theorem exists. 

\bigskip

In the present article, we give the first general definition of a
Type 1 computable topological space, and the first general definition
of the computable topology generated by a computable metric. As we
will see, the second point is non-trivial. We rely on a notion of
formal inclusion relation and on ideas that go back to Dieter Spreen
\cite{Spr98}. 

In Theorem \ref{thm:Moschovakis-theorem-on Spreen bases}, we give
a new version of Moschovakis' Theorem, which is more general than
the previous one, as it relies on effective separability as sole hypothesis.
The proof of this new theorem is in fact very easy: we show that by
using an elaborate notion of effective basis, the difficulty of the
proof disappears. 

\bigskip

The relevance of our study of the topology generated by a metric goes
well beyond the scope of Type 1 computability: while we are only using
one of several possible formalisms to study computable topologies,
our results easily translate to the study of metric spaces in realizability
\cite{Bauer2000}, in synthetic topology \cite{Bauer2012}, in Type
2 computability \cite{Iljazovic2021}. For instance, a theorem of
Matthias Schröder \cite{Schroeder1998} establishes an effective version
of the Urysohn Metrization theorem. The effective metrization theorem
applies to spaces that need not be computably separable, and, because
of this, it has never been stated before how the computable topology
that one is applying metrization to is related to the computable metric
given by the theorem: in \cite{Schroeder1998}, and in \cite{Weihrauch2013}
where Weihrauch revisits this theorem, the statement of the theorem
ends up being ``there is a computable metric which (classically)
generates the topology under scrutiny''. Thanks to the presented
article, we can strengthen this result: the computable topology one
started with is indeed the computable topology generated by the computable
metric. 

\bigskip

We will now discuss the contents of this paper, giving slightly more
details. 

\subsection{Type 1 computability }

The present article is written in the context of Type 1 computability.
This choice was mostly dictated by the immediate needs of the author
(see \cite{RauzyV}), but it is on the whole rather inconsequential.
In Section \ref{subsec:Remark-on-Type-2} and \ref{subsec:The-realizability-approach}
of this introduction, we discuss relation to other approaches. 

Type 1 computability is the approach to computable mathematics one
obtains by working only with finite descriptions. It was formalized
thanks to the use of \emph{numberings} in a pioneering article of
Malcev \cite{Maltsev1961,Malcev1971} on numbered algebras. A \emph{numbering}
of a set $X$ is a surjection $\nu$ defined on a subset of the natural
numbers and whose codomain is $X$. This is denoted by 
\[
\nu:\mathbb{N}\rightharpoonup X,
\]
where the symbol $\rightharpoonup$ indicates that $\nu$ does not
have to be total. A numbering can be interpreted as a ``semantic
function'', mapping symbols to their meaning. A natural number $n$
that is mapped to a point $x$ of $X$ via $\nu$ is called a $\nu$\emph{-name}
of $x$. 

One of the innovations of the present article is that we consider
also \emph{subnumberings}. A subnumbering $\nu$ of a set $X$ is
a numbering of a subset of $X$. The image of $\nu$ is then called
the set of $\nu$\emph{-computable points}. We denote by $X_{\nu}$
the set of $\nu$-computable points. A function $f:X\rightarrow Y$
between subnumbered sets $(X,\nu)$ and $(Y,\mu)$ is called $(\nu,\mu)$\emph{-computable}
when there exists a partial computable function $F:\mathbb{N}\rightharpoonup\mathbb{N}$
which maps any $\nu$-name of a point $x$ of $X$ to a $\mu$-name
of its image $f(x)$ in $Y$. 

A commonly accepted idea states that Type 1 computability has something
to say only about countable sets, because only countable sets admit
numberings. The use of subnumberings allows to disprove this idea.
Since subnumberings are present in mathematics as soon as some mathematical
objects have finite descriptions, they in fact arise very naturally
even in the study of uncountable sets. For instance, every mathematician
can easily ``compute'' the union of the intervals $]0,2[$ and $]1,3[$
in $\mathbb{R}$, and the fact that not all real numbers need to have
a finite description, i.e. that there is no numbering of the set of
real numbers, is not usually seen as a fundamental obstacle to finding
the result $]0,3[$. While this example is rather trivial, very similar
arguments are commonly accepted as showing that the use of numberings
should be restricted to settings where all points admit finite descriptions.

\subsection{Formal inclusion relations, computable topological spaces }

We now present the definitions of \emph{formal inclusion relations,
Type 1 computable topological spaces}, and \emph{effectively continuous
functions}. 

The following definition is exactly Spreen's definition of a \emph{strong
inclusion}, except that we also demand reflexivity. We thus use the
terms ``formal inclusion relation'' instead of ``reflexive strong
inclusion relation''. This is also related to the notion of formal
inclusion relation as used in the theory of domain representation
\cite{StoltenbergHansen2008}. 
\begin{defn}
[Spreen, \cite{Spr98}, Definition 2.3]\label{def:Formal inclusion relation }Let
$\mathfrak{B}$ be a subset of $P(X)$, and $\beta$ a numbering of
$\mathfrak{B}$. Let $\mathring{\subseteq}$ be a binary relation
on $\text{dom}(\beta)$. We say that $\mathring{\subseteq}$ is a
\emph{formal inclusion relation for $(\mathfrak{B},\beta)$} if the
following hold:
\begin{itemize}
\item The relation $\mathring{\subseteq}$ is reflexive and transitive (i.e.
$\mathring{\subseteq}$ is a preorder); 
\item $\forall n,m\in\text{dom}(\beta),\,n\mathring{\subseteq}m\implies\beta(n)\subseteq\beta(m)$.
\end{itemize}
\end{defn}

We now define Type 1 computable topological spaces\emph{. }See Definition
\ref{def:MAIN DEF} for more details. 
\begin{defn}
\label{def:MAIN DEF-1-1}A \emph{Type 1 computable topological space}
is a quintuple $(X,\nu,\mathcal{T},\tau,\mathring{\subseteq})$ where
$X$ is a set, $\mathcal{T}\subseteq\mathcal{P}(X)$ is a topology
on $X$, $\nu:\mathbb{N}\rightharpoonup X$ is a subnumbering of $X$,
$\tau:\mathbb{N}\rightharpoonup\mathcal{T}$ is a subnumbering of
$\mathcal{T}$, $\mathring{\subseteq}$ is a formal inclusion relation
for $\tau$, and such that: 
\begin{enumerate}
\item The image of $\tau$ generates the topology $\mathcal{T}$;
\item The empty set and $X$ both belong to the image of $\tau$;
\item The open sets in the image of $\tau$ are uniformly $\nu$-semi-decidable; 
\item The operations of taking computable unions and finite intersections
are computable, i.e.:
\begin{enumerate}
\item The function $\bigcup:\mathcal{T}^{\mathbb{N}}\rightarrow\mathcal{T}$
which maps a computable sequence of open sets $(A_{n})_{n\in\mathbb{N}}$
to its union is computable, and it can be computed by a function that
is expanding and increasing for $\mathring{\subseteq}$;
\item The function $\bigcap:\mathcal{T}\times\mathcal{T}\rightarrow\mathcal{T}$
which maps a pair of open sets to their intersection is computable,
and it can be computed by a function increasing for $\mathring{\subseteq}$. 
\end{enumerate}
\end{enumerate}
\end{defn}

Recall that a function $f:A\rightarrow A$ on a poset $(A,\preceq)$
is called expanding if $\forall x\in A,\,x\preceq f(x)$. 

The first condition is here to record the fact that the abstract topology
we are ``effectivizing'' is indeed $\mathcal{T}$. The third condition
was coined Malcev's condition by Spreen \cite{Spr98} in reference
to Malcev's work on numberings from \cite{Malcev1971}. 

The open sets in the image of $\tau$ are called the \emph{computably
open sets}. 

Finally, we have the definition of effectively continuous functions.
Note that we give a \emph{single} definition. There are different
definitions of ``effectively continuous functions'' that can be
found in the literature. They in fact correspond to effective continuity
with respect to different computable topologies. 
\begin{defn}
A function $f$ between two computable topological spaces $(X,\nu,\mathcal{T}_{1},\tau_{1},\mathring{\subseteq}_{1})$
and $(Y,\mu,\mathcal{T}_{2},\tau_{2},\mathring{\subseteq}_{2})$ is
called \emph{effectively continuous} if it is classically continuous,
if the function $f^{-1}:\mathcal{T}_{2}\rightarrow\mathcal{T}_{1}$
is $(\tau_{2},\tau_{1})$-computable, and if it can be computed by
a function increasing with respect to the formal inclusion relation:
there should exist a computable function $h:\text{dom}(\tau_{2})\rightarrow\text{dom}(\tau_{1})$
for which:
\[
\forall n\in\text{dom}(\tau_{2}),\,\tau_{1}(h(n))=f^{-1}(\tau_{2}(n)),
\]
\[
\forall n_{1},n_{2}\in\text{dom}(\tau_{2}),\,n_{1}\mathring{\subseteq}_{2}n_{2}\implies h(n_{1})\mathring{\subseteq}_{1}h(n_{2}).
\]
\end{defn}

An important remark is the fact that the use of a formal inclusion
relation gives a \emph{strictly more general definition }than what
we would obtain by omitting it. Indeed, in any setting where a formal
inclusion relation is not specified, it is as though the actual inclusion
relation $n\mathring{\subseteq}m\iff\tau(n)\subseteq\tau(m)$ was
chosen, by default. And thus the definitions written above provide
the extra flexibility of using a formal inclusion that is different
from the actual inclusion relation. 

\subsection{Lacombe and Nogina bases }

The obvious approach towards defining the computable topology generated
by a computable metric is to use a notion of computable basis, and
to say that the computable topology generated by the metric is the
topology whose basis are the open balls. 

When proceeding like this, one is immediately confronted to a problem:
classically equivalent notions of bases yield non-equivalent effective
definitions. Indeed, there are two common formulations of the classical
definition of a basis for a topology on a set $X$. A basis is a set
$\mathcal{B}$ of subsets of $X$ such that:
\begin{itemize}
\item Every element of $X$ belongs to an element of $\mathcal{B}$;
\item For any two elements $B_{1}$ and $B_{2}$ of $\mathcal{B}$, and
for any $x$ in $B_{1}\cap B_{2}$, there is an element $B_{3}$ in
$\mathcal{B}$ containing $x$ and such that $B_{3}$ is a subset
of $B_{1}\cap B_{2}$.
\item A subset $O$ of $X$ is called \emph{open} if for any $x$ in $O$
there is $B$ in $\mathcal{B}$ such that $x\in B$ and $B\subseteq O$. 
\end{itemize}
Or, a second approach: 
\begin{itemize}
\item The union of elements of $\mathcal{B}$ gives $X$: 
\[
\underset{B\in\mathcal{B}}{\bigcup}B=X;
\]
\item The intersection of two elements of the basis can be written as a
union of elements of this basis: for any $B_{1}$ and $B_{2}$ in
$\mathcal{B}$, there is a subset $\mathcal{C}$ of $\mathcal{B}$
such that 
\[
B_{1}\cap B_{2}=\underset{B\in\mathcal{C}}{\bigcup}B.
\]
\item A subset $O$ of $X$ is called \emph{open} if it can be written as
a union of basic sets. 
\end{itemize}
Both of these approaches are equivalent. Note that one must prove
that the sets we have called ``open'' do form a topology. 

In terms of computational contents, the above conditions are however
not equivalent. We sketch here the corresponding effective definitions.
The first one we associate to Nogina, since it is the approach followed
in \cite{Nogina1966}, the second one corresponds to the already discussed
Lacombe approach. 

Nogina approach: 
\begin{itemize}
\item There is a program which, given the name of a point in $X$, produces
the name of a basic open set that contains $x$;
\item There is a program which, given the names of two basic open sets $B_{1}$
and $B_{2}$ and the name of a point $x$ in their intersection, produces
the name of a basic open set $B_{3}$, containing $x$, and which
is a subset of $B_{1}\cap B_{2}$.
\item An effective open set is a semi-decidable set $O$ for which there
is a program that, on input a name for a point $x$ in $O$, produces
the name of a basic set $B$ with $x\in B$ and $B\subseteq O$. 
\end{itemize}
Lacombe approach: 
\begin{itemize}
\item The ambient set $X$ can be written as a computable union of basic
open sets; 
\item There is a program which, given the name of two basic open sets $B_{1}$
and $B_{2}$, produces the code of a computable sequence of basic
open sets whose union gives the intersection $B_{1}\cap B_{2}$.
\item An effective open set is a computable union of basic open sets. 
\end{itemize}
Moschovakis' theorem (which applies only to countable sets equipped
with onto numberings) then states:
\begin{thm}
[Moschovakis, \cite{Moschovakis1964}, Theorem 11] \label{thm:Moschovakis, INTRO Theorem}Let
$(X,\nu)$ be numbered set with a computable metric $d$, where limits
of computably convergent computable sequences can be computed, and
which admits a dense and computable sequence. 

Then, the set of open balls with computable radii, equipped with its
natural numbering, forms both a Lacombe and a Nogina basis, and the
computable topologies generated by these bases are identical.
\end{thm}

We include a proof of this theorem in Section \ref{subsec:Proof of Moschovakis-1}. 

\bigskip

In fact, neither of the two notions of computable basis introduced
above is satisfactory, and neither one of them will yield the desired
notion of \emph{the computable topology generated by a computable
metric}. 
\begin{itemize}
\item The conditions imposed on Nogina bases are as weak as possible, and
thus the notion of Nogina basis is very general, but Nogina bases
can only handle countable sets and actual numberings. Also, it was
noted by Bauer in \cite{Bauer2000} that Nogina bases are ``not as
well behaved'' as Lacombe bases\footnote{We discuss the relation between this article and the results of \cite{Bauer2000}
in Section \ref{subsec:The-realizability-approach} of this introduction.}.
\item Lacombe bases do yield computable topologies even when we consider
possibly uncountable subnumbered spaces. However, the conditions imposed
on a basis to be a Lacombe basis are too restrictive, and in particular
they do not apply to metric spaces unless we add effective separability
conditions. 
\item Finally, whereas it is easy to see that Lacombe bases are well behaved
as soon as the ambient space is computably separable, Moschovakis'
Theorem requires some additional hypotheses to identify the Nogina
and Lacombe approaches. While the condition of looking at a metric
space can be dropped, as was shown by Spreen \cite{Spr98}, the hypothesis
that ``limits of computably convergent sequences can be computed''
cannot be omitted. 
\end{itemize}
Note that in the literature, the notion of Lacombe basis is the most
studied one. The article \cite{Weihrauch2009ElementaryCT} by Grubba
and Weihrauch was particularly influential, it is based on the use
of c.e. Lacombe bases. Other articles where a Lacombe approach is
followed are: \cite{Morozov2004,Wei08,Selivanov2008,Freer2010,KOROVINA2016,Hoyrup2016,Schroeder2021,Lupini2023,Melnikov2023}.
Nogina \cite{Nogina1966,Nogina_1969} and Spreen \cite{Spr98,SPREEN_2016}
seem to be the only authors to have used a Nogina approach. 

\subsection{Spreen bases }

The main result of this paper is the definition of a new type of computable
basis which lies in between Nogina bases and Lacombe bases and which
suffers from none of the issues mentioned above. The idea of using
formal inclusion relations to define notions of computable bases that
lies in between Nogina and Lacombe bases is due to Dieter Spreen \cite{Spr98}.
We thus define what we call \emph{Spreen bases}. 

Because the general definition of a Spreen basis and of the computable
topology it generates are slightly technical (see Section \ref{part: Non Surjective Numberings}),
we will in this introduction only illustrate them in the case of metric
spaces.

Suppose that we want to prove that open balls in a metric space form
a (classical) basis. We consider two balls $B(x,r_{1})$ and $B(y,r_{2})$
and a point $z$ in their intersection, and define $r_{3}=\text{min}(r_{1}-d(x,z),r_{2}-d(y,z))$.
As a consequence of the triangular inequality, one can see that the
ball centered at $z$ and with radius $r_{3}$ must be a subset of
the intersection $B(x,r_{1})\cap B(y,r_{2})$. 

The reader will surely be convinced that this is essentially the only
proof available to us to show that the open balls form the basis of
a topology.

Studying computable topology, we will thus rely on the map $\Theta$
given by 
\[
\Theta:(x,r_{1},y,r_{2},z)\mapsto\text{min}(r_{1}-d(x,z),r_{2}-d(y,z))
\]
 to try and prove that open balls form an effective basis. 

The situation is then the following: 
\begin{itemize}
\item Thanks to the map $\Theta$, we can show that the open balls in a
metric space form a Nogina basis, but ``being a Nogina basis'' is
too weak a condition for our purpose.
\item On the other hand, the map $\Theta$ is not sufficient to be able
to express the intersection of two balls as a computable union of
balls, if we do not add hypotheses which state somewhat that ``points
can be constructed''. 
\end{itemize}
The crucial step to obtaining a new type of computable basis is to
remark that the map $\Theta$ satisfies some additional conditions,
which are not demanded in the plain Nogina setting. Namely:
\begin{itemize}
\item $\Theta$ is increasing in $r_{1}$ and $r_{2}$, and it does not
depend on the given names of $x$, $y$ and $z$. We associate the
symbol $\boxed{\mathring{\sim}}$ to this condition, which will be
called ``respect of the formal inclusion condition''. 
\item The function $\Theta$ does not vanish inside $B(x,r_{1})\cap B(y,r_{2})$:
for any $z$ in $B(x,r_{1})\cap B(y,r_{2})$, there is some $s>0$
such that for any point $p$ in $B(z,s)$, $\Theta(x,r_{1},y,r_{2},p)>s$.
We associate the symbol $\circledcirc$ to this ``non-vanishing condition''. 
\end{itemize}
Those two conditions are precisely those that we will demand of Spreen
bases. However, we will first have to be able to express them in the
general context of topological spaces, without relying on the metric. 

It becomes clear at this point that we will have to use an intensional
notion to generalize these conditions. Indeed, the function $\Theta$
is not a function that takes balls, seen as sets, as input, but balls
\emph{parametrized by pairs center-radius}. And given different pairs
center-radius that in fact define the same ball, i.e. the same set
of points, the function $\Theta$ may still produce different results. 

Thus the condition $\boxed{\mathring{\sim}}$, in the general context
of topological spaces, will say that the function that computes intersections
is \emph{formally increasing}, increasing with respect to a formal
inclusion relation. We cannot expect it to be actually increasing
with respect to set inclusion, since it is not a function that depends
only on its input seen as a set, it also depends on \emph{how} a certain
set is given as input. 

The condition $\circledcirc$ in the general context of topological
spaces will say that when computing the intersection of two basic
sets $B_{1}$ and $B_{2}$ in the neighborhood of a point $z$ of
$B_{1}\cap B_{2}$, the computed basic sets will not be needlessly
small, in particular they will always contain a neighborhood of $z$.
However, this condition is again stated with some intentionality:
indeed, the non-vanishing condition depends on the given names of
$B_{1}$ and $B_{2}$, and not only on $B_{1}$ and $B_{2}$ given
as abstract sets.

As a simple example, consider the closed interval $[0,1]\subseteq\mathbb{R}$
as a metric space. Suppose that we compute the intersections of the
balls $B(1/2,1/2+\epsilon)$ and $B(1,1)$ at $1$ for different values
of $\epsilon>0$. In this case, the radius computed around $1$ by
the function $\Theta$ is exactly $\epsilon$. And thus this radius
can become arbitrarily small, even though in fact the open balls that
we are considering are always the same sets: the considered intersection
is always $[0,1]\cap]0,1]$, and any radius below one would be an
acceptable answer. This explains again why we cannot expect to have
an extensional non-vanishing condition $\circledcirc$. 

Let us translate the two conditions $\boxed{\mathring{\sim}}$ and
$\circledcirc$ written above in the language of formal inclusion
relations. 

Fix a subnumbered set $(X,\nu)$. Suppose that we have a numbered
basis $(\mathfrak{B},\beta)$ equipped with a formal inclusion relation
$\mathring{\subseteq}$, and a map $\Theta$ which takes as input
the $\beta$-names of two basic open sets $B_{1}$ and $B_{2}$ and
the $\nu$-name of a point $x$ in their intersection, and produces
the $\beta$-name of a ball that contains $x$ and that is contained
in $B_{1}\cap B_{2}$. 

We will then ask that: 
\begin{itemize}
\item $\Theta$ should be increasing for $\mathring{\subseteq}$: if $b_{1}\mathring{\subseteq}\hat{b}_{1}$,
$b_{2}\mathring{\subseteq}\hat{b}_{2}$ and if $n$ and $m$ are two
different names of a point $x$ in $\beta(b_{1})\cap\beta(b_{2})$,
then we should have 
\[
\Theta(b_{1},b_{2},n)\mathring{\subseteq}\Theta(\hat{b}_{1},\hat{b}_{2},m).
\]
This is the condition $\boxed{\mathring{\sim}}$ as written above
for metric spaces formulated in terms of formal inclusion relations. 
\item $\Theta$ does not vanish: if $b_{1}$ and $b_{2}$ are the names
of two basic open sets, and if $x$ is a point in $\beta(b_{1})\cap\beta(b_{2})$,
there should be some basic set $B_{3}=\beta(b_{3})$ that contains
$x$ and such that for each computable point $z=\nu(n)$ in $B_{3}$
we have:
\[
b_{3}\mathring{\subseteq}\Theta(b_{1},b_{2},n).
\]
This is exactly the reformulation of the non-vanishing condition $\circledcirc$
written above for metric spaces to the general context of numbered
bases with formal inclusion relations. 
\end{itemize}
The definition of a Spreen basis, which appears in Section \ref{part: Non Surjective Numberings},
is then simply the Nogina definition to which we have added the two
conditions above: respect of the formal inclusion relation and non
vanishing condition.

\subsection{Spreen bases and Lacombe bases}

We then prove that Spreen bases are well behaved and generalize properly
Lacombe bases. 
\begin{lem}
\label{Lem: Lacombe basis is Spreen basis-1-1}Let $(X,\nu)$ be a
subnumbered set. As soon as each non-empty basic open set of $\mathfrak{B}$
contains a $\nu$-computable point, any Lacombe basis \textup{$(\mathfrak{B},\beta)$}
on $(X,\nu)$ is a Spreen basis. Furthermore, there is a process that
transforms the code of a Lacombe set into a name of this same set
seen as a Spreen open set.
\end{lem}

The proof of this lemma requires the introduction of a formal inclusion
relation for any Lacombe basis. The formal inclusion relation naturally
associated to a Lacombe basis is equality: $a\mathring{\subseteq}b\iff a=b$.
It is the finest formal inclusion relation, and thus the most restrictive
one. 

As soon as the considered set has a dense and computable sequence,
Spreen and Lacombe topologies in fact agree. The following results
appear as Lemma \ref{lem: dense seq spreen =00003D> lacombe} and
Theorem \ref{thm:Moschovakis-theorem-on Spreen bases} in the text. 
\begin{lem}
Let $(\mathfrak{B},\beta,\mathring{\subseteq})$ be a Spreen basis
on $(X,\nu)$. If there exists a $\nu$-computable sequence which
is dense for the topology generated by $\mathfrak{B}$, then $(\mathfrak{B},\beta)$
is also a Lacombe basis. 
\end{lem}

\begin{thm}
[Moschovakis' theorem on Lacombe sets for Spreen bases]\label{thm:Moschovakis THM Spreen bases}Let
$(\mathfrak{B},\beta,\mathring{\subseteq})$ be a Spreen basis on
$(X,\nu)$. If there exists a $\nu$-computable sequence which is
dense for the topology generated by $\mathfrak{B}$, then the computable
topologies generated by $(\mathfrak{B},\beta,\mathring{\subseteq})$,
respectively seen as a Spreen and as a Lacombe basis, are computably
equivalent (i.e. they are identical). 
\end{thm}

\subsection{\label{subsec:Intro Computable-metric-spaces}Definition of the metric
topology }

Finally, we define the computable topology coming from a computable
metric. 
\begin{lem}
Suppose that $(X,\nu,d)$ is a subnumbered space with a computable
metric. Suppose that the image of $\nu$ is dense with respect to
the classical metric topology. Then, the set of open balls of $X$
centered at computable points equipped with their natural numbering
forms a Spreen basis. 
\end{lem}

The formal inclusion relation implicit in the lemma above is the one
that comes from the relation on pairs in $X\times\mathbb{R}$: $(y,r_{1})\mathring{\subseteq}(z,r_{2})\iff d(z,y)+r_{1}\le r_{2}$.
Thus we can introduce the computable topology generated by this basis.
This topology has a nice direct description, which we now give. It
involves the numbering naturally associated to left computable reals,
which we denote $c_{\nearrow}$. Denote $\mathbb{R}^{+}$ the set
of positive real numbers. 
\begin{defn}
[Direct definition of the metric topology, Definition \ref{def:Direct-definition-of-Metric-TOP}]
\label{def:Direct-definition-of-Metric-TOP-1-1}Suppose that $(X,\nu,d)$
is a subnumbered space with a computable metric and with dense computable
points. We define a computable topology $(\mathcal{T},\tau,\mathring{\subseteq})$
on $(X,\nu)$ thanks to the following. $\mathcal{T}$ is the metric
topology on $X$, and $\tau$ is a subnumbering defined as follows.
The $\tau$-name of an open set $O$ is an encoded pair $(n,m)$,
where $n$ gives the code of $O\cap X_{\nu}$ as a semi-decidable
set, and where $m$ encodes a $(\nu,c_{\nearrow})$-computable function
$F:O\cap X_{\nu}\rightarrow\mathbb{R}^{+}$, which satisfies the following:
\[
\forall x\in O\cap X_{\nu},\,B(x,F(x))\subseteq O;
\]
\[
\circledcirc\,\,\forall x\in O,\,\exists r>0,\,\forall y\in B(x,r)\cap X_{\nu},\,F(y)\ge r.
\]

Finally, we define a formal inclusion relation on names of open sets
by saying that that an open set $O_{1}$ described by a pair $(O_{1}\cap X_{\nu},F_{1})$
is formally included in an open set $O_{2}$ described by a pair $(O_{2}\cap X_{\nu},F_{2})$
when 
\[
\forall n\in\text{dom}(\nu),\,\nu(n)\in O_{1}\implies\nu(n)\in O_{2}\,\&\,F_{1}(n)\le F_{2}(n).
\]
\end{defn}

\subsection{\label{subsec:Remark-on-constructive}Effectively given domains }

The main example of computable topological spaces that are considered
in the present paper are spaces with a computable metric. Another
important example that has been considered in the literature is that
of \emph{effectively given domains. }

On a partially ordered set $(Q,\sqsubseteq)$, the way below relation
$\ll$ is defined by: $y\ll z$ if every directed set whose supremum
is above $z$ already contains a point above $y$. A domain basis
for $Q$ is a set $Z$ such that for every $y\in Q$ the set $\{z\in Z,\,z\ll y\}$
is directed and has supremum $y$. A domain is continuous if it has
a basis. When $Z$ is a domain basis, sets of the form $B_{z}=\{y\in Q,y\ll z\}$
will form a topological basis for the Scott topology on $Q$.

Weihrauch's approach to ``effectivizing'' continuous domains \cite{Weihrauch1983}
is to rely on a domain basis equipped with a total numbering for which
the way below relation is semi-decidable. These effectively given
domains indeed provide another example of Type 1 computable topological
spaces \cite{Spr98}. 

We want here to remark the following. 

The approach that relies on a totally numbered domain basis to define
a notion of \emph{computable domain} only allows to define effectively
given domains that have a c.e. domain basis, and that are in turn
computably second countable. In the present paper, we take care in
defining the computable topology coming from a computable metric without
relying on a computable and dense sequence. A similar endeavour could
be undertaken for continuous domains. 
\begin{problem}
Is it possible to find a more intrinsic approach to computable domains,
that does not rely on a c.e. basis to define a notion of computability,
but which becomes equivalent to Weihrauch's approach for domains that
do have a c.e. basis?
\end{problem}

\subsection{Remark on Type 2 computability\label{subsec:Remark-on-Type-2}}

The study of \emph{computable topological spaces} \cite{Weihrauch2009ElementaryCT}
and of \emph{computable metric spaces} \cite{Brattka2003} is usually
set in the context of computable analysis, as developed by Weihrauch
and his students. There, the main object of study is \emph{represented
sets}: sets equipped with ways of describing elements by sequences
of natural numbers, i.e. elements of $\mathbb{N}^{\mathbb{N}}$. A
\emph{representation} of a set $X$ is a partial surjection $\rho:\mathbb{N}^{\mathbb{N}}\rightharpoonup X$.
And thus the study of representations is strictly more general than
the study of numberings, since a numbering can be seen for instance
as a representation whose domain is a subset of the set of constant
sequences. 

And we would indeed obtain a strictly more general definition of computable
topological space if we defined a Type 2 computable topological space
to be a set equipped with a representation, together with a classical
topology on this set, also equipped with a representation, and such
that the operations of taking countable unions and finite intersections
are computable. However, this is not the approach that is usually
followed in computable analysis. 

Indeed, following ideas that were obtained independently by Matthias
Schröder \cite{Schroeder2003}, Paul Taylor \cite{Taylor2011} and
Martin Escardo \cite{Escardo2004}, computable topologies are studied
in Type 2 computability via the \emph{Sierpi\'{n}ski representation
of open sets. }See \cite{Pauly2016} for a very good survey. Whenever
$(X,\rho)$ is a represented set, $X$ can be equipped with the final
topology of the representation $\rho$. Then, one can prove that every
open set of $X$ is Type 2 semi-decidable modulo some oracle (using
the fact that this holds on Baire space). This allows to define the
Sierpi\'{n}ski representation of open sets of $X$: the name of an
open set $O$ is the concatenation of the code of a Type 2 Turing
machine that stops exactly on the elements of $O$ together with the
infinite sequence which is the oracle that this function requires
to operate. 

One can see that this representation is the ``Type 2 equivalent''
of the numbering of semi-decidable sets associated to a numbered set.
This ``equivalence'' can in fact be formalized in terms of realizability
theory\emph{, }it is in each case\emph{ }the\emph{ intrinsic topology
}\cite{Bauer2000} that comes from the representation/numbering. 

One could thus say that in the Type 2 theory of effectivity, one always
\emph{restricts} their attention to the intrinsic topology of the
considered sets, the topology of semi-decidable sets. 

However, a crucial fact in the Type 2 theory of effectivity is that
this ``restriction'' is actually not one: Weihrauch and Kreitz \cite{Kreitz1985}
have built admissible representations for all second countable spaces,
and Schröder extended their results in \cite{Schroeder2003} to $\text{T}_{0}$
quotients of countably based spaces, the so called qc$\text{b}_{0}$-spaces.
(The $\text{T}_{0}$ hypothesis can be removed thanks to multi-representations.)
It follows from these that \emph{the set of topologies that are final
topologies of represented spaces is sufficiently rich to develop computable
analysis. }And in turn, the systematic use of the Sierpi\'{n}ski representation
provides a very beautiful and robust theory, it also greatly simplifies
proofs.\emph{ }

In the study of Type 1 computability, the above fact fails. The intrinsic
topology of a numbered set, the topology of semi-decidable sets, is
called the \emph{Ershov topology }(see for instance \cite{Spr98})\emph{.}
And it is not possible to study Type 1 computable topologies by restricting
our attention to Ershov topologies, because Type 1 Ershov topologies
cannot be ``calibrated'' to match any desired abstract topology.
The most famous example of this phenomenon is Friedberg's theorem
of existence of a semi-decidable set of computable real numbers which
is not open in the usual topology of $\mathbb{R}$ \cite{Friedberg1958}.
This example was explained by Hoyrup and Rojas in \cite{Hoyrup2016}
in terms of Kolmogorov complexity. The fact that the use of intrinsic
topologies is not adapted to Type 1 computability was already remarked
in \cite{Bauer2012}. 

Because of this, the definition of \emph{a Type 1 computable topological
space} that we present in this article does not directly translate
into a relevant Type 2 notion. 

However, the definition that we give of the computable topology generated
by a computable metric can be relevant in Type 2 computability. 

Adapting this definition, Definition \ref{def:Direct-definition-of-Metric-TOP-1-1},
to the context of represented spaces yields the following. Let $(X,\rho)$
be a represented space, and let $d:X\times X\rightarrow\mathbb{R}$
be a $\rho$-computable metric. Let $\mathbb{R}_{\nearrow}$ be the
set of real numbers equipped with the left representation: a real
number is given by the list of rationals below it. Let $\mathcal{C}_{p}(X,\mathbb{R}_{\nearrow})$
be the represented space of partial continuous functions from $X$
to $\mathbb{R}_{\nearrow}$. We say that \emph{the metric $d$ induces
the computable topology }of $(X,\rho)$ if the following multi-function
is computable:
\begin{align*}
\mathcal{O}(X) & \rightrightarrows\mathcal{C}_{p}(X,\mathbb{R}_{\nearrow})\\
O & \mapsto\{f:O\rightarrow\mathbb{R}_{\nearrow},\,\forall x\in O,\,B(x,f(x))\subseteq O\}.
\end{align*}
Note that the non-vanishing condition that we have to demand in the
case of Type 1 computability is automatically satisfied in Type 2.
One can check that the above condition is always satisfied when applying
the Schröder metrization theorem to a represented space, thus the
computable metric given by the theorem does generate the computable
topology one started with. 

\subsection{Realizability and synthetic topology\label{subsec:The-realizability-approach}}

Computable topologies have been studied in the context of realizability
by Andrej Bauer in his dissertation \cite{Bauer2000}, and more precisely,
topological spaces equipped with countable bases. The two main approaches
studied in \cite{Bauer2000} are called the \emph{point-wise topology},
which corresponds to the Nogina approach, and the \emph{point-less
topology}, which corresponds to the Lacombe approach. The initial
analysis of these types of basis presented in \cite{Bauer2000} is
very similar to the one we present here (Section \ref{subsec:First-two-notions-of-bases}
is mostly already contained in \cite{Bauer2000}). 

The main differences between our analysis and the one presented in
\cite{Bauer2000} are as follow.

The different notions of bases studied in \cite{Bauer2000} are notions
of \emph{computably enumerable bases}, the focus is thus on \emph{computably
second countable spaces}. One of the conclusions of \cite{Bauer2000}
is that the Lacombe approach is better behaved that the Nogina approach.
We claim that the Lacombe approach is not well behaved, but the examples
that we provide are not ``computably second countable'', thus there
is no contradiction between our results and those of \cite{Bauer2000}.

Metric and topological spaces have also been studied in the context
of synthetic topology by Bauer and Lešnik in \cite{Bauer2012}, and
more recently by Bauer in \cite{Bauer2023}. There, the metric topology
is defined as the topology of \emph{overt unions of open balls}. For
subsets of $\mathbb{N}$, the over sets and c.e. subsets agree. Thus
while we have presented the Lacombe approach by saying that ``computably
open sets are defined to be the c.e. unions of open sets'', we could
equivalently have used overt unions of basic open sets. Differences
would appear if we were to consider bases indexed by sets other than
sets of natural numbers. The set of \emph{metric open sets }considered\emph{
}in\emph{ }\cite{Bauer2012} is not stable by finite intersections
in general, unless we consider overt metric spaces, and thus it is
not a computable topology in our sense. 

\subsection{Contents }

This article is organized as follows. 

In Section \ref{part: Preliminaries}, we introduce basic definitions
about numberings and multi-numberings. 

In Section \ref{part:Two-notions-of-Bases}, we introduce the general
notion of Type 1 computable topological space, and the first two notions
of effective bases: Lacombe and Nogina bases. 

In Section \ref{sec:Problems-with-Lacombe-bases}, we compare Lacombe
and Nogina bases in different contexts. We give an example of a non-computably
separable computable metric space where open balls do not form a Lacombe
basis. We provide a new proof of Moschovakis' Theorem which shows
that on recursive Polish spaces, the Lacombe and Nogina approaches
coincide. 

In Section \ref{part: Non Surjective Numberings}, we introduce Spreen
bases. This is the longest section of this article, we give more details
about its contents. 

In Section \ref{subsec:Spreen Bases}, we introduce Spreen bases,
their associated topology, and prove the theorem which states that
one indeed defines a topology in this way. 

In Section \ref{subsec:New notion of equivalence of basis }, we describe
the notion of ``equivalence of bases'' appropriate to the study
of Spreen bases. 

In Section \ref{subsec:Indep Numbering basis Spreen}, we explain
that when using a formal inclusion relation, numberings are not considered
``up to equivalence'' (i.e. bi-translatability), but ``up to equivalence
that preserves the formal inclusion relation''. 

In Section \ref{subsec:Effective-continuity-and preimage of basic open set },
we discuss the effective equivalent of ``a function is continuous
if the preimage of any basic open set in its codomain is an open set
of its domain''. 

In Section \ref{subsec:Lacombe-bases-define-Spreen-Bases}, we show
that Spreen bases generalize Lacombe bases. 

In Section \ref{subsec:Effective-separability-Moschovakis THM}, we
establish a Moschovakis-type theorem, which states that as soon as
the considered space has a computable and dense sequence, the notions
of Lacombe bases and Spreen bases agree. 

In Section \ref{subsec:Recursive-metric-spaces II }, we discuss metric
spaces, and describe the Spreen topology induced by a computable metric. 

In Section \ref{subsec:Spreen-bases-and-Nogina bases}, we translate
Moschovakis' Theorem on countable sets in terms of identification
of Spreen and Nogina topologies. 

\subsection*{Acknowledgements }

I thank Andrej Bauer for taking the time to answer my questions, Vasco
Brattka for very helpful discussions, and Matthieu Hoyrup for comments
on a previous version of this paper.

\section{\label{part: Preliminaries}Preliminaries on numberings}

\subsection{Subnumberings, multi-subnumberings }

A \emph{numbering} of a set $X$ is a partial surjection from $\mathbb{N}$
to $X$. A \emph{subnumbering} of $X$ is a numbering of a subset
of $X$. We denote this by $\nu:\mathbb{N}\rightharpoonup X$. Denote
by $\text{dom}(\nu)$ the domain of $\nu$. 

Denote by $\mathcal{N}_{X}$ the set of subnumberings on $X$. 

The points in the image of $\nu$ are called the $\nu$\emph{-computable
points}. The set of $\nu$-computable points is denoted $X_{\nu}$.
If $\nu(n)=x$, then $n$ is a $\nu$\emph{-name} of $x$. 

We will sometimes refer ourselves to multi-subnumberings\footnote{I do not know who introduced multi-numberings first, but they are
well known objects. They are used for instance by Spreen in \cite{Spreen_2010}.
Multi-representations were first systematically investigated by Matthias
Schröder in his dissertation \cite{Schroeder2003}. See also \cite{Weihrauch2008}. }. 

A \emph{multi-subnumbering} of a set $X$ is a multi-function from
$\mathbb{N}$ to $X$. We denote this by $\nu:\mathbb{N}\rightrightarrows X$. 

Extensionally, this is exactly saying that $\nu$ is a total function
from $\mathbb{N}$ to the power set of $X$: $\nu:\mathbb{N}\rightarrow\mathcal{P}(X)$.
However, intensionally, those definitions differ. For instance, the
preimage of a subset $A\subseteq X$ by $\nu$ is not defined as $\{n\in\mathbb{N},\,\nu(n)=A\}$,
but as $\nu^{-1}(A)=\{n\in\mathbb{N},\,\nu(n)\cap A\neq\emptyset\}$
(i.e. we use the \emph{upper pre-image} and not the \emph{lower preimage}
of multi-subnumberings \cite{Berge1997}.).

If $\nu$ is a multi-subnumbering, we denote by $\text{dom}(\nu)$
the set $\{n\in\mathbb{N},\,\nu(n)\ne\emptyset\}$. 

The set of multi-subnumberings on $X$ is denoted $\mathcal{MN}_{X}$.
In case $x\in\nu(n)$, we say that $n$ is a $\nu$\emph{-name} of
$x$. The number $n$ may not characterize $x$, and thus it is in
general a partial description. 
\begin{defn}
A function $f:X\rightarrow Y$ between sets equipped with multi-subnumberings
$\nu$ and $\mu$ is called\emph{ $(\nu,\mu)$-computable} if there
exists a partial computable function $F:\mathbb{N}\rightharpoonup\mathbb{N}$
defined at least on $\text{dom}(\nu)$ such that $\forall n\in\text{dom}(\nu),\,f(\nu(n))\subseteq\mu(F(n))$. 
\end{defn}

An equivalent formulation of this definition is: for any $x$ in $X$
and $n$ in $\mathbb{N}$, if $n$ is a $\nu$-name of $x$, then
$F(n)$ is a $\mu$-name of $f(x)$. In this case, $F$ is called
a $(\nu,\mu)$\emph{-realizer} for $f$, and a function is called
$(\nu,\mu)$-computable exactly when it has a computable $(\nu,\mu)$-realizer. 

A subnumbering $\nu$ is identified with a multi-subnumbering $\hat{\nu}$
by 
\[
\forall n\in\text{dom}(\nu),\,\hat{\nu}(n)=\{\nu(n)\},
\]
\[
\forall n\notin\text{dom}(\nu),\,\hat{\nu}(n)=\emptyset.
\]

Note that in this case we have $\text{dom}(\nu)=\text{dom}(\hat{\nu})$,
so that our definitions are coherent. 

If $(X,\nu)$ is a multi-subnumbered set, the union of the image of
$\nu$, i.e. the set of points of $X$ that have names for $\nu$,
is called \emph{the set of $\nu$-computable points}, and denoted
$X_{\nu}$. 

A \emph{multi-function} between a set $X$ and a set $Y$ is a function
$f:X\rightarrow\mathcal{P}(Y)$. However, we denote a multi-function
from $X$ to $Y$ by $f:X\rightrightarrows Y$, and not by $f:X\rightarrow\mathcal{P}(Y)$. 

It seems that Weihrauch was the first to define computability of multi-functions
in \cite{Weihrauch1987}. 
\begin{defn}
A multi-function $f$ between sets $(X,\nu)$ and $(Y,\mu)$ equipped
with multi-subnumberings is called $(\nu,\mu)$\emph{-computable}
if and only if there exists a partial computable function $F:\mathbb{N}\rightharpoonup\mathbb{N}$,
whose domain contains at least the set $\text{dom}(\nu)$, and such
that:
\[
\forall n\in\text{dom}(\nu),\,\forall x\in\nu(n),\,\mu(F(n))\cap f(x)\ne\emptyset.
\]
\end{defn}

In words: $f$ is $(\nu,\mu)$-computable if, given the name of a
point $x$ for $\nu$, it is possible to compute a $\mu$-name of
\emph{some} point in its image $f(x)$. Note that which point is computed
might depend of the given name for $x$. 

In the definition above, we again say that $F$ is a\emph{ computable
realizer of $f$ for $\nu$ and $\mu$. }

Throughout, we fix a standard enumeration of all 1-ary partial computable
function, denoted $(\varphi_{0},\varphi_{1},\varphi_{2},...)$. We
write $\varphi_{n}(k)\downarrow$ if $k\in\text{dom}(\varphi_{n})$
and $\varphi_{n}(k)\uparrow$ if $k\notin\text{dom}(\varphi_{n})$. 

For $\nu$ and $\mu$ multi-subnumberings, we define a multi-numbering
of the set of $(\nu,\mu)$-computable functions, which we denote $\mu^{\nu}$,
by the following: 
\begin{align*}
\mu^{\nu} & :\begin{cases}
\mathbb{N}\rightrightarrows & Y^{X}\\
n\mapsto & \{f\in Y^{X},\,\varphi_{n}\text{ computes \ensuremath{f} for \ensuremath{\nu} and \ensuremath{\mu}}\}
\end{cases}
\end{align*}

The multi-numbering $\mu^{\nu}$ becomes a numbering when $\nu$ is
a numbering and $\mu$ is a subnumbering, since in this case a function
$\varphi_{n}$ can compute for at most one function $X\rightarrow Y$. 

When $\text{id}_{\mathbb{N}}$ is the identity of $\mathbb{N}$, $\nu^{\text{id}_{\mathbb{N}}}$
is the multi-numbering associated to $\nu$-computable sequences,
it is a numbering when $\nu$ is a subnumbering.  

Let $X$ be a set, and $\nu$ and $\mu$ be two multi-subnumberings
on $X$. We define a relation, denoted $\nu\le\mu$, by 
\[
\nu\le\mu\iff\text{id}_{X}:X\rightarrow X\text{ is \ensuremath{(\nu,\mu)}-computable}.
\]
In this case we say that $\nu$ \emph{translates} to $\mu$. One interprets
$\nu\le\mu$ by: the name of a point for $\nu$ contains more information
than a name of that same point for $\mu$. 

Denote $\nu\equiv\mu$ if $\nu\le\mu$ and $\mu\le\nu$. If $\nu\equiv\mu$,
we say that $\nu$ and $\mu$ are \emph{equivalent}. It follows directly
from the properties of $\le$ that $\equiv$ is an equivalence relation.
The equivalence classes for $\equiv$ are called the \emph{multi-subnumbering
types}. 

Results that interest us are almost always independent of a choice
of a multi-subnumbering, up to $\equiv$. An important exception arises
in Section \ref{subsec:Indep Numbering basis Spreen}. The set of
equivalence classes of multi-subnumberings is a lattice for the order
induced by $\le$. 

On a set $X$, define the \emph{trivial multi-numbering on $X$, }denoted\emph{
}$t_{X}$ defined by $t_{X}(n)=X$ for all $n$. The following is
immediate: 
\begin{prop}
A multi-numbering of $X$ is equivalent to $t_{X}$ if and only if
it has the set $X$ in its image. 
\end{prop}

The following vocabulary was introduced by Malcev \cite{Malcev1971}. 
\begin{defn}
A subnumbering $\nu:\mathbb{N}\rightharpoonup X$ is \emph{positive}
if equality is semi-decidable for $\nu$ names, i.e. if there is an
algorithm that on input two $\nu$-names, stops if and only if they
define the same element of $X$. It is \emph{negative} if equality
is co-semi-decidable. It is \emph{decidable} if equality is a decidable
relation of $\text{dom}(\nu)\times\text{dom}(\nu)$. 
\end{defn}

\subsection{Pairings and product multi-subnumberings }

Throughout, we denote by $(n,m)\mapsto\langle n,m\rangle$ a pairing
function, such as Cantor's pairing function. We extend $\langle\rangle$
recursively to arbitrarily many arguments: $\langle n,m,p\rangle=\langle n,\langle m,p\rangle\rangle$,
and so on. 

If $\nu$ is a multi-subnumbering of $X$ and $\mu$ a multi-subnumbering
of $Y$, denote by $\nu\times\mu$ the multi-subnumbering of $X\times Y$
defined by: $\langle n,m\rangle$ is a $\nu\times\mu$-name of $(x,y)\in X\times Y$
if and only if $n$ is a $\nu$-name of $x$ and $m$ a $\mu$-name
of $y$. This is the \emph{product multi-subnumbering}. We thus have:
\[
\forall\langle n,m\rangle\in\mathbb{N},\,(\nu\times\mu)(\langle n,m\rangle)=\nu(n)\times\mu(m).
\]

\subsection{Decidable sets }

A subset $A$ of $(X,\nu)$ is $\nu$\emph{-decidable} if there is
an algorithm that, on input the $\nu$-name of a point $x$ of $X$,
decides whether or not $x$ belongs to $A$. This induces a multi-subnumbering
of subsets of $X$ denoted $\nu_{D}$. 

If $c_{\{0,1\}}$ designates the natural decidable numbering of $\{0,1\}$,
one has $\nu_{D}\equiv c_{\{0,1\}}^{\nu}$. 
\begin{rem}
It is not true that a subset $A$ of $X$ is decidable if and only
if there exists a computable subset $R$ of $\mathbb{N}$ such that
$\nu^{-1}(A)=\text{dom}(\nu)\cap R$. Consider for instance c.e. computably
inseparable sets $U$ and $V$, with $U,V\subseteq\mathbb{N}$. Then
$U$ is a $\text{id}_{\mathbb{N}}$-decidable subset of $U\cup V$,
but the above fails. 
\end{rem}

\subsection{\label{subsec:Semi-decidable-sets-OPERATOR}Semi-decidable sets }

If $\nu$ is a subnumbering of $X$, a subset $A$ of $X$ is $\nu$-semi-decidable
if there is a procedure that, on input a $\nu$-name of a point in
$X$, will halt if and only if this point belongs to $A$. 
\begin{defn}
Let $\nu$ be a subnumbering on $X$. A subset $A$ of the set $X_{\nu}$
of $\nu$-computable points of $X$ is \emph{$\nu$-semi-decidable}
if and only if there is a computable function $\varphi_{n}$ such
that 
\[
\forall i\in\text{dom}(\nu),\,\varphi_{n}(i)\downarrow\iff\nu(i)\in A.
\]

We define a subnumbering $\nu_{SD}\in\mathcal{N}_{\mathcal{P}(X)}$
by putting $\nu_{SD}(n)=A$ for each $n$ and $A$ as above. 
\end{defn}

Note that we can either consider that $\nu_{SD}$ is a subnumbering
of $\mathcal{P}(X)$, or a subnumbering of $\mathcal{P}(X_{\nu})$.
(The $\nu_{SD}$-name encodes a subset of $X_{\nu}$, but of course
any subset of $X_{\nu}$ is a subset of $X$.)

If $\nu$ is a multi-subnumbering of $X$, the definition above remains
almost unchanged.
\begin{defn}
Let $\nu$ be a multi-subnumbering on $X$. A subset $A$ of the set
$X_{\nu}$ is $\nu$-semi-decidable if and only if there is a computable
function $\varphi_{n}$ such that 
\[
\forall i\in\text{dom}(\nu),\,\varphi_{n}(i)\downarrow\implies\nu(i)\subseteq A;
\]
\[
\forall i\in\text{dom}(\nu),\,\varphi_{n}(i)\uparrow\implies\nu(i)\cap A=\emptyset.
\]

We define a subnumbering $\nu_{SD}\in\mathcal{N}_{\mathcal{P}(X)}$
by putting $\nu_{SD}(n)=A$ for each $n$ and $A$ as above. 
\end{defn}

Remark that $\nu_{SD}$ remains a subnumbering even when $\nu$ is
a multi-subnumbering. 

If $\nu$ is a multi-subnumbering and $A$ is $\nu$-semi-decidable,
then for all $n$ in $\text{dom}(\nu)$, either $\nu(n)\subseteq A$,
or $\nu(n)\cap A=\emptyset$. This condition is obvious when $\nu$
is a subnumbering, since, in this case, $\nu(n)$ is a singleton,
and so $\nu(n)\subseteq A$ or $\nu(n)\cap A=\emptyset$ holds for
any $A$. 

When $\nu$ is a subnumering, it can be convenient to consider a multi-subnumbering
of $\mathcal{P}(X)$ associated to $\nu_{SD}$, we thus introduce
a multi-subnumbering $\nu_{SD,m}$ defined as follows:

\[
\forall n\in\text{dom}(\nu_{SD}),\,\nu_{SD,m}(n)=\{A\subseteq X,\,A\cap X_{\nu}=\nu_{SD}(n)\},
\]
\[
\forall n\notin\text{dom}(\nu_{SD}),\,\nu_{SD,m}(n)=\emptyset.
\]
Thus the name of a set $A$ for $\nu_{SD,m}$ encodes a description
of the set $A\cap X_{\nu}$ of computable points of $A$ for $\nu_{SD}$.
Since this does not define $A$ uniquely, $\nu_{SD,m}$ is a multi-subnumbering. 

We then extend the notion of semi-decidability to subsets of $X$
that need not be subsets of $X_{\nu}$: a subset $A$ of $X$ is \emph{semi-decidable}
if it has a $\nu_{SD,m}$-name, i.e. if $A\cap X_{\nu}$ is semi-decidable
following the previous definition. 

Define a numbering $s$ of $\{0,1\}$ by the following: for any $n\in\mathbb{N}$,
$s(n)=0$ if $\varphi_{n}(n)\uparrow$ and $s(n)=1$ if $\varphi_{n}(n)\downarrow$
(the\emph{ Sierpi\'{n}ski numbering}). Then $\nu_{SD,m}\equiv s^{\nu}$.

\subsection{Computably enumerable sets }
\begin{defn}
A numbered set $(X,\nu)$ is \emph{computably enumerable} (or \emph{$X$
is $\nu$-computably enumerable}) if it is empty or if there exists
a surjective function $f:\mathbb{N}\twoheadrightarrow X$ which is
$(\text{id}_{\mathbb{N}},\nu)$-computable. 
\end{defn}

This corresponds to a possible definition for ``being effectively
countable''. We use the abbreviation ``c.e.'' for computably enumerable.

We define the subnumbering of $\mathcal{P}(X)$ associated to computably
enumerable subsets. 

Consider the usual numbering $W$ of c.e. subsets of $\mathbb{N}$:
$W_{i}=\text{dom}(\varphi_{i})$. 

We define the subnumbering $\nu_{ce}$ as follows: 
\[
\text{dom}(\nu_{ce})=\{i\in\mathbb{N},\,W_{i}\subseteq\text{dom}(\nu)\},
\]
\[
\forall i\in\text{dom}(\nu_{ce}),\,\nu_{ce}(i)=\nu(W_{i}).
\]

This subnumbering will be very useful: if $\beta$ is the numbering
of a basis, the associated numbering of Lacombe sets is $i\mapsto\bigcup\beta_{ce}(i)$. 

\subsection{\label{subsec:Computable-real-numbers}Computable real numbers and
computable metric spaces}

\subsubsection{Numbering of $\mathbb{Q}$}

Denote by $c_{\mathbb{Q}}:\mathbb{N}\rightarrow\mathbb{Q}$ the map
$\langle p,q,r\rangle\mapsto(-1)^{p}\frac{q}{r+1}$. It is a natural
numbering of $\mathbb{Q}$, and any other natural numbering of $\mathbb{Q}$
can be seen to be equivalent to this one. 

\subsubsection{Computable real numbers}

We denote by $c_{\mathbb{R}}$ the Cauchy subnumbering of $\mathbb{R}$:
a real is described by the $c_{\mathbb{Q}}^{\text{id}_{\mathbb{N}}}$-name
of a computable sequence of rationals that converges towards it at
exponential speed. 

More precisely: 
\begin{defn}
The subnumbering $c_{\mathbb{R}}$ is defined by 
\[
\text{dom}(c_{\mathbb{R}})=\{i\in\mathbb{N},\,\exists x\in\mathbb{R},\,\forall n\in\mathbb{N},\,\left|c_{\mathbb{Q}}(\varphi_{i}(n))-x\right|<2^{-n}\};
\]
\[
\forall i\in\text{dom}(c_{\mathbb{R}}),\,c_{\mathbb{R}}(i)=\lim_{n\rightarrow\infty}c_{\mathbb{Q}}(\varphi_{i}(n)).
\]
\end{defn}

We will also use the subnumbering associated to\emph{ }left computable
reals, which we denote $c_{\nearrow}$. Left computable reals are
the limits of computable and increasing sequences of rationals. Left
computable reals are sometimes also called \emph{lower semi-computable
reals,} the term left computable seems to be the more common one\emph{,
}following\emph{ \cite{Rettinger_2021}.}
\begin{defn}
Define a subnumbering $c_{\nearrow}$ of $\mathbb{R}$ by 
\[
\text{dom}(c_{\nearrow})=\{i\in\mathbb{N},\,\exists x\in\mathbb{R},\,\forall n\in\mathbb{N},\,c_{\mathbb{Q}}(\varphi_{i}(n))\le c_{\mathbb{Q}}(\varphi_{i}(n+1))<x\};
\]
\[
\forall i\in\text{dom}(c_{\nearrow}),\,c_{\nearrow}(i)=\lim_{n\rightarrow\infty}c_{\mathbb{Q}}(\varphi_{i}(n)).
\]

The $c_{\nearrow}$-computable reals are called the \emph{left computable
reals. }
\end{defn}

We usually extend $c_{\nearrow}$ to $\mathbb{R}\cup\{+\infty\}$,
by considering that a description for $+\infty$ is an unbounded increasing
sequence. 

The following folklore results will be very useful: 
\begin{prop}
The function $\sup:(\mathbb{R}\cup\{+\infty\})^{\mathbb{N}}\rightarrow\mathbb{R}\cup\{+\infty\}$
(which maps a sequence to its supremum) is $(c_{\nearrow}^{id_{\mathbb{N}}},c_{\nearrow})$-computable. 
\begin{prop}
The set of left computable reals is $c_{\nearrow}$-computably enumerable,
as are intervals in $\mathbb{R}_{c_{\nearrow}}$ (open or closed,
possibly with infinite bounds).
\end{prop}

\end{prop}

\subsubsection{Computable metric spaces }
\begin{defn}
[\cite{Hertling1996}]\label{def: (non-effectively-separable) CMS}A
\emph{non-necessarily effectively separable computable metric space}
is a triple $(X,\nu,d)$, where $(X,d)$ is a metric space, $\nu$
is a subnumbering of $X$ with dense image, and such that $d$ is
a $(\nu\times\nu,c_{\mathbb{R}})$-computable function. 
\end{defn}

Throughout, we will use the following notation: by\emph{ computable
metric space$^{*}$}, or $\text{CMS}^{*}$, we designate a non-necessarily
effectively separable computable metric space.

If, in Definition \ref{def: (non-effectively-separable) CMS}, we
ask $\nu$ to be defined on all of $\mathbb{N}$, then this definition
is exactly the more common definition of a \emph{computable metric
space} \cite{Brattka2003,Iljazovic2021}. 

One of the crucial objectives of the present study is to guarantee
that the notion of computable topological space always generalizes
that of computable metric space\emph{$^{*}$}. 

A metric space is \emph{effectively complete} if there is an algorithm
that on input the code for a Cauchy sequence that converges exponentially
fast (i.e. the distance between the $n$-th term of the sequence to
the $m$-th term, with $m>n$, is less than $2^{-n}$) produces the
name of its limit, which must thus always exist. One in fact practically
never needs this hypothesis, but only the weaker one, that \emph{provided
that the limit exists}, it can be computed. 

It is well known \cite{Kushner1984} that existence of an algorithm
that computes limits of sequences that converge at exponential speed
is equivalent to existence of an algorithm that takes as input both
a computable and convergent sequence and a computable function that
bounds its convergence speed and computes the limit of this sequence. 
\begin{defn}
A computable metric space\emph{$^{*}$} $(X,\nu,d)$ admits \emph{an
algorithm of passage to the limit} \emph{\cite{Kushner1984}} if there
is a procedure that on input the $\nu^{id_{\mathbb{N}}}$-name of
a computable sequence that converges at exponential speed produces
the name of the limit of this sequence. 
\end{defn}

This notion has various names in the literature:\emph{ }$(X,\nu,d)$\emph{
allows effective limit passing \cite{Spr98}, }or is \emph{pre-complete
\cite{Kushner1984}, satisfies condition (A) \cite{Moschovakis1964}. }

Whenever $(X,\nu,d)$ does not admit an algorithm of passage to the
limit, it is possible to consider the subnumbering $\mu$ of $X$
\emph{that provides the most information while allowing limits to
be computed. }

Denote by $\text{Lim}$ the function defined on the subset of $X^{\mathbb{N}}$
of sequences that converge at exponential speed and which maps a sequence
to its limit. 

The above sentence is then formalized by saying that the set 
\[
I=\{\mu\in\mathcal{N}_{X},\,\nu\le\mu,\,\text{Lim is \ensuremath{(\mu^{\text{id}_{\mathbb{N}}},\mu)}-computable}\}
\]
has a least element, it is the \emph{Cauchy subnumbering associated
to $\nu$}. 

The Cauchy subnumbering associated to $\nu$ is denoted $c_{\nu}$
and defined as follows: 
\[
\text{dom}(c_{\nu})=\{i\in\mathbb{N},\,\exists x\in X,\,\forall n\in\mathbb{N},\,\left|\nu(\varphi_{i}(n))-x\right|<2^{-n}\};
\]
\[
\forall i\in\text{dom}(c_{\nu}),\,c_{\nu}(i)=\lim_{n\rightarrow\infty}\nu(\varphi_{i}(n)).
\]

The Cauchy subnumbering is thus constructed following a general principle
formalized by Weihrauch in \cite[Chapter 2.7]{Weihrauch1987}: if
$(X,\nu)$ is a numbered set and $F$ is a finite set of functions
from $X$ (or $X^{n}$ or $X^{\mathbb{N}}$) to $X$, there is a canonical
way of rendering the functions of $F$ computable, while diminishing
as little as possible the amount of information given by $\nu$. 

The following is well known \cite{Kushner1984}: 
\begin{lem}
If $(X,\nu,d)$ is a computable metric space, then $(X,c_{\nu},d)$
is also a computable metric space, and it admits an algorithm of passage
to the limit. 
\end{lem}

\subsection{Usual definition of computable topological space}

We will sometimes compare our definition of computable topological
space to the following definition: 
\begin{defn}
[Usual definition, \cite{Weihrauch2009ElementaryCT}]\label{def:Usual Definition }A
\emph{``}computable topological space\emph{''} is a triple $(X,\mathcal{B},\beta)$,
where $X$ is a set, $\mathcal{B}$ is a topological basis on $X$
that makes of it a $T_{0}$ space, and $\beta:\mathbb{N}\rightarrow\mathcal{B}$
is a total numbering of $\mathcal{B}$, for which there exists a c.e.
subset\footnote{Note that in the present article, we prefer the formulation ``there
is a computable function that maps $(i,j)$ to the code of a c.e.
subset of $\mathbb{N}$'' to the formulation which involves a c.e.
subset of $\mathbb{N}^{3}$. Both formulations are equivalent, as
can be seen using the universal Turing machine theorem for one direction,
and the Smn Theorem for the converse. The functional formulation is
however more general, in that it does not have to be modified to handle
non-computably enumerable topological bases. It is closer to the formulation
that would be used in the general context of realizability. } $R$ of $\mathbb{N}^{3}$ such that for any $i$, $j$ in $\mathbb{N}$:
\[
\beta(i)\cap\beta(j)=\underset{(i,j,k)\in R}{\bigcup\beta(k)}.
\]
An open set $U$ of $X$ is \emph{computably enumerable} or \emph{effectively}
$\Sigma_{1}^{0}$ if there is a c.e. set $I\subseteq\mathbb{N}$ such
that $U=\underset{i\in I}{\bigcup}\beta(i)$. This defines a numbering
of effective open sets, by associating a code of $I$ to $U$, for
$I$ and $U$ as above. 

A point $x$ of $X$ is called \emph{computable} if the set $\{i\in\mathbb{N}\,\vert\,x\in\beta(i)\}$
is c.e., this also allows to define a subnumbering $\nu:\mathbb{N}\rightharpoonup X$
in the obvious way. 
\end{defn}

In our vocabulary, the above definition is a special case of the notion
of \emph{Lacombe topology}, i.e. of topology arising from a Lacombe
basis. The following additional hypotheses are added to the Lacombe
conditions: the considered space should be $T_{0}$, the basis should
be computably enumerable (thus the spaces is ``effectively second
countable''), and the underlying set $X$ should be equipped with
the subnumbering induced by the numbering $\beta$ of its basis (``a
point is described by the basic open sets that contain it''). 

Each of these additional hypotheses is very natural, and Lacombe bases
are very well behaved on sets that are computably separable. Thus
the above definition provides a very good framework to study computable
topologies. However, it should be clear that the name ``computable
topological space'' cannot be attached to this definition, since
it contains many more hypotheses than the bare fact that we are considering
a computable topology. In particular, if we were to follow this definition,
``computable topological spaces'' do not generalize computable metric
spaces\emph{$^{*}$}, and not every finite topological space is a
``computable topological space''. This last fact should stand out
as especially odd, considering that computability deals with problems
that stem from the existence of infinite sets. The fact that ``every
finite topological space is a computable topological space'' should
be as natural as the fact that ``every finite graph is a computable
graph''. 

Note also that the above definition is too restrictive from the point
of view of Type 2 computability: thanks to the work of Schröder \cite{Schroeder2001,Schroeder2003},
we know that the topological spaces that admit admissible (multi-)representations
are the qcb spaces, which need not be second countable. Yet every
space that admits an admissible representation surely deserves to
be called a (Type 2) computable topological space.

\section{\label{part:Two-notions-of-Bases}Type 1 Computable topological spaces,
first two notions of bases }

\subsection{Computable topological spaces}

\subsubsection{Formal inclusion relations}

We define here formal inclusion relations. We add reflexivity to the
more general definition of ``strong inclusion'' introduced by Spreen. 
\begin{defn}
[Spreen, \cite{Spr98}, Definition 2.3] \label{def:Formal inclusion relation Text }Let
$\mathfrak{B}$ be a subset of $\mathcal{P}(X)$, and $\beta$ a numbering
of $\mathfrak{B}$. Let $\mathring{\subseteq}$ be a binary relation
on $\text{dom}(\beta)$. We say that $\mathring{\subseteq}$ is a
\emph{formal inclusion relation for $(\mathfrak{B},\beta)$} if the
following hold:
\begin{itemize}
\item The relation $\mathring{\subseteq}$ is reflexive and transitive (i.e.
$\mathring{\subseteq}$ is a preorder); 
\item $\forall n,m\in\text{dom}(\beta),\,n\mathring{\subseteq}m\implies\beta(n)\subseteq\beta(m)$.
\end{itemize}
\end{defn}

When working with formal inclusion relations, we will be interested
in functions that can be computed by maps that preserve the formal
inclusion relations. We will also need the definition of a function
computed by an expanding map. 

Let $X$ and $Y$ be sets. Consider a subnumbering $\beta$ of $\mathcal{P}(X)$,
and $\gamma$ a subnumbering of $\mathcal{P}(Y)$. Let $\mathring{\subseteq}_{\beta}$
and $\mathring{\subseteq}_{\gamma}$ be formal inclusion relations
for $\beta$ and $\gamma$ respectively. 
\begin{itemize}
\item A $(\beta,\gamma)$-computable function $f:\mathcal{P}(X)\rightarrow\mathcal{P}(Y)$
\emph{can be computed by a function that is increasing for $\mathring{\subseteq}_{\beta}$
and $\mathring{\subseteq}_{\gamma}$} if there exists a computable
function $F:\mathbb{N}\rightharpoonup\mathbb{N}$ such that 
\[
\forall n\in\text{dom}(\beta),\,f(\beta(n))=\gamma(F(n));
\]
\[
\forall n,m\in\text{dom}(\beta),\,n\mathring{\subseteq}_{\beta}m\implies F(n)\mathring{\subseteq}_{\gamma}F(m).
\]
\item A $(\beta,\beta)$-computable function $g:\mathcal{P}(X)\rightarrow\mathcal{P}(X)$
c\emph{an be computed by a function that is expanding} for $\mathring{\subseteq}_{\beta}$
if there exists a computable function $G:\mathbb{N}\rightharpoonup\mathbb{N}$
such that 
\[
\forall n\in\text{dom}(\beta),\,g(\beta(n))=\beta(G(n));
\]
\[
\forall n\in\text{dom}(\beta),\,n\mathring{\subseteq}_{\beta}G(n).
\]
\end{itemize}
In the same conditions as above, the subnumberings $\beta$ and $\gamma$
allow us to define a subnumbering of $\mathcal{P}(X\times Y)$ by
$\beta\times\gamma(\langle n,m\rangle)=\beta(n)\times\gamma(m)$ (with
the obvious domain). One checks easily that putting 
\[
\langle n_{1},m_{1}\rangle\mathring{\subseteq}_{\beta\times\gamma}\langle n_{2},m_{2}\rangle\iff n_{1}\mathring{\subseteq}_{\beta}n_{2}\,\&\,m_{1}\mathring{\subseteq}_{\gamma}m_{2},
\]
we obtain a formal inclusion relation for $\beta\times\gamma$. 

In metric spaces, we consider the formal inclusion relation that comes
from the relation on balls parametrized by pairs (center-radii):
\[
\forall y,z\in X,\,\forall r_{1},r_{2}\in\mathbb{R},\,(y,r_{1})\mathring{\subseteq}(z,r_{2})\iff d(z,y)+r_{1}\le r_{2}.
\]
 Figure \ref{fig:Inclusion-and-strong} illustrates the fact that
in the square $[0,1]\times[0,1]\subseteq\mathbb{R}^{2}$, the natural
formal inclusion is not the actual inclusion relation. 

In some cases, the formal inclusion relations under scrutiny happens
to be extensional, see the case of continuous domains in \cite{Spr98},
for which the natural formal inclusion is the way below relation.
But sometimes the intentionality is an unavoidable feature. Note that
there are examples in the literature of techniques that may look like
ways of going around the use of formal inclusion relations, but which
in fact do not. For instance, in \cite{Kalantari2008}, Kalantari
and Welch consider a notion of ``sharp filter that converges to a
point $x$'', which is a sequence of basic open sets $(B_{i})_{i\in\mathbb{N}}$
with $\overline{B_{i+1}}\subseteq B_{i}$ and $\bigcap B_{i}=\{x\}$.
The relation $\overline{B_{i+1}}\subseteq B_{i}$ is extensional,
but it cannot always replace the formal inclusion relation: in computable
metric spaces, the Cauchy name of a point $x$ should provide balls
around this point which are \emph{explicitly} \emph{given by arbitrarily
small radii}, and this does not amount to the condition that a sequence
of balls $(B_{i})_{i\in\mathbb{N}}$ with $\overline{B_{i+1}}\subseteq B_{i}$
converges to this point. See \cite{Rauzy2023RPZ} for an explicit
example where these two descriptions differ. And thus the description
of a point via a sharp filter in a metric space as considered in \cite{Kalantari2008}
provides in some cases strictly less information than a Cauchy name
for this point. 

\begin{figure}
\includegraphics[scale=0.8]{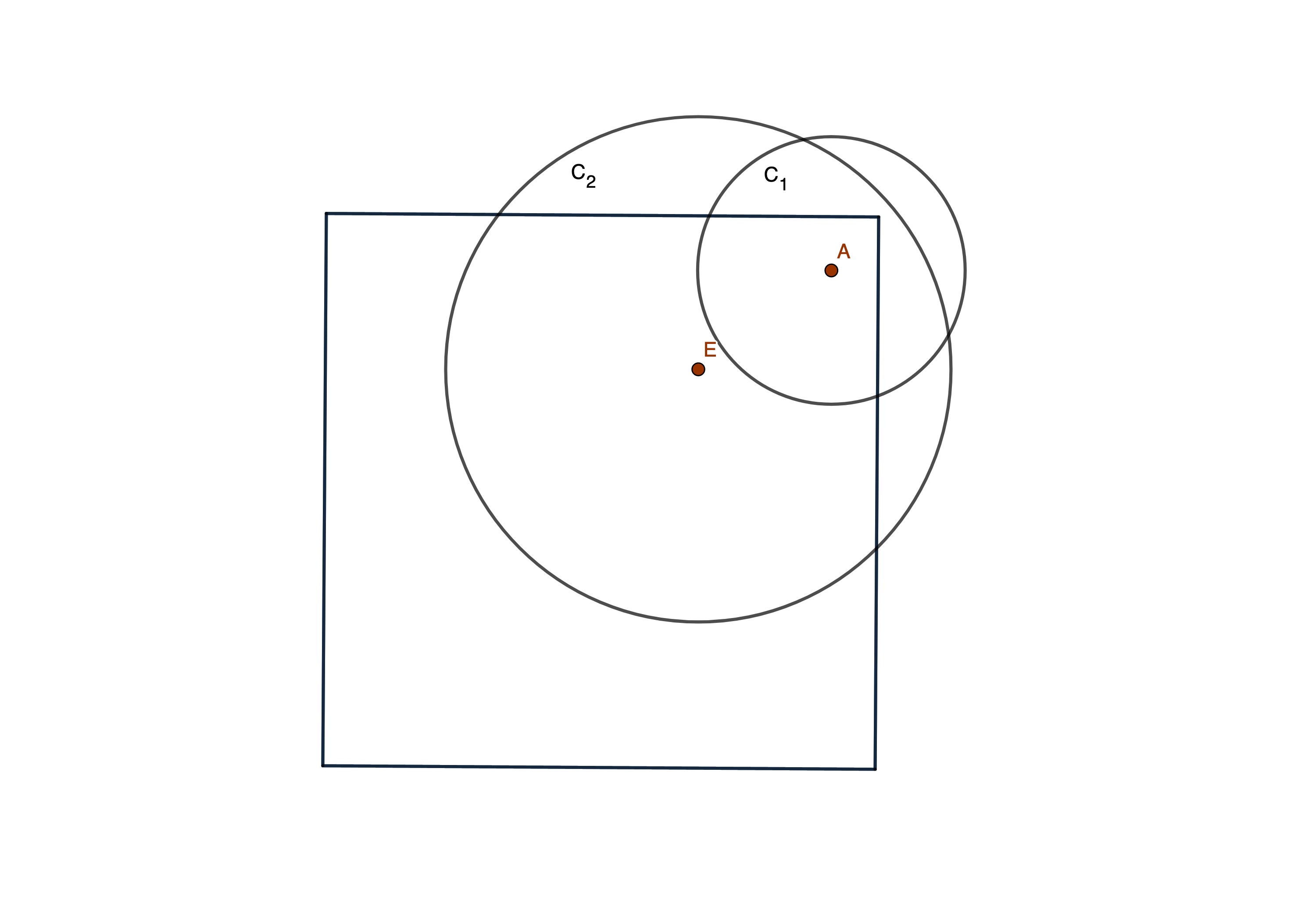}\caption{\label{fig:Inclusion-and-strong}Inclusion and formal inclusion inside
a square}
\end{figure}

\subsubsection{Computable topological spaces}

We give here a general definition that does not suppose that $\nu$
nor $\tau$ be surjective. 
\begin{defn}
\label{def:MAIN DEF}A \emph{Type 1 computable topological space}
is a quintuple $(X,\nu,\mathcal{T},\tau,\mathring{\subseteq})$ where
$X$ is a set, $\mathcal{T}\subseteq\mathcal{P}(X)$ is a topology
on $X$, $\nu:\mathbb{N}\rightharpoonup X$ is a subnumbering of $X$,
$\tau:\mathbb{N}\rightharpoonup\mathcal{T}$ is a subnumbering of
$\mathcal{T}$, and such that: 
\begin{enumerate}
\item The image of $\tau$ generates the topology $\mathcal{T}$;
\item The empty set and $X$ both belong to the image of $\tau$;
\item (Malcev's Condition). The open sets in the image of $\tau$ are uniformly
semi-decidable, i.e. $\tau\le\nu_{SD,m}$; 
\item The operations of taking computable unions and finite intersections
are computable, i.e.:
\begin{enumerate}
\item The function $\bigcup:\mathcal{T}^{\mathbb{N}}\rightarrow\mathcal{T}$
which maps a sequence of open sets $(A_{n})_{n\in\mathbb{N}}$ to
its union is $(\tau^{\text{id}_{\mathbb{N}}},\tau)$-computable and
it can be computed by a function that is expanding and increasing
for $\mathring{\subseteq}$;
\item The function $\bigcap:\mathcal{T}\times\mathcal{T}\rightarrow\mathcal{T}$
which maps a pair of open sets to its intersection is $(\tau\times\tau,\tau)$-computable,
and it can be computed by a function increasing for $\mathring{\subseteq}$. 
\end{enumerate}
\end{enumerate}
The third quoted condition is very important: it is the only condition
that relates the subnumbering of $X$ to its computable topology.
Without it, we do not relate the effective structure of $X$ to its
computable topology.
\end{defn}

More general definitions are possible, and probably better, but we
do not study them here:
\begin{enumerate}
\item To render effective the condition that open sets be stable by unions,
we want to say that taking unions of open sets is computable. 

The union function, which maps a set of open sets to its union, is
defined on the power set of $\mathcal{T}$: denote it $\bigcup:\mathcal{P}(\mathcal{T})\rightarrow\mathcal{T}$.
To ask that this function be computable requires the introduction
of a certain subnumbering of $\mathcal{P}(\mathcal{T})$. We have
chosen above a natural subnumbering of $\mathcal{P}(\mathcal{T})$
induced by the numbering $\tau$ of $\mathcal{T}$: the subnumbering
associated to computable sequences. But there is no reason to think
that this is in general the only valid choice of a subnumbering of
$\mathcal{P}(\mathcal{T})$ to define a computable topology. 

In particular, uncountable sets of open sets can have finite descriptions,
there is no reason to exclude those if they are to arise. 
\item Definition \ref{def:MAIN DEF} can be rewritten using multi-subnumberings,
both of $X$ and of $\mathcal{T}$, instead of subnumberings. All
other conditions can remain as they are. Such a definition is more
general than Definition \ref{def:MAIN DEF} and is useful in different
cases. A general definition of the Ershov topology uses this multi-subnumbering
definition.
\end{enumerate}
The following lemma shows one can use either $\tau^{\text{id}_{\mathbb{N}}}$
or $\tau_{ce}$ in the definition above. 
\begin{lem}
Let $\mu$ be a subnumbering of $\mathcal{P}(X)$ with the empty set
in its image. Then the union function $\bigcup:\mathcal{P}(\mathcal{P}(X))\rightarrow\mathcal{P}(X)$
is $(\mu_{ce},\mu)$-computable if and only if it is $(\mu^{\text{id}_{\mathbb{N}}},\mu)$-computable. 
\end{lem}

\begin{proof}
This simply comes from the fact that the $\mu_{ce}$-name of a possibly
empty or finite sequence of names can be transformed into a proper
sequence by filling gaps with a $\mu$-name for the empty set, which
does not change the resulting union. 
\end{proof}
\begin{defn}
The $\tau$-computable sets are called the \emph{effective open sets.
}Sets whose complement is $\tau$-computable are the \emph{effective
closed sets. }
\end{defn}

We define effectively continuous functions. 
\begin{defn}
\label{def: eff C0 text def}A function $f$ between two computable
topological spaces $(X,\nu,\mathcal{T}_{1},\tau_{1},\mathring{\subseteq}_{1})$
and $(Y,\mu,\mathcal{T}_{2},\tau_{2},\mathring{\subseteq}_{2})$ is
called \emph{effectively continuous} if it is classically continuous,
if the function $f^{-1}:\mathcal{T}_{2}\rightarrow\mathcal{T}_{1}$
is $(\tau_{2},\tau_{1})$-computable, and if it can be computed by
a function increasing with respect to the formal inclusion relation:
there should exist a computable function $h:\text{dom}(\tau_{2})\rightarrow\text{dom}(\tau_{1})$
for which:
\[
\forall n\in\text{dom}(\tau_{2}),\,\tau_{1}(h(n))=f^{-1}(\tau_{2}(n)),
\]
\[
\forall n_{1},n_{2}\in\text{dom}(\tau_{2}),\,n_{1}\mathring{\subseteq}_{2}n_{2}\implies h(n_{1})\mathring{\subseteq}_{1}h(n_{2}).
\]
\end{defn}

In other words, there should be a procedure that, given an effective
open set for $\tau_{2}$, produces a description of its preimage as
an effective open for $\tau_{1}$, and which respects formal inclusions.

We also define continuity at a point. 
\begin{defn}
\label{def: C0 at a point }A function $f$ between two computable
topological spaces $(X,\nu,\mathcal{T}_{1},\tau_{1},\mathring{\subseteq}_{1})$
and $(Y,\mu,\mathcal{T}_{2},\tau_{2},\mathring{\subseteq}_{2})$ is
called \emph{effectively continuous at a point $x\in X$} if there
exists a procedure that, given the $\tau_{2}$-name of an open set
$O_{2}$ that contains $f(x)$, produces the $\tau_{1}$-name of an
open set $O_{1}$ such that $x\in O_{1}$ and $O_{1}\subseteq f^{-1}(O_{2})$.
In addition, this procedure should be increasing with respect to the
formal inclusion relations of $\tau_{1}$ and $\tau_{2}$. 
\end{defn}

Note that this definition does not require $x$ to be a computable
point.
\begin{defn}
\label{def: Eff finer eff coarser }Let $(\mathcal{T}_{1},\tau_{1},\mathring{\subseteq}_{1})$
and $(\mathcal{T}_{2},\tau_{2},\mathring{\subseteq}_{2})$ be two
computable topologies defined on $(X,\nu)$. Then $(\mathcal{T}_{1},\tau_{1},\mathring{\subseteq}_{1})$
is \emph{effectively coarser} than $(\mathcal{T}_{2},\tau_{2},\mathring{\subseteq}_{2})$,
or $(\mathcal{T}_{2},\tau_{2},\mathring{\subseteq}_{2})$ is \emph{effectively
finer }than $(\mathcal{T}_{1},\tau_{1},\mathring{\subseteq}_{1})$,
if:
\begin{itemize}
\item $\mathcal{T}_{1}\subseteq\mathcal{T}_{2}$;
\item There is a procedure that, given the $\tau_{1}$-name of a set of
$\mathcal{T}_{1}$, produces the $\tau_{2}$-name of that set, and
which is increasing with respect to $\mathring{\subseteq}_{1}$ and
$\mathring{\subseteq}_{2}$. 
\end{itemize}
This is denoted by $(\mathcal{T}_{1},\tau_{1},\mathring{\subseteq}_{1})\subseteq_{e}(\mathcal{T}_{2},\tau_{2},\mathring{\subseteq}_{2})$
\cite{Spr98}. 
\end{defn}

Equivalently, one has $(\mathcal{T}_{1},\tau_{1},\mathring{\subseteq}_{1})\subseteq_{e}(\mathcal{T}_{2},\tau_{2},\mathring{\subseteq}_{2})$
if and only if the identity on $X$, seen as going from $X$ equipped
with $(\mathcal{T}_{2},\tau_{2},\mathring{\subseteq}_{2})$ to $X$
equipped with $(\mathcal{T}_{1},\tau_{1},\mathring{\subseteq}_{1})$,
i.e. $\text{id}_{X}:\,(X,\nu,\mathcal{T}_{2},\tau_{2},\mathring{\subseteq}_{2})\rightarrow(X,\nu,\mathcal{T}_{1},\tau_{1},\mathring{\subseteq}_{1})$,
is effectively continuous.

If two computable topologies $(X,\nu,\mathcal{T}_{1},\tau_{1},\mathring{\subseteq}_{1})$
and $(X,\nu,\mathcal{T}_{2},\tau_{2,},\mathring{\subseteq}_{2})$
are each effectively coarser than the other, then $\mathcal{T}_{1}=\mathcal{T}_{2}$
and $\tau_{1}\equiv\tau_{2}$. In this case, these two computable
topologies should be considered identical, we sometimes say that they
are \emph{effectively equivalent}.

\subsubsection{Inclusion as a formal inclusion }

In this section, we make the remark that when the formal inclusion
we are using is the actual inclusion, i.e. $n\mathring{\subseteq}m\iff\tau(n)\subseteq\tau(m)$,
all the requirements related to ``being increasing with respect to
the formal inclusions'' are trivially satisfied.

This is what shows that the use of formal inclusion relations is not
restrictive. 

We denote by $\subseteq$ the relation on names of open sets given
by $n\subseteq m\iff\tau(n)\subseteq\tau(m)$. The following proposition
is immediate. 
\begin{prop}
When using the inclusion $\subseteq$ as a formal inclusion relation,
the conditions that appear in Definition \ref{def:MAIN DEF} involving
the formal inclusion, namely that computation of unions and intersections
is increasing for $\subseteq$, and that computation of unions is
expanding, are always trivially satisfied. 
\end{prop}

The following is also immediate:
\begin{prop}
If $f$ is a function between two computable topological spaces $(X,\nu,\mathcal{T}_{1},\tau_{1},\subseteq)$
and $(Y,\mu,\mathcal{T}_{2},\tau_{2},\subseteq)$ both equipped with
the inclusion as a formal inclusion relation, then $f$ is effectively
continuous if and only if $f^{-1}:\mathcal{T}_{2}\rightarrow\mathcal{T}_{1}$
is $(\tau_{2},\tau_{1})$-computable. That is to say the monotonicity
condition with respect to formal inclusion is automatically satisfied. 
\end{prop}

We leave it to the reader to remark that similar results hold for
Definitions \ref{def: C0 at a point } and \ref{def: Eff finer eff coarser }. 

Because of these remarks, we make the following convention: whenever
the formal inclusion attached to a computable topological space is
the actual inclusion relation, we will simply omit it, and we will
thus write $(X,\nu,\mathcal{T}_{1},\tau_{1})$ instead of $(X,\nu,\mathcal{T}_{1},\tau_{1},\subseteq)$. 

In the following section, we discuss Nogina bases and Lacombe bases.
These do not use a formal inclusion, and thus we will not make an
explicit reference to the fact that we are using inclusion as a formal
inclusion. 

\subsection{\label{subsec:First-two-notions-of-bases}First two notions of bases}

\subsubsection{Nogina Bases}

We define Nogina bases. The notion of a Nogina basis is valid only
when $\nu$ is a numbering, and the set under consideration is countable.
What this means is the following: in any setting, whether $\nu$ is
a numbering or only a subnumbering, we can write down the two following
definitions, Definition \ref{def: Nogina basis Text } and Definition
\ref{def: top gen by eff basis}. But then, in order to be able to
prove that those definitions yield an effective topology (Proposition
\ref{prop:The-effective-topology of a basis}), we need the hypothesis
that $\nu$ is an actual numbering. 
\begin{defn}
[Nogina, \cite{Nogina1966}]\label{def: Nogina basis Text }A \emph{Nogina
basis }for\emph{ $(X,\nu)$, }where $\nu$ is a numbering, is a pair
$(\mathcal{B},\beta)$, where $\mathcal{B}$ is a subset of $P(X)$
and $\beta$ is a numbering of $\mathcal{B}$, and that satisfies
the following conditions:
\begin{enumerate}
\item The elements of the basis $(\mathcal{B},\beta)$ are uniformly semi-decidable,
i.e. $\beta\le\nu_{SD}$;
\item There is a procedure that given a $\nu$-name for a point $x$ in
$X$ produces the $\beta$-name of a set $B$ in $\mathcal{B}$ such
that $x\in B$ (and in particular any point in $X$ belongs to an
element of $\mathcal{B}$); 
\item There is a procedure that given $\beta$-names for two sets $B_{1}$
and $B_{2}$ and the $\nu$-name of a point $x$ in the intersection
$B_{1}\cap B_{2}$ produces the $\beta$-name of a set $B_{3}$ such
that $B_{3}\subseteq B_{1}\cap B_{2}$ and $x\in B_{3}$. 
\end{enumerate}
\end{defn}

This is a particular case of Bauer's definition of the ``Pointwise
Topology'' \cite[Definition 5.4.1]{Bauer2000}. We will now explicitly
define the computable topology associated to this basis. 
\begin{defn}
\label{def: top gen by eff basis}Let $(X,\nu,\mathcal{B},\beta)$
be a numbered set equipped with a Nogina basis. The \emph{Nogina topology
associated to the Nogina basis $(\mathcal{B},\beta)$ }is a pair $(\mathcal{T},\tau)$
defined as follows:
\begin{enumerate}
\item $\mathcal{T}$ is the topology on $X$ generated by $\mathcal{B}$
(as a classical basis);
\item $\tau$ is a subnumbering of $\mathcal{T}$, such that the $\tau$-name
of an open set $O$ is an encoded pair of algorithms, one that testifies
for the fact that $O$ is semi-decidable, and one which, given the
$\nu$-name of a point $x$ in $O$, produces the $\beta$-name of
an element $B$ of the basis such that $x\in B$ and $B\subseteq O$. 
\end{enumerate}
Thus a set $O$ is an effective open for the topology generated by
the Nogina basis \emph{$(\mathcal{B},\beta)$} if and only if it is
semi-decidable and there is an algorithm that, given a point in $O$,
produces a element of the basis which contains this point and is contained
in $O$. 

The fundamental result to use Nogina bases is of course:
\end{defn}

\begin{prop}
\label{prop:The-effective-topology of a basis}The Nogina topology
associated to a Nogina basis is a computable topology. 
\end{prop}

\begin{proof}
Let $(X,\nu,\mathcal{B},\beta)$ be a numbered set equipped with a
Nogina basis, and let $(\mathcal{T},\tau)$ denote the Nogina topology
associated to $(X,\nu,\mathcal{B},\beta)$. We show that it is a computable
topology, proving each condition that appeared in Definition \ref{def:MAIN DEF}
in order: 
\begin{enumerate}
\item The image of $\tau$ generates the topology $\mathcal{T}$. 

This comes from the easy fact that each element of the basis $\mathcal{B}$
is effectively open for $(\mathcal{T},\tau)$, in fact we have $\beta\le\tau$.
Thus the topology generated by the image of $\tau$ is $\mathcal{T}$. 
\item The empty set and $X$ both belong to the image of $\tau$. 

This is clear by vacuity for the empty set, for $X$ it follows directly
from point (2) of the definition of a Nogina basis. 
\item (Malcev's Condition). The open sets in the image of $\tau$ are uniformly
semi-decidable, i.e. $\tau\le\nu_{SD}$. 

This was presupposed in our definition. 
\item The function $\bigcup:\mathcal{T}^{\mathbb{N}}\rightarrow\mathcal{T}$
which maps a sequence of open sets $(A_{n})_{n\in\mathbb{N}}$ to
its union is $(\tau^{\text{id}_{\mathbb{N}}},\tau)$-computable.

This is straightforward: a computable union of semi-decidable sets
is semi-decidable, and this is uniform by dovetailing. Given a point
$x$ in such a union, one can find at least one element of the union
to which $x$ belongs, and then find a basic open set which contains
$x$ and is a subset of the union. 
\item The function $\bigcap:\mathcal{T}\times\mathcal{T}\rightarrow\mathcal{T}$
which maps a pair of open sets to its intersection is $(\tau\times\tau,\tau)$-computable. 

This follows directly from the intersection condition for the Nogina
basis $(\mathcal{B},\beta)$. 

\end{enumerate}
\vspace{-0,5cm}
\end{proof}

\subsubsection{Lacombe Bases}

We define Lacombe bases on a subnumbered set $(X,\nu)$. This definition
is similar to \cite[Definition 5.4.2]{Bauer2000} and to \cite[Definition 4]{Weihrauch2009ElementaryCT}. 

Recall that if $\mu$ is a subnumbering of $Y$, $\mu_{ce}$ designates
the subnumbering associated to computably enumerable sets of elements
of $Y$. 
\begin{defn}
\label{def: Lacombe-basis}A \emph{Lacombe basis} for $(X,\nu)$ is
a pair $(\mathcal{B},\beta)$, where $\mathcal{B}$ is a subset of
$P(X)$, $\beta$ is a numbering of $\mathcal{B}$, and that satisfies
the following conditions:
\begin{enumerate}
\item The elements of the basis are uniformly semi-decidable, i.e. $\beta\le\nu_{SD,m}$;
\item Any point in $X$ belongs to an element of $\mathcal{B}$, and $X$
can be written as a computable union of basic sets;
\item There is a procedure that given the $\beta$-names for two sets $B_{1}$
and $B_{2}$ produces the $\beta_{ce}$-name of a sequence of basic
sets $(B_{i})_{i\ge3}$ such that: 
\[
B_{1}\cap B_{2}=\underset{i\ge3}{\bigcup}B_{i}.
\]
\end{enumerate}
\end{defn}

It is clear that those conditions are more restrictive than those
of Nogina bases:
\begin{prop}
If $\nu$ is onto, any Lacombe basis is a Nogina basis. 
\end{prop}

We now define the computable topology associated to a Lacombe basis. 
\begin{defn}
\label{def:Top from Lacombe basis}Let $(X,\nu,\mathcal{B},\beta)$
be a subnumbered set equipped with a Lacombe basis. The \emph{Lacombe
topology associated to the Lacombe basis $(\mathcal{B},\beta)$ }is
a pair $(\mathcal{T},\tau)$ defined as follows:
\begin{enumerate}
\item $\mathcal{T}$ is the topology on $X$ generated by $\mathcal{B}$
(as a classical basis);
\item The subnumbering $\tau$ is obtained as the composition of the union
function with the subnumbering $\beta_{ce}$. This means that a $\tau$-name
of a set $O$ is the $\beta_{ce}$-name for a finite or infinite sequence
$(B_{n})_{n\in I}$ such that $\underset{n\in I}{\cup}B_{n}=O$. 
\end{enumerate}
Thus the effective open sets of this topology are exactly the Lacombe
sets. And of course we must show that this defines correctly a computable
topology: 
\end{defn}

\begin{prop}
\label{prop:Lacombe-topology}The Lacombe topology associated to a
Lacombe basis is a computable topology. 
\end{prop}

\begin{proof}
Let $(X,\nu,\mathcal{B},\beta)$ be a numbered set equipped with a
Lacombe basis, and let $(\mathcal{T},\tau)$ denote the computable
topology associated to $(X,\nu,\mathcal{B},\beta)$. We show that
it is a computable topology, proving each condition that appeared
in Definition \ref{def:MAIN DEF} in order: 
\begin{enumerate}
\item The image of $\tau$ generates the topology $\mathcal{T}$. 

This comes from the easy fact that $\beta\le\tau$ (each basic open
set is an effective open set). Thus the topology generated by the
image of $\tau$ is $\mathcal{T}$. 
\item The empty set and $X$ both belong to the image of $\tau$. 

The empty set is the union of no basic open sets. For the whole set
$X$, this was guaranteed thanks to (2) in Definition \ref{def: Lacombe-basis}. 
\item (Malcev's Condition). The open sets in the image of $\tau$ are uniformly
semi-decidable, i.e. $\tau\le\nu_{SD,m}$. 

This comes from the fact that a computable union of semi-decidable
sets is semi-decidable. 
\item The function $\bigcup:\mathcal{T}^{\mathbb{N}}\rightarrow\mathcal{T}$
which maps a sequence of open sets $(A_{n})_{n\in\mathbb{N}}$ to
its union is $(\tau^{\text{id}_{\mathbb{N}}},\tau)$-computable.

This comes from the fact that a computable union of a computable union
of basic sets can be obtained as a single computable union by dovetailing. 
\item The function $\bigcap:\mathcal{T}\times\mathcal{T}\rightarrow\mathcal{T}$
which maps a pair of open sets to its intersection is $(\tau\times\tau,\tau)$-computable. 

This follows directly from the intersection condition for the Lacombe
basis $(\mathcal{B},\beta)$, together with the previous point. 

\end{enumerate}
\vspace{-0,5cm}
\end{proof}

\subsection{\label{subsec:Equivalence-of-bases}Equivalence of bases}

Note that the two notions of bases come with different notions of
equivalence: 
\begin{itemize}
\item Nogina bases: $(\mathcal{B}_{1},\beta_{1})$ is equivalent to $(\mathcal{B}_{2},\beta_{2})$
if there is an algorithm that, given a $\nu$-name for a point $x$
and a $\beta_{1}$-name of a ball that contains it, constructs the
$\beta_{2}$-name of a ball that contains $x$ and that is contained
in the first ball, and another algorithm that does the converse operation. 
\item Lacombe bases: $(\mathcal{B}_{1},\beta_{1})$ is equivalent to $(\mathcal{B}_{2},\beta_{2})$
if there is an algorithm that, given the $\beta_{1}$-name of a basic
open $B$ of the first basis, produces a name of $B$ as a Lacombe
set for $(\mathcal{B}_{2},\beta_{2})$, and conversely. 
\end{itemize}
And we of course have the following obvious propositions:
\begin{prop}
The Nogina topologies generated by equivalent Nogina bases are identical. 
\begin{prop}
The Lacombe topologies generated by equivalent Lacombe bases are identical. 
\end{prop}

\end{prop}

\section{\label{sec:Problems-with-Lacombe-bases}Comparison between Lacombe
and Nogina bases in different contexts}

In this section, we compare Nogina bases and Lacombe bases. Since
Nogina bases are defined only for numberings, we consider here only
countable sets equipped with numberings. 

\subsection{Ershov topology}

Here we detail the following result: the Lacombe topology generated
by singletons is not always the ``effective discrete topology'',
i.e. the Ershov topology. 

The Ershov topology on a subnumbered set $(X,\nu)$ is the topology
generated by all $\nu$-semi-decidable sets. We denote it $\mathcal{E}$.
When $\nu$ is a numbering, we can naturally equip $X$ with a computable
topology, which we also call the Ershov topology, using as a subnumbering
of $\mathcal{E}$ the numbering of semi-decidable sets. Denote this
set $(X,\nu,\mathcal{E},\nu_{SD})$. (When $\nu$ is a subnumbering,
$\nu_{SD}$ has to be replaced by the multi-numbering $\nu_{SD,m}$,
the Ershov topology is then a computable topology following the multi-subnumbering
version of Definition \ref{def:MAIN DEF}.)

The Ershov topology can be seen as the \emph{effective discrete topology},
as is shown by the following obvious proposition:
\begin{prop}
Any computable function from $(X,\nu,\mathcal{E},\nu_{SD})$ to a
computable topological space is effectively continuous. 
\end{prop}

Another classical characterization of the discrete topology is that
it is the topology generated by singletons. Indeed, on any set $X$,
the set of singletons always forms, classically, a topological basis,
and the topology this basis generates is exactly the power-set of
$X$. 

For the singletons to form the basis of a Nogina topology on $(X,\nu)$,
equality has to be semi-decidable for $\nu$. When it is, they form
a Nogina  basis. For the singletons to form a Lacombe basis, we need
to add the additional requirement that $X$ be $\nu$-r.e., for it
to be effectively open. However, we can dispense with that requirement
if we consider the basis that consists in all singletons of $X$ \emph{together
with the set $X$. }
\begin{prop}
Singletons form a Nogina  basis on $(X,\nu)$ if and only if equality
is semi-decidable for $\nu$. In this case, the computable topology
they generate is the Ershov topology. 
\end{prop}

\begin{proof}
Straightforward.
\end{proof}
\begin{prop}
The set that consists of all singletons of $X$ together with $X$
defines a Lacombe basis on $(X,\nu)$ if and only if equality is semi-decidable
for $\nu$.
\end{prop}

\begin{proof}
Straightforward.
\end{proof}
\begin{defn}
The Lacombe topology generated by singletons and $X$ is called the
\emph{discrete Lacombe topology. }
\end{defn}

We finally give here an example where the Ershov topology differs
from the discrete Lacombe topology, this is the simplest example of
a set where the two notions of effective bases give different results. 
\begin{example}
\label{exa:MAIN Example }Suppose that $X\subseteq\mathbb{N}$ is
a non-computably enumerable set, and consider the set $Y=2X\cup(2X+1)$.
Consider the numbering $\text{id}_{Y}$ of $Y$, it is the identity
on $Y$. 

It follows directly from the definition of the discrete Lacombe topology
that $2X$ and and $2X+1$ are not effective open sets for this topology.
However, they are (semi-)decidable subsets of $Y$, and thus effective
open sets in the Ershov topology. 

Thus those computable topologies differ. 
\end{example}

\subsection{\label{subsec:Metric-spaces I }Metric spaces}

What may be the most important problem with Lacombe bases is the following:
in a computable metric space\emph{$^{*}$}, the open balls do not
always form a Lacombe basis, and when they do, effective continuity
with respect to the Lacombe basis does not always correspond to effective
continuity with respect to the metric. 

And thus the notion of ``computable topological space'', if restricted
to Lacombe spaces, does not generalize that of a computable metric
space\emph{$^{*}$}. 

\subsubsection{Nogina topology associated to a computable metric space\emph{$^{*}$}}

Let $(X,\nu,d)$ be a computable metric space\emph{$^{*}$}. Suppose
that $\nu$ is onto. Recall that $c_{\mathbb{R}}$ designates the
Cauchy subnumbering of $\mathbb{R}$. Denote by $B(x,r)$ the open
ball in $X$ of center $x$ and of radius $r$. 

The numbering $\nu$ induces a numbering $\beta$ of open balls in
$X$: define $\beta(\langle p,q\rangle)=B(\nu(p),c_{\mathbb{R}}(q))$
for all $p$ in $\text{dom}(\nu)$ and $q$ in $\text{dom}(c_{\mathbb{R}})$. 
\begin{prop}
\label{prop: Open-balls Nogina basis}The open balls associated to
the numbering $\beta$ form a Nogina  basis. 
\end{prop}

\begin{proof}
This is straightforward: an open ball is always a semi-decidable set,
and given two balls $B(x_{1},r_{1})$ and $B(x_{2},r_{2})$ and a
point $y$ in their intersection, a radius $r_{3}$ can be computed
such that $B(y,r_{3})\subseteq B(x_{1},r_{1})\cap B(x_{2},r_{2})$:
take $r_{3}=\min(d(y,x_{1})-r_{1},d(y,x_{2})-r_{2})$, $r_{3}$ is
clearly given thanks to a $(\nu^{3}\times c_{\mathbb{R}}^{2},c_{\mathbb{R}})$-computable
function. 
\end{proof}
\begin{defn}
The \emph{metric Nogina topology associated to a computable metric
space$^{*}$ $(X,\nu,d)$ }is the computable topology generated by
the open balls of $X$ seen as forming a Nogina basis. 
\end{defn}

Of course, continuity for functions between metric spaces also has
its definition, which can be effectivised: 
\begin{defn}
\label{def: metric continuous Nogina }A function $f$ between computable
metric spaces\emph{$^{*}$} $(X,\nu,d)$ and $(Y,\mu,d)$ is \emph{effectively
metric continuous} if there is a procedure that, given a $\nu$-name
for a point $x$ in $X$ and a $c_{\mathbb{R}}$-name for a computable
real $\epsilon>0$, produces a $c_{\mathbb{R}}$-name for a computable
real $\eta$ such that: 
\[
\forall y\in X,\,d(x,y)<\eta\implies d(f(x),f(y))<\epsilon.
\]
\end{defn}

The following theorem shows that the two notions of continuity we
can define on metric spaces agree. 
\begin{thm}
\label{thm: Nogina-metric-continuity - Nogina continuity}Effective
metric continuity agrees with effective continuity with respect to
the Nogina metric topology. 
\end{thm}

The proof is straightforward.
\begin{proof}
Let $f$ be a computable function between computable metric spaces\emph{$^{*}$}
$(X,\nu,d)$ and $(Y,\mu,d)$. 

Suppose first that $f$ is effectively metric continuous. Let $O_{1}$
be an effective open of $Y$ for the metric topology, we show that
$f^{-1}(O_{1})$ is effectively open for the metric topology of $X$.
First, $f^{-1}(O_{1})$ is automatically semi-decidable, and a $\nu_{SD}$
code for it can be obtained. Notice now that if $n$ is the $\nu$-name
for a point $x$ of $f^{-1}(O_{1})$, we can find an open ball contained
in $f^{-1}(O_{1})$ that contains $x$, using the effective metric
continuity: because $f$ is computable, find a $\mu$-name of $f(x)$,
then, because $O_{1}$ is an effective open set of $(Y,\mu,d)$, find
a radius $\epsilon$ such that $B(f(x),\epsilon)\subseteq O_{1}$,
then use the metric continuity to find a radius $\eta$ such that
$f(B(x,\eta))\subseteq B(f(x),\epsilon)$, the open ball $B(x,\eta)$
is thus a ball contained in $f^{-1}(O_{1})$ that contains $x$. This
shows that $f^{-1}(O_{1})$ is an effective open set and that a name
for it can be obtained from that of $O_{1}$. 

Suppose now that $f$ is effectively continuous for the Nogina topology.
Let $x=\nu(n)$ be a point in $X$ and $\epsilon$ be a radius. A
name of the open ball $B(f(x),\epsilon)$ as an effective open of
$Y$ can be computed from a name of $\epsilon$ and a $\mu$-name
of $f(x)$. Then, a name of $f^{-1}(B(f(x),\epsilon))$ as an effective
open in $X$ can be computed, and thus also the name of an open ball
$B(y,\eta)$ that contains $x$ and whose image is in $B(f(x),\epsilon)$.
This ball does not have to be centered at $x$, but replacing $\eta$
by $\eta_{1}=\eta-d(x,y)$ gives the number required for the metric
continuity at $x$. 
\end{proof}

\subsubsection{Lacombe topology for metric spaces}

On a computable metric space\emph{$^{*}$} $(X,\nu,d)$, the open
balls do not always form a Lacombe basis. Indeed, when producing points
of $X$ is difficult, like when $X$ is not $\nu$-r.e. or does not
have a computable and dense sequence, computing intersections might
be impossible. 
\begin{example}
\label{exa:metric not lacombe}Let $X$ be any non-computably enumerable
subset of $\mathbb{N}$, and put $Y=2\mathbb{N}\cup(2X+1)$. The identity
$\text{id}_{Y}$ defines a natural numbering of $Y$, and the distance
$d(n,m)=\vert n-m\vert$ is computable for this numbering. However,
the open balls do not form a Lacombe basis. Indeed, if it was possible
to compute the intersections $B(2n,2)\cap B(2n+2,2)$ as unions of
balls, it would be possible to enumerate $X$. 
\end{example}

We will now see that when the open balls of a computable metric space\emph{$^{*}$}
do form such a Lacombe basis, the Lacombe topology they generate does
not have to be the metric topology, and so continuity with respect
to the Lacombe topology will not be equivalent to the metric continuity
given by the epsilon-delta statement. 
\begin{example}
\label{exa: cont metric not cont Lacombe}We use again Example \ref{exa:MAIN Example }:
$X\subseteq\mathbb{N}$ is a non-computably enumerable set, $Y$ is
the set $2X\cup(2X+1)$. Consider the identity numbering $\text{id}_{Y}$
on $Y$. 

Equipping $(Y,\nu)$ with the discrete distance, defined by $d(n,m)=1$
if $n\ne m$ and $d(n,m)=0$ if $n=m$, makes of it a $\text{CMS}^{*}$.
The open balls on $Y$ are just the singletons and $Y$, and thus
open balls do form a Lacombe basis, and the Lacombe topology generated
by open balls is just the discrete Lacombe topology. 

Consider the set $\{0,1\}$ equipped with the obvious numbering $\mu:n\mapsto n\mod2$,
and also equipped with the discrete distance. 

The function $f$ defined by 
\[
f:\begin{cases}
Y & \rightarrow\{0,1\}\\
n & \mapsto n\mod2
\end{cases}
\]
is $(\text{id}_{Y},\mu)$-computable, and it is effectively metric
continuous: a radius of $1/2$ always works to prove the effective
metric continuity. However, for the Lacombe topology generated by
open balls, the function $f$ is not effectively continuous. 
\end{example}

Thus on a metric space, the open balls do not always form a Lacombe
basis, and even when they do, the topology they generate is not always
the effective metric topology. However, an important result of Moschovakis
states that on countable sets that are effectively Polish, those two
topologies agree. This was quoted as Theorem \ref{thm:Moschovakis, INTRO Theorem}.
We present a proof of this theorem in the following section. 

\subsection{\label{subsec:Proof of Moschovakis-1}A proof of Moschovakis' Theorem
based on Markov's Lemma }

We revisit Moschovakis' theorem on Lacombe sets in the case of subnumberings
in Section \ref{part: Non Surjective Numberings}, see Theorem \ref{thm:Moschovakis-theorem-on Spreen bases}.
Note that the theorem below does not generalize to subnumberings if
we keep the Nogina definition as it is. In particular, the argument
which says that the \emph{a priori} arbitrary radii produced by the
program associated to a Nogina name will \emph{automatically} not
vanish near computable points fails for non-computable points, as
is proven by the existence of a Lacombe subset of $\mathbb{R}$ that
contains all computable reals but that is still different from $\mathbb{R}$
(this theorem is also known as the existence of a $\Pi_{1}^{0}$-class
which contains no computable reals, but which is non-empty, it is
in fact a Cantor space of non-computable points). 

We include here a new proof of the original theorem. Our purpose is
to highlight how existence of an algorithm that computes limits plays
a central role for this theorem, whereas it will not be needed anymore
in Theorem \ref{thm:Moschovakis-theorem-on Spreen bases}. 

Our proof is similar to the original proof of Moschovakis, but it
improves it on a couple of points. It has the advantage of showing
the existence of a ``universal dense sequence of names'': there
is a certain computable sequence $(w_{n})_{n\in\mathbb{N}}$ such
that, on input the Nogina name of an open set $O$, the sequence of
balls centered at $\nu(w_{n})$ and with radius given by the Nogina
name of $O$ exhausts $O$. The proof given in \cite{Moschovakis1964}
uses the given Nogina name of $O$ to computably construct a computable
sequence of names appropriate to $O$. Different sequences of names
are thus applied to different Nogina open sets, we show that this
is not necessary. 

Finally, this proof highlights how Markov's Lemma \cite[Theorem 4.2.2]{Markov1963}
plays a central role in this theorem. 

Markov's Lemma says: in a computable metric space\emph{$^{*}$} $(X,\nu,d)$
where limits can be computed, if $(x_{n})_{n\in\mathbb{N}}$ is a
computable sequence that converges to a computable point $y$, then
$\{y\}$ is not a $\nu$-semi-decidable subset of $\{y\}\cup\{x_{n},\,n\in\mathbb{N}\}$. 

The proof of Moschovakis' Theorem can then be seen as an analysis
of Markov's Lemma, explicitly writing down the set of $\nu$-names
that appear in the proof of Markov's Lemma, and showing that they
belong to a computably enumerable set of $\nu$-names of points of
$X$. 

Using work of Spreen \cite{Spr98}, Moschovakis' theorem can be generalized
to any topological space which admits a computable and dense sequence,
where limits can be computed, and where points have neighborhood bases
that consist of co-semi-decidable sets. Note that ``topological spaces
where limit of computably converging sequences can be computed''
are defined in \cite{Spr98} thanks to a formal inclusion relation.
\begin{thm}
[Moschovakis, \cite{Moschovakis1964}, Theorem 11]\label{thm:Moschovakis,-,-Theorem-1}
Let $(X,\nu,d)$ be a numbered set with a computable metric, where
limits of computably convergent computable sequences can be computed,
and which admits a dense and computable sequence. 

Then, the set of open balls with computable radii, equipped with its
natural numbering, forms both a Lacombe and a Nogina basis, and the
computable topologies generated by these bases are identical. 
\end{thm}

For $i$ a natural number, we denote by $\varphi_{i}(n)\uparrow^{t}$
the fact that the computation of $\varphi_{i}$ at $n$ does not stop
in $t$ steps. We write $\varphi_{i}(n)\downarrow^{t}$ to say that
this computation does halt in less than $t$ steps. 

And $\varphi_{i}(n)\uparrow$ means that $n\notin\text{dom}(\varphi_{i})$,
$\varphi_{i}(n)\downarrow$ means that $n\in\text{dom}(\varphi_{i})$. 
\begin{proof}
The only non-trivial part is to compute, from the name of a Nogina
open set, a Lacombe name for the same set. 

Fix a computable sequence of names, $(u_{n})_{n\in\mathbb{N}}\in\text{dom}(\nu)^{\mathbb{N}}$,
whose image by $\nu$ is dense. Consider a limit algorithm $\text{Lim}$
which on input the name of a sequence $(x_{n})_{n\in\mathbb{N}}$
that converges at exponential speed to a point $l$ of $X$ produces
a $\nu$-name of $l$. 

Fix a Nogina open $O$. It is associated to a $(\nu,c_{\mathbb{R}})$-computable
multi-function $C:O\rightrightarrows\mathbb{R}_{>0}$ which satisfies
that for any $x\in O$ and $r\in C(x)$, $B(x,r)\subseteq O$. Denote
by $\hat{C}:\mathbb{N}\rightharpoonup\mathbb{N}$ a computable realizer
for $C$. 

We must justify that computing $\hat{C}$ along a dense sequence will
permit to capture all of $O$, and thus that the radii produced by
$\hat{C}$ do not vanish near points of $O$. This is in fact false
if we do not modify $(u_{n})_{n\in\mathbb{N}}$ to obtain a richer
set of names. Note also that, because $C$ is a multi-function, we
cannot directly apply continuity results to it. 

We will in fact consider a new computable sequence $(w_{n})_{n\in\mathbb{N}}\in\text{dom}(\nu)^{\mathbb{N}}$
of names, that defines the same elements as $(u_{n})_{n\in\mathbb{N}}$
(i.e. $\{\nu(u_{n}),\,n\in\mathbb{N}\}=\{\nu(w_{n}),\,n\in\mathbb{N}\}$),
but which is much richer. It is obtained as follows. 

For $k$ and $n$ in $\mathbb{N}$, define the program $P_{n,k}$
by: 
\[
P_{n,k}(t)=\text{if }\varphi_{n}(n)\uparrow^{t},\,\text{output }u_{\varphi_{k}(t)},\text{ otherwise let \ensuremath{t_{0}}=\ensuremath{\min}(\ensuremath{t,\,\varphi_{n}(n)\downarrow^{t})}, and output }u_{\varphi_{k}(t_{0})}.
\]
The sequence $(P_{n,k}(t))_{t\in\mathbb{N}}$ thus depends of a run
of the computation of $\varphi_{n}(n)$. While this run lasts, $P_{n,k}$
enumerates the sequence $(u_{\varphi_{k}(t)})_{t\in\mathbb{N}}$.
If this run halts in $t_{0}$ steps, the sequence $P_{n,k}$ becomes
constant to $u_{\varphi_{k}(t_{0})}$.

We will consider a certain subsets of the set $\{P_{n,k},\,n,k\in\mathbb{N}\}$. 

First, we select out of the programs $P_{n,k}$ those for which $\varphi_{n}(n)\downarrow$.

Then, we select those for which $\varphi_{k}(t)\downarrow$ for each
$t$ less or equal to the number of steps it takes to the computation
of $\varphi_{n}(n)$ to halt. 

After this selection, each of the remaining $P_{n,k}$ defines successfully
an eventually constant sequence extracted from $(u_{n})_{n\in\mathbb{N}}$. 

Finally, we select those $P_{n,k}$ which define a sequence that converges
at exponential speed (faster than $n\mapsto2^{-n}$). 

It is easy to see that the first two selection steps defined above
correspond to semi-decidable properties. 

This also holds for the third one. Indeed, since the sequence produced
by $P_{n,k}$ is eventually constant, and because a rank after which
it is constant can be computed from $n$, finitely many strict inequalities
need to be checked to prove that it converges exponentially fast. 

The limit algorithm can thus be successfully applied to each of the
remaining $P_{n,k}$, we denote by $(w_{m})_{m\in\mathbb{N}}$ the
computable sequence obtained by dovetailing the double sequence $\text{Lim}(P_{n,k})$,
for $n$ and $k$ in the semi-decidable subset described above.

We then claim that 
\[
O=\bigcup_{\nu(w_{n})\in O}B(\nu(w_{n}),c_{\mathbb{R}}(\hat{C}(w_{n}))).
\]

This will imply that the code for $O$ as a Lacombe set can be computed
from the code for $O$ as a Nogina open set. Indeed, $O$ being semi-decidable,
the set $\{n\in\mathbb{N},\nu(w_{n})\in O\}$ is a semi-decidable
subset of $\mathbb{N}$, and a $W$-name for this set is exactly a
Lacombe name for $O$. ($W$ is the usual numbering of semi-decidable
subsets of $\mathbb{N}$.)

The proof from now on is simply an analysis of Markov's Lemma \cite[Theorem 4.2.2]{Markov1963}. 

Let $x$ be any point of $O$.

We show that $x$ belongs to $B(\nu(w_{n}),c_{\mathbb{R}}(\hat{C}(w_{n})))$
for some $n$. 

Note that if $x$ is an element of the dense sequence, i.e. $x\in\{\nu(w_{n}),\,n\in\mathbb{N}\}$,
we have nothing to do. So we can suppose that $x\notin\{\nu(w_{n}),\,n\in\mathbb{N}\}$. 

There exists a total computable function $g$ such that for all $n$
in $\mathbb{N}$, $d(\nu(u_{g(n)}),x)<2^{-n}$. It exists for the
simple reason that a brute search algorithm for a point of $(\nu(u_{n}))_{n\in\mathbb{N}}$
whose distance to $x$ is less than $2^{-n}$ will always terminate.
Denote by $k$ a code of $g$, $g=\varphi_{k}$. 

Consider now the program $P_{n,k}$ as defined above, for the chosen
$k$, and varying $n$. 

Recall that $P_{n,k}$ produces a sequence as follows: while a run
of $\varphi_{n}(n)$ lasts, it produces $(u_{\varphi_{k}(0)},u_{\varphi_{k}(1)},u_{\varphi_{k}(2)},...)$,
if this run halts in $t_{0}$ steps, it becomes constant: $(u_{\varphi_{k}(0)},u_{\varphi_{k}(1)},u_{\varphi_{k}(2)},...,u_{\varphi_{k}(t_{0}-1)},u_{\varphi_{k}(t_{0})},u_{\varphi_{k}(t_{0})},...)$.

If $\varphi_{n}(n)\downarrow$, $P_{n,k}$ produces a sequence eventually
constant, which ends at a point $u_{l}$. Note that $\nu(u_{l})\ne x$. 

If $\varphi_{n}(n)\uparrow$, $P_{n,k}$ produces exactly the sequence
$(u_{\varphi_{k}(t)})_{t\in\mathbb{N}}$, and note that $(\nu(u_{\varphi_{k}(t)}))_{t\in\mathbb{N}}$
converges to $x$. 

Thus applying the limit algorithm to (the code of) $P_{n,k}$, we
obtain a sequence of names $(z_{n})_{n\in\mathbb{N}}$ such that $\nu(z_{n})=x\iff\varphi_{n}(n)\uparrow$.
Note also that if $\varphi_{n}(n)\downarrow$, then $z_{n}$ is a
name appearing in the sequence $(w_{n})_{n\in\mathbb{N}}$. 

Finally, recall that the halting problem is unsolvable. 

It is thus impossible, given $z_{n}$, to determine whether or not
it is a $\nu$-name of $x$. 

Applying the multi-function $\hat{C}$ to $z_{n}$, if we had that
\[
\nu(z_{n})\ne x\implies x\notin B(\nu(z_{n}),\hat{C}(z_{n})),
\]
a solution to the halting problem would follow: we would have the
equivalences
\[
\nu(z_{n})\ne x\iff d(\nu(z_{n}),x)\ge\hat{C}(z_{n}),
\]
\[
\nu(z_{n})=x\iff d(\nu(z_{n}),x)\le\hat{C}(z_{n})/2.
\]
But choosing between these options is always a decidable problem (because
the following problem is decidable: given a computable real $t$ such
that either $t\ge a$, or $t\le a-\epsilon$, for some fixed $a$
and $\epsilon>0$, choose which option holds). 

This means that the implication 
\[
\nu(z_{n})\ne x\implies x\notin B(\nu(z_{n}),\hat{C}(z_{n}))
\]
 cannot hold. And thus that there exists $n$ such that $\nu(z_{n})\ne x$,
but such that $x$ belongs to $B(\nu(z_{n}),\hat{C}(z_{n}))$. As
the name $z_{n}$ appears in the rich sequence of names $(w_{n})_{n\in\mathbb{N}}$,
we indeed have shown that for any $x$ in $O$, $x$ belongs to $B(\nu(w_{n}),\hat{C}(w_{n}))$
for some $n$ in $\mathbb{N}$. 
\end{proof}

\subsection{Uniform effective continuity at every point}

In this paragraph, we detail the following phenomenon: a computable
function $f:X\rightarrow Y$ can be effectively continuous at every
point $x$ of $X$, and this uniformly in $x$, while still not being
effectively continuous. But this phenomenon cannot happen for topologies
that are defined thanks to Nogina bases.
\begin{thm}
\label{thm:unif cont at each point and eff cont}If $(X,\nu,\mathcal{T},\tau)$
is a Nogina topology, then any function from $X$ to a computable
topological space $(Y,\mu,\mathcal{T}_{Y},\tau_{Y})$ that is uniformly
effectively continuous at each point of $X$ is effectively continuous
on $X$. 

And this statement is uniform, there is an effective procedure that
witnesses for the above equivalence. 
\end{thm}

The proof is straightforward. 
\begin{proof}
Suppose that $(X,\nu,\mathcal{T},\tau)$ is a computable topological
space that comes from a Nogina basis $(\mathfrak{B},\beta)$. 

Suppose that a function $f:X\rightarrow Y$ is uniformly effectively
continuous at every point. We show that it is effectively continuous. 

Let $O$ be an effective open set in $Y$. We want to compute $f^{-1}(O)$
as an effective open set for $(X,\nu,\mathcal{T},\tau)$. 

The $\tau$-name of $f^{-1}(O)$ as a Nogina open is constituted of
an encoded pair $(n,m)$, where $n$ codes for the fact that $f^{-1}(O)$
is semi-decidable, and $m$ codes a function that, given a point in
$f^{-1}(O)$, produces the $\beta$-name of a set $B$ that contains
this point and that is contained in $f^{-1}(O)$. 

The reciprocal image of a semi-decidable set is always semi-decidable,
and this is uniform, so computing $n$ is straightforward. 

Computing $m$ follows directly from the hypothesis that $f$ is effectively
continuous at every point: given a point $x$ in $f^{-1}(O)$, it
is possible to compute an open set $A_{x}$ that contains $x$ and
that is contained in $f^{-1}(O)$, by effective continuity at $x$
with $O$ being the input neighborhood of $f(x)$. The code for the
map $x\mapsto A_{x}$ is precisely the required $m$ that gives the
code of $f^{-1}(O)$ as an effectively open set. 
\end{proof}
The result of this theorem is not valid in general. 
\begin{example}
\label{exa:unif cont at each point not eff cont }We use again Example
\ref{exa:MAIN Example }: $X\subseteq\mathbb{N}$ is a non-computably
enumerable set, and consider the set $Y=2X\cup(2X+1)$. Consider the
identity numbering $\text{id}_{Y}$ on $Y$, and the discrete Lacombe
topology on $Y$. 

Consider the Ershov topology on $\{0,1\}$ equipped with the obvious
numbering $\mu:n\mapsto n\mod2$. 

The function $f$ defined by 
\[
f:\begin{cases}
Y & \rightarrow\{0,1\}\\
n & \mapsto n\mod2
\end{cases}
\]
is obviously $(\text{id}_{Y},\mu)$-computable. 

It is uniformly effectively continuous at each point of $Y$: given
a point $y$ of $Y$, $\{y\}$ is an effectively open set for the
discrete Lacombe topology on $Y$, it is a neighborhood of $y$, and
of course $\{y\}\subseteq f^{-1}(f(y))$. 

However, $f$ is not effectively continuous, since $f^{-1}(0)=2X$
is not effectively open for the discrete Lacombe topology on $Y$. 
\end{example}

\section{\label{part: Non Surjective Numberings}Effective bases for computable
metrics: Spreen bases}

\subsection{\label{subsec:Spreen Bases}Spreen bases}

We consider a basis $(\mathfrak{B},\beta)$ equipped with a formal
inclusion relation, see Definition \ref{def:Formal inclusion relation Text }. 

If $\mathring{\subseteq}$ is a formal inclusion relation for $(\mathfrak{B},\beta)$,
we extend $\mathring{\subseteq}$ to allow us to compare a natural
number with a sequence: for $(u_{k})_{k\in\mathbb{N}}$ in $\text{dom}(\beta)^{\mathbb{N}}$,
define $n\mathring{\subseteq}(u_{k})_{k\in\mathbb{N}}$ iff $\exists k\in\mathbb{N},\,n\mathring{\subseteq}u_{k}$.
We also extend $\mathring{\subseteq}$ to allow us to compare pairs
of sequences: for sequences $(u_{k})_{k\in\mathbb{N}}$ and $(w_{k})_{k\in\mathbb{N}}$,
define $(u_{k})_{k\in\mathbb{N}}\mathring{\subseteq}(w_{k})_{k\in\mathbb{N}}$
iff $\forall p\in\mathbb{N},\,u_{p}\mathring{\subseteq}(w_{k})_{k\in\mathbb{N}}$.

Note that the extensions defined above remain reflexive and transitive
relations.

We denote by $\mathring{\sim}$ both the equivalence relation induced
by $\mathring{\subseteq}$ on $\mathbb{N}$, and that induced on $\text{dom}(\beta)^{\mathbb{N}}$:
\[
\forall n,m\in\text{dom}(\beta),\,n\mathring{\sim}m\iff n\mathring{\subseteq}m\text{ and }m\mathring{\subseteq}n;
\]
\begin{align*}
\forall(u_{k})_{k\in\mathbb{N}},(w_{k})_{k\in\mathbb{N}} & \in\text{dom}(\beta)^{\mathbb{N}},\\
 & \,(u_{k})_{k\in\mathbb{N}}\mathring{\sim}(w_{k})_{k\in\mathbb{N}}\iff(u_{k})_{k\in\mathbb{N}}\mathring{\subseteq}(w_{k})_{k\in\mathbb{N}}\text{ and }(w_{k})_{k\in\mathbb{N}}\mathring{\subseteq}(v_{k})_{k\in\mathbb{N}}.
\end{align*}

\begin{defn}
\label{def: Spreen Basis Text}A\emph{ Spreen topological basis for
$(X,\nu)$} is a triple $(\mathfrak{B},\beta,\mathring{\subseteq})$,
where $\mathfrak{B}$ is a subset of $P(X)$, $\beta$ is a numbering
of $\mathfrak{B}$, and $\mathring{\subseteq}$ is a formal inclusion
relation for $(\mathfrak{B},\beta)$, that satisfies the following
conditions:
\begin{enumerate}
\item \label{enu: Cond SD }The elements of the basis $(\mathfrak{B},\beta)$
are uniformly semi-decidable, i.e. $\beta\le\nu_{SD,m}$;
\item \label{enu:Cond Dense }Each non-empty basic open set contains a $\nu$-computable
point; 
\item \label{enu:Cond X eff open }There is a $(\text{id}_{\text{dom}(\nu)},\text{id}_{\text{dom}(\beta)}^{\text{id}_{\mathbb{N}}})$-computable
function $G_{1}:\text{dom}(\nu)\rightarrow\text{dom}(\beta)^{\mathbb{N}}$
that satisfies the following conditions: 
\[
\forall n\in\text{dom}(\nu),\,G_{1}(n)=(u_{k})_{k\in\mathbb{N}}\implies\forall k\in\mathbb{N},\,\nu(n)\in\beta(u_{k});
\]
\[
\boxed{\mathring{\sim}}\,\,\forall n,m\in\text{dom}(\nu),\,\nu(n)=\nu(m)\implies G_{1}(n)\mathring{\sim}G_{1}(m);
\]
\[
\circledcirc\,\,\forall x\in X,\,\exists b\in\text{dom}(\beta),\,x\in\beta(b)\,\&\,\forall n\in\text{dom}(\nu),\,\nu(n)\in\beta(b)\implies b\mathring{\subseteq}G_{1}(n).
\]
\item \label{enu:Cond Inter }There is a computable function $G_{2}:\text{dom}(\nu)\times\text{dom}(\beta)\times\text{dom}(\beta)\rightarrow\text{dom}(\beta)^{\mathbb{N}}$,
such that for any $n$ in $\text{dom}(\nu)$, $b_{1}$ and $b_{2}$
in $\text{dom}(\beta)$, with $\nu(n)\in\beta(b_{1})\cap\beta(b_{2})$,
we have $G_{2}(n,b_{1},b_{2})=(u_{k})_{k\in\mathbb{N}}$, with: 
\[
\forall k\in\mathbb{N},\,\nu(n)\in\beta(u_{k});
\]
\[
\forall k\in\mathbb{N},\,\beta(u_{k})\subseteq\beta(b_{1})\,\&\,\beta(u_{k})\subseteq\beta(b_{2});
\]
\begin{align*}
\boxed{\mathring{\sim}}\,\,\forall n,m & \in\text{dom}(\nu),\,\forall b_{1},b_{2},\hat{b}_{1},\hat{b}_{2}\in\text{dom (\ensuremath{\beta)}},\\
 & \,\nu(n)=\nu(m)\,\&\,b_{1}\mathring{\subseteq}\hat{b}_{1}\,\&\,b_{2}\mathring{\subseteq}\hat{b}_{2}\implies G_{2}(n,b_{1},b_{2})\mathring{\subseteq}G_{2}(m,\hat{b}_{1},\hat{b}_{2});
\end{align*}
\begin{align*}
 & \circledcirc\,\,\forall b_{1},b_{2}\in\text{dom}(\beta),\,\forall x\in X,\,x\in\beta(n_{1})\cap\beta(n_{2})\\
 & \implies\exists b_{3}\in\text{dom}(\beta),\,x\in\beta(b_{3})\,\&\,\forall m\in\text{dom}(\nu),\,\nu(m)\in\beta(b_{3})\implies b_{3}\mathring{\subseteq}G_{2}(m,b_{1},b_{2}).
\end{align*}
\end{enumerate}
\end{defn}

This ends this definition. 

Conditions written with the symbol $\circledcirc$ are here to prevent
that the basic open sets produced by different functions be needlessly
small, they are non-vanishing requirements. Notice that they always
contain a quantification that concerns all points of $X$, not only
the computable ones. 

Conditions written with the symbol $\boxed{\mathring{\sim}}$ are
conditions that impose that the results of computations should not
depend on a given name, up to formal inclusion, and should be increasing
with respect to formal inclusion. 
\begin{itemize}
\item Condition (\ref{enu:Cond Dense }) ensures that computable points
are dense, and that the conditions denoted $\circledcirc$ that follow
are not vacuous, without it, statements such as ``$\forall n\in\text{dom}(\nu),\,\nu(n)\in\beta(b)\implies b\mathring{\subseteq}G_{1}(n)$''
could be true only because $\beta(b)\cap X_{\nu}=\emptyset$. 
\item In (\ref{enu:Cond X eff open }), we ensure that $X$ is covered by
basic open sets, the function $G_{1}$ ensures that $X$ itself is
an effective open set.
\item In (\ref{enu:Cond Inter }), given a point $x$ in an intersection
$B_{1}\cap B_{2}$, we give a sequence of basic open sets that contain
$x$ and are contained in $B_{1}\cap B_{2}$, and we must ensure that
the sets obtained this way cover $B_{1}\cap B_{2}$. The function
$G_{2}$ ensures that intersection of open sets can be computed. 
\item Note that the condition $\circledcirc$ for $G_{1}$ has the following
obvious consequence: 
\[
X=\underset{n\in\text{dom}(\nu)\,}{\bigcup}\underset{\,G_{1}(n)=(u_{k})_{k\in\mathbb{N}}}{\bigcup}\beta(u_{k}).
\]
\item Similarly, condition $\circledcirc$ for $G_{2}$ has the following
obvious consequence: for any $b_{1}$ and $b_{2}$ in $\text{dom}(\beta)$,
we have 
\[
\beta(b_{1})\cap\beta(b_{2})=\underset{\{n\,\vert\,\nu(n)\in\beta(n_{1})\cap\beta(n_{2})\}}{\bigcup}\underset{\,G_{2}(n,b_{1},b_{2})=(u_{k})_{k\in\mathbb{N}}}{\bigcup}\beta(u_{k}).
\]
\end{itemize}
We will now explicitly define the computable topology associated to
this basis. 
\begin{defn}
\label{def: Spreen Top text }Let $(X,\nu,\mathcal{B},\beta,\mathring{\subseteq})$
be a subnumbered set equipped with a Spreen basis. The\emph{ Spreen
topology associated to the Spreen basis $(\mathcal{B},\beta,\mathring{\subseteq})$
}is a triple $(\mathcal{T},\tau,\mathring{\subseteq})$ defined as
follows:
\begin{enumerate}
\item The topology $\mathcal{T}$ is the topology on $X$ generated by $\mathcal{B}$
(as a classical basis);
\item The subnumbering $\tau$ is defined as follows. Define $\tau(\langle n,m\rangle)=O$
if and only if: 
\begin{enumerate}
\item $n$ is a $\nu_{SD}$-name for $O\cap X_{\nu}$; 
\item $m$ codes a computable function $F:\text{dom}(\nu)\rightharpoonup\text{dom}(\beta)^{\mathbb{N}}$,
defined at least on $\{n\in\text{dom}(\nu)\,\vert\,\nu(n)\in O\}$,
such that:
\[
\forall n\in\text{dom}(\nu),\,\nu(n)\in O\,\&\,F(n)=(u_{k})_{k\in\mathbb{N}}\implies\forall k\in\mathbb{N},\,\nu(n)\in\beta(u_{k})\,\&\,\beta(u_{k})\subseteq O;
\]
\[
\boxed{\mathring{\sim}}\,\,\forall n,m\in\text{dom}(F),\,\nu(n)=\nu(m)\implies F(n)\mathring{\sim}F(m);
\]
\begin{align*}
\circledcirc\,\,\forall x & \in O,\,\exists b\in\text{dom}(\beta),\,x\in\beta(b)\,\&\,\\
 & \forall n\in\text{dom}(\nu),\,\nu(n)\in\beta(b)\implies b\mathring{\subseteq}F(n).
\end{align*}
\end{enumerate}
\item We extend the formal inclusion $\mathring{\subseteq}$ to names of
open sets: if $O_{1}$ is given by a name $n_{1}$ that encodes a
pair $(O_{1}\cap X_{\nu},F_{1})$, and $O_{2}$ by a name $n_{2}$
that encodes a pair $(O_{2}\cap X_{\nu},F_{2})$, we put 
\[
n_{1}\mathring{\subseteq}n_{2}\iff(\forall p\in\text{dom}(\nu),\,\nu(p)\in O_{1}\implies\nu(p)\in O_{2}\,\&\,F_{1}(p)\mathring{\subseteq}F_{2}(p)).
\]
\end{enumerate}
\end{defn}

\vspace{0.3cm}

Note that an easy consequence of the $\circledcirc$ condition is
that the function $F$ associated to $O$ does allow to reconstruct
the whole of $O$ starting from its computable points: 
\[
O=\underset{\{n,\,\nu(n)\in O\}}{\bigcup}\underset{F(n)=(u_{k})_{k\in\mathbb{N}}}{\bigcup}\beta(u_{k}).
\]
(We even could change the union that appears above: instead of taking
a union on all possible names for points of $O\cap X_{\nu}$, any
choice of one name for each such point will also permit to rebuild
$O$.)
\begin{thm}
\label{thm: Main THM text }The Spreen topology associated to a Spreen
basis is a computable topology. 
\end{thm}

\begin{proof}
On the subnumbered set $(X,\nu)$, denote by $(\mathfrak{B},\beta,\mathring{\subseteq})$
a Spreen basis, and by $(\mathcal{T},\tau,\mathring{\subseteq})$
the Spreen topology it generates. 

A simple check shows that the relation $\mathring{\subseteq}$ as
extended to names of open sets is indeed a formal inclusion relation
for $\tau$ (reflexivity and transitivity are immediate, and so is
the fact that $n_{1}\mathring{\subseteq}n_{2}\implies\tau(n_{1})\subseteq\tau(n_{2})$). 

We prove that $(\mathcal{T},\tau)$ is a computable topology, we follow
points (1)-(2)-(3)-(4-a) and (4-b) of Definition \ref{def:MAIN DEF}. 

We must first prove (1): the image of $\tau$ generates $\mathcal{T}$. 

We defined $\mathcal{T}$ to be the topology generated by $\mathfrak{B}$.
It thus suffices to show that each element of $\mathfrak{B}$ is in
the image of $\tau$ to prove that the image of $\tau$ indeed generates
$\mathcal{T}$. To check that basic open sets are indeed effectively
open sets, note the following: the basic open sets are semi-decidable,
and for a basic open set $B=\beta(b)$, the function $F$ which maps
any computable point of $B$ to the constant sequence $(b,\,b,\,...)$
clearly satisfies the conditions that appear in Definition \ref{def: Spreen Top text },
thus ensuring that $B$ is in the image of $\tau$. (Condition $\boxed{\mathring{\sim}}$
is immediate since $F$ is a constant function, and the number $b$
itself is a witness for the fact that $F$ satisfies the $\circledcirc$
condition: for any $x=\nu(n)$ in $B$, $b\mathring{\subseteq}F(n)=(b,b,b,...)$). 

We then prove (2): the empty set and $X$ are effective open sets
for $\tau$. 

This is straightforward for the empty set. The fact that $X$ is in
the image of $\tau$ follows directly from the third point of Definition
\ref{def: Spreen Basis Text}, the function $G_{1}$ introduced here
precisely guarantees that $X$ is effectively open for $\tau$. 

We now prove (3): Malcev's condition is verified. 

Recall that this means that given the $\tau$-name for an open set
$O$, it is possible to construct the $\nu_{SD}$-name of $O\cap X_{\nu}$,
or again that $\tau\le\nu_{SD,m}$. But the $\tau$-name for an effective
open set $O$ is a pair whose first element is precisely the $\nu_{SD}$
code for $O\cap X_{\nu}$, as ensured by condition (2-a) in Definition
\ref{def: Spreen Top text }. 

We now prove (4-a): unions are computable for $\tau$. 

A $\tau$-computable sequence $(O_{i})_{i\in\mathbb{N}}$ is described
by a sequence of pairs $(A_{i},F_{i})_{i\in\mathbb{N}}$, where the
sets $A_{i}=O_{i}\cap X_{\nu}$ are given by $\nu_{SD}$-names, and
where $F_{i}$ is a function that, given the name of a point in $O_{i}$,
produces a sequence of names for basic open sets in $O_{i}$ that
contain $x$. 

A $\tau$-name for the union $\underset{i\in\mathbb{N}}{\bigcup}O_{i}$
is obtained as follows. It will be given by a pair $(A_{\infty},F_{\infty})$,
where $A_{\infty}$ can just be taken as the union of the $A_{i}$,
whose $\nu_{SD}$-name can be computed. The function $F_{\infty}$
is obtained by dovetailing: given the $\nu$-name $n$ of a point
$x$ in $\underset{i\in\mathbb{N}}{\bigcup}O_{i}$, to compute $F_{\infty}(n)$,
one searches in parallel for each $i\in\mathbb{N}$ whether $x$ belongs
to $A_{i}$ (because they are semi-decidable), and, when it does,
it enumerates the sequence given by the computation of $F_{i}(n)$. 

We must check that the pair $(A_{\infty},F_{\infty})$ thus defined
indeed constitutes a $\tau$-name for $O_{\infty}=\underset{i\in\mathbb{N}}{\bigcup}O_{i}$,
i.e. that the conditions that appear in (2-b) of Definition \ref{def: Spreen Top text }
are satisfied. 

The fact that, if $x=\nu(n)$, $F_{\infty}(n)$ produces only basic
open sets that contain $x$ and are contained in $O_{\infty}$ is
straightforward. 

We now show that the function $F_{\infty}$ depends on a given name
only up to formal inclusion (condition $\boxed{\mathring{\sim}}$).
This follows directly from the fact that each $F_{i}$ satisfies this
property, and that the arbitrary permutation that can occur during
the dovetailing operation does not affect the relation $\mathring{\subseteq}$
as defined on sequences. 

We prove the $\circledcirc$ condition. Let $x$ be any point of $O_{\infty}$.
Then $x$ must belong to some $O_{i}$, and since $(A_{i},F_{i})$
is a $\tau$-name for $O_{i}$, the $\circledcirc$ condition is satisfied
for $F_{i}$: there must be $b$ in $\text{dom}(\beta)$ such that
$\beta(b)$ contains $x$, and such that for any $y=\nu(n)$ in $\beta(b)\cap X_{\nu}$,
we have $b\mathring{\subseteq}F_{i}(n)$. By construction, the sequence
$F_{i}(n)$ is formally included in the sequence $F_{\infty}(n)$
(as $F_{i}(n)$ appears as a subsequence of $F_{\infty}(n)$), and
thus we have $b\mathring{\subseteq}F_{i}(n)\mathring{\subseteq}F_{\infty}(n)$,
which shows that $b\mathring{\subseteq}F_{\infty}(n)$. 

We finally have to prove that computing unions is indeed increasing
and expanding with respect to $\mathring{\subseteq}$. In each case,
this follows immediately from the fact that for each $i$ in the union
$O_{\infty}=\underset{i\in\mathbb{N}}{\bigcup}O_{i}$, if $O_{i}$
was described by a pair $(A_{i},F_{i})$, and if $x=\nu(n)$ is a
point of $O_{i}$, then $F_{i}(n)$ is a subsequence of $F_{\infty}(n)$. 

This ends the proof of (4-a). 

Finally, we prove (4-b): finite intersections are computable. 

Suppose we are given the $\tau$-names of two sets $O_{1}$ and $O_{2}$,
described respectively by pairs $(A_{1},F_{1})$ and $(A_{2},F_{2})$.
Write $O_{3}=O_{1}\cap O_{2}$. A $\nu_{SD}$-name for the intersection
$O_{3}\cap X_{\nu}$, which we denote $A_{3}$, can easily be obtained
by computing the intersection of the semi-decidable sets $A_{1}$
and $A_{2}$.

We now describe a function $F_{3}:\text{dom}(\nu)\rightharpoonup\text{dom}(\beta)^{\mathbb{N}}$,
defined at least on $\nu^{-1}(A_{3})$, associated to $O_{3}$. 

Given the $\nu$-name $n$ of a point $x$ in $A_{3}$, compute $F_{1}(n)$,
and $F_{2}(n)$. This yields two sequences $(u_{k}^{1})_{k\in\mathbb{N}}$
and $(u_{k}^{2})_{k\in\mathbb{N}}$. 

We then apply the function $G_{2}$ given in the definition of a Spreen
basis to each pair $(u_{i}^{1},u_{j}^{2})$ and to $n$, this yields
a sequence $(t_{k}^{i,j})_{k\in\mathbb{N}}$, define the result of
the computation of $F_{3}$ on $x=\nu(n)$ to be a sequence obtained
by dovetailing all the sequences $(t_{k}^{i,j})_{k\in\mathbb{N}}$
obtained by the computations of $G_{2}(n,u_{i}^{1},u_{j}^{2})$, for
$i,j\in\mathbb{N}$. 

It is clear that the sequence obtained this way consists only of $\beta$-names
of basic open sets that contain $x$ and that are contained in $O_{3}=O_{1}\cap O_{2}$. 

We now need to check that the function $F_{3}$ defined this way satisfies
the conditions denoted $\boxed{\mathring{\sim}}$ and $\circledcirc$. 

Condition $\boxed{\mathring{\sim}}$ asks that the computation of
$F_{3}(n)$ be independent of a given name of $\nu(n)=x$, modulo
formal inclusion. This follows from the fact that $F_{1}$, $F_{2}$
and $G_{2}$ all satisfy this condition, and that $G_{2}$ is increasing
for $\mathring{\subseteq}$. We show this now. 

Let $n$ and $m$ be two names of $x$. Computing $F_{1}(n)$ and
$F_{1}(m)$ yields two $\mathring{\sim}-$equivalent sequences $(u_{k}^{1,n})_{k\in\mathbb{N}}$
and $(u_{k}^{1,m})_{k\in\mathbb{N}}$. Similarly for $F_{2}$, we
obtain $(u_{k}^{2,n})_{k\in\mathbb{N}}$ and $(u_{k}^{2,m})_{k\in\mathbb{N}}$. 

Then, for any $i$ and $j$, there are some $\hat{i}$ and $\hat{j}$
such that $u_{i}^{1,n}\mathring{\subseteq}u_{\hat{i}}^{1,m}$ and
$u_{j}^{2,n}\mathring{\subseteq}u_{\hat{j}}^{2,m}$. Then, by the
hypothesis $\boxed{\mathring{\sim}}$ on $G_{2}$, we have $G_{2}(n,u_{i}^{1,n},u_{j}^{2,n})\mathring{\subseteq}G_{2}(m,u_{\hat{i}}^{1,m},u_{\hat{j}}^{2,m})$.
This shows that $F_{3}(n)\mathring{\subseteq}F_{3}(m)$, and the other
direction follows by symmetry. 

This ends the proof for the $\boxed{\mathring{\sim}}$ property. 

We show that the $\circledcirc$ property holds:
\begin{align*}
\circledcirc\,\,\forall x & \in O_{3},\,\exists b\in\text{dom}(\beta),\,x\in\beta(b)\,\&\,\\
 & \forall n\in\text{dom}(\nu),\,\nu(n)\in\beta(b)\implies b\mathring{\subseteq}F_{3}(n).
\end{align*}
Let $x$ be as above. Apply the $\circledcirc$ property for $F_{1}$
and $F_{2}$ to obtain $b_{1}$ and $b_{2}$ that satisfy the above
statement but with respect to $F_{1}$ and $F_{2}$. 

We then apply the $\circledcirc$ property of $G_{2}$ to $x$, $b_{1}$
and $b_{2}$: this yields $b_{3}$ such that $x\in\beta(b_{3})$,
and for any computable point $y=\nu(n)$ in $\beta(b_{3})$, we have
$b_{3}\mathring{\subseteq}G_{2}(n,b_{1},b_{2})$. 

We will then show that $b_{3}$ indeed witnesses that $\circledcirc$
is satisfied for $x$ and $F_{3}$. 

For any computable point $y=\nu(n)$ in $\beta(b_{3})$, denote $F_{1}(n)=(u_{k}^{1})_{k\in\mathbb{N}}$.
We have that $b_{1}\mathring{\subseteq}(u_{k}^{1})_{k\in\mathbb{N}}$,
i.e. $b_{1}\mathring{\subseteq}u_{k_{1}}^{1}$ for some number $k_{1}$.
Similarly, for $F_{2}(n)=(u_{k}^{2})_{k\in\mathbb{N}}$, we have $b_{2}\mathring{\subseteq}(u_{k}^{2})_{k\in\mathbb{N}}$,
i.e. $b_{2}\mathring{\subseteq}u_{k_{2}}^{2}$ for some natural number
$k_{2}$. Then, computing $F_{3}(n)$, we build a sequence which contains
as a subsequence the result of $G_{2}(n,u_{k_{1}}^{1},u_{k_{2}}^{2})$.
But $b_{3}\mathring{\subseteq}G_{2}(n,b_{1},b_{2})$ by hypothesis,
and $G_{2}(n,b_{1},b_{2})\mathring{\subseteq}G_{2}(n,u_{k_{1}}^{1},u_{k_{2}}^{2})$
because $G_{2}$ is increasing for $\mathring{\subseteq}$. As the
relation $\mathring{\subseteq}$ is transitive, we obtain $b_{3}\mathring{\subseteq}F_{3}(n)$.

We finally have to check that computing intersections is indeed increasing
with respect to $\mathring{\subseteq}$. But this follows directly
from the corresponding property of the function $G_{2}$.

This finishes the proof. 
\end{proof}

\subsection{\label{subsec:New notion of equivalence of basis }New notion of
equivalence of basis }

We already gave two notions of equivalence of basis for Nogina and
Lacombe bases in Section \ref{subsec:Equivalence-of-bases}. We give
the corresponding definition for Spreen bases. Let $(X,\nu)$ be a
subnumbered set.
\begin{defn}
\label{def:equiv Spreen bases}A Spreen basis $(\mathcal{B}_{1},\beta_{1},\mathring{\subseteq}_{1})$
is equivalent to another one $(\mathcal{B}_{2},\beta_{2},\mathring{\subseteq}_{2})$
if there are two computable functions $f_{12}:\text{dom}(\nu)\times\text{dom}(\beta_{1})\rightarrow\text{dom}(\beta_{2})^{\mathbb{N}}$
and $f_{21}:\text{dom}(\nu)\times\text{dom}(\beta_{2})\rightarrow\text{dom}(\beta_{1})^{\mathbb{N}}$
that satisfy the following conditions: 

For $f_{12}$:

\[
\forall n\in\text{dom}(\nu),\,\forall b\in\text{dom}(\beta_{1}),\,f_{12}(n,b)=(b_{i})_{i\in\mathbb{N}}\,\implies\,\forall i\ge0,\,\nu(n)\in\beta_{2}(b_{i})\,\&\,\beta_{2}(b_{i})\subseteq\beta_{1}(b);
\]
\begin{align*}
\boxed{\mathring{\sim}}\,\,\forall n_{1},n_{2} & \in\text{dom}(\nu),\,\forall b_{1},b_{2}\in\text{dom}(\beta_{1}),\,\nu(n_{1})=\nu(n_{2})\,\&\,b_{1}\mathring{\subseteq}_{1}b_{2}\\
 & \implies f_{12}(n_{1},b_{1})\mathring{\subseteq}_{2}f_{12}(n_{2},b_{2});
\end{align*}
\begin{align*}
 & \circledcirc\,\forall x\in X,\,\forall b_{1}\in\text{dom}(\beta_{1}),\,x\in\beta_{1}(b_{1})\implies\exists b_{2}\in\text{dom}(\beta_{2}),\\
 & x\in\beta_{2}(b_{2})\,\&\,\forall n\in\text{dom}(\nu),\,\nu(n)\in\beta_{2}(b_{2})\implies b_{2}\mathring{\subseteq}_{2}f_{12}(n,b_{1}).
\end{align*}

The function $f_{21}$ should satisfy symmetric conditions. 
\end{defn}

The usefulness of this definition lies in the following proposition: 
\begin{prop}
\label{prop:Equiv Basis Spreen}The Spreen topologies generated by
equivalent Spreen bases on $(X,\nu)$ are identical. In this case,
it means that the identity function $\text{id}_{X}$, seen as going
from $X$ equipped with one topology to $X$ equipped with the other
one, is effectively continuous in both directions. 
\end{prop}

\begin{proof}
Consider equivalent bases $(\mathcal{B}_{1},\beta_{1},\mathring{\subseteq}_{1})$
and $(\mathcal{B}_{2},\beta_{2},\mathring{\subseteq}_{2})$ that generate
Spreen topologies $(\mathcal{T}_{1},\tau_{1},\mathring{\subseteq}_{1})$
and $(\mathcal{T}_{2},\tau_{2},\mathring{\subseteq}_{2})$. Consider
functions $f_{12}$ and $f_{21}$ as in Definition \ref{def:equiv Spreen bases}.
By symmetry, we only have to prove that $(\mathcal{T}_{1},\tau_{1},\mathring{\subseteq}_{1})$
is effectively finer than $(\mathcal{T}_{2},\tau_{2},\mathring{\subseteq}_{2})$.
Consider an effective open $O$ given by a $\tau_{2}$ name that encodes
a pair $(A,F_{2})$. We compute a $\tau_{1}$ name for $O$. 

As always, the code of $A=O\cap X_{\nu}$ as a semi-decidable set
can be taken as it is to compute a $\tau_{1}$ name of $O$. 

The function $F_{2}$ maps points to sequences of basic open sets,
by currying we can consider that it is a function of two variables
$F_{2}:\nu^{-1}(A)\times\mathbb{N}\rightarrow\text{dom}(\beta_{2})$.
Similarly, we define a function $F_{1}$ as a function of two variables:
$F_{1}:\nu^{-1}(A)\times\mathbb{N}\rightarrow\text{dom}(\beta_{1})$.
Put 
\[
F_{1}(n,p)=f_{21}(n,F_{2}(n,p)).
\]
Because $f_{21}$ respects the formal inclusion relation, it is easy
to see that $F_{1}$ satisfies all the required criterions of Definition
\ref{def: Spreen Basis Text}. And the formula above indeed respect
the formal inclusion: if $F_{2}$ is replaced by $\hat{F}_{2}$, with
$F_{2}(n)\mathring{\subseteq}\hat{F}_{2}(n)$ for all $n\in\nu^{-1}(A)$,
then the formula above gives $\hat{F}_{1}$ with $F_{1}(n)\mathring{\subseteq}\hat{F}_{1}(n)$
for all $n\in\nu^{-1}(A)$. 

The remaining details are easy to check. 
\end{proof}

\subsection{\label{subsec:Indep Numbering basis Spreen}Dependence on the numbering
of the basis}

We will now show that being a Spreen basis is independent of a choice
of a numbering for the basis, but not up to equivalence of numbering:
up to equivalence \emph{that respects the formal inclusion relation}. 

Let $(X,\nu)$ be a subnumbered set, and $(\mathfrak{B},\beta_{1},\mathring{\subseteq})$
a Spreen basis. Let $\beta_{2}$ be another numbering of $\mathfrak{B}$,
equivalent to $\mathfrak{B}$. We thus have a formal inclusion relation
$\mathring{\subseteq}$ defined on $\text{dom}(\beta_{1})$, but a
priori no formal inclusion relation of $\text{dom}(\beta_{2})$. Recall
that $\mathring{\subseteq}$ induces an equivalence relation $\mathring{\sim}$. 
\begin{defn}
\label{def:Equivalence of numbering preserves strong inclusion }We
say that the equivalence $\beta_{1}\equiv\beta_{2}$ \emph{preserves
the formal inclusion relation} \emph{of $\text{dom}(\beta_{1})$ }if
there are two computable functions $f_{12}:\text{dom}(\beta_{1})\rightarrow\text{dom}(\beta_{2})$
and $f_{21}:\text{dom}(\beta_{2})\rightarrow\text{dom}(\beta_{1})$
such that 
\[
\forall b\in\text{dom}(\beta_{1}),\,\beta_{1}(b)=\beta_{2}(f_{12}(b)),
\]
\[
\forall b\in\text{dom}(\beta_{2}),\,\beta_{2}(b)=\beta_{1}(f_{21}(b)),
\]
\[
\forall b\in\text{dom}(\beta_{1}),\,b\mathring{\sim}f_{21}(f_{12}(b)).
\]
\end{defn}

The first and second conditions just state that $f_{12}$ and $f_{21}$
witness for the fact that $\beta_{1}\equiv\beta_{2}$, but we ask
more: starting from the $\beta_{1}$-name of a basic set $B$, we
can turn it into a $\beta_{2}$-name for $B$ thanks to $f_{12}$,
and then turn it back into a $\beta_{1}$ name of $B$ via $f_{21}$.
We thus obtain two names for the same basic open set, we ask that
they be equivalent for the formal inclusion of $(\mathfrak{B},\beta_{1})$. 
\begin{example}
Suppose that we have a computable metric space\emph{$^{*}$} $(X,\nu,d)$.
There is a numbering $\beta_{1}$ of open balls of $X$ defined by
$\beta_{1}(\langle n,m\rangle)=B(\nu(n),c_{\mathbb{R}}(m))$. 

As usual we define a formal inclusion for $\beta_{1}$ by 
\[
\langle n_{1},m_{1}\rangle\mathring{\subseteq}\langle n_{2},m_{2}\rangle\iff d(\nu(n_{1}),\nu(n_{2}))+c_{\mathbb{R}}(m_{1})\le c_{\mathbb{R}}(m_{2}).
\]

If $\beta_{2}$ is another numbering of the open balls of $(X,d)$,
obtained in a similar fashion, but defined thanks to subnumberings
$\nu':\mathbb{N}\rightharpoonup X$ and $c'_{\mathbb{R}}:\mathbb{N}\rightharpoonup\mathbb{R}$,
which are $\equiv$-equivalent respectively to $\nu$ and $c_{\mathbb{R}}$,
then it is easy to see that $\beta_{1}$ and $\beta_{2}$ are equivalent,
and that this equivalence is witnessed by a function that preserves
the formal inclusion of $\beta_{1}$. 
\end{example}

We now give an example of functions witnessing for an equivalence
of numbering that do not preserve the formal inclusion relation. 
\begin{example}
We consider the numbering $\beta_{1}$ defined in the example above. 

Suppose that for some point $x_{0}=\nu(n)$ of $X_{\nu}$, we have
$B(x_{0},1)=B(x_{0},2)$. (It is obviously necessary to have a hypothesis
that guarantees that the formal inclusion relation is not equal to
the inclusion relation: $n\mathring{\subseteq}m\nLeftrightarrow\beta(n)\subseteq\beta(m)$.)
Denote by $m_{1}$ a $c_{\mathbb{R}}$-name of $1$ and by $m_{2}$
a $c_{\mathbb{R}}$-name of $2$. We define a function $f_{21}:\text{dom}(\beta_{1})\rightarrow\text{dom}(\beta_{1})$
by $f_{21}(\langle n,m_{1}\rangle)=\langle n,m_{2}\rangle$, and by
$f_{21}(u)=u$ elsewhere. 

Then $f_{21}$ witnesses for the fact that $\beta_{1}\le\beta_{1}$:
it turns the $\beta_{1}$-name of a ball into another name of that
same ball. Take $f_{12}$ to be the identity. Then these two functions
together witness for the trivial fact that $\beta_{1}\equiv\beta_{1}$,
but $f_{21}\circ f_{12}$ does not preserve the formal inclusion relation,
and its induced equivalence relation $\mathring{\sim}$: $f_{21}\circ f_{12}$
will indeed map a name of $B(x_{0},1)$ to another name of $B(x_{0},1)$,
but one where the radius explicitly given is $2$, thus different
from the original radius. 

This situation is precisely the one we avoid thanks to Definition
\ref{def:Equivalence of numbering preserves strong inclusion }. 
\end{example}

We now prove the following proposition:
\begin{prop}
Let $(\mathfrak{B},\beta_{1},\mathring{\subseteq}_{1})$ be a Spreen
basis for $(X,\nu)$, and $\beta_{2}$ another numbering of $\mathfrak{B}$,
equivalent to $\beta_{1}$, and suppose that this equivalence is compatible
with the formal inclusion relation. Then it is possible to define
a formal inclusion $\mathring{\subseteq}_{2}$ on $\text{dom}(\beta_{2})\times\text{dom}(\beta_{2})$,
such that $(\mathfrak{B},\beta_{2},\mathring{\subseteq}_{2})$ is
also a Spreen basis for $(X,\nu)$, and the Spreen topologies defined
by these bases are identical. 
\end{prop}

\begin{proof}
Consider the two functions $f_{12}$ and $f_{21}$ that witness for
$\beta_{1}\equiv\beta_{2}$, and that respect the formal inclusion
relation $\mathring{\subseteq}_{1}$ of $\text{dom}(\beta_{1})$.
We define a formal inclusion relation $\mathring{\subseteq}_{2}$
on $\text{dom}(\beta_{2})$ by 
\[
b_{1}\mathring{\subseteq}_{2}b_{2}\iff f_{21}(b_{1})\mathring{\subseteq}_{1}f_{21}(b_{2}).
\]

This is indeed a formal inclusion relation: it is transitive, reflexive,
and we do have the condition 
\[
b_{1}\mathring{\subseteq}_{2}b_{2}\implies\beta_{2}(b_{1})\subseteq\beta_{2}(b_{2}).
\]

We now show that $(\mathfrak{B},\beta_{2},\mathring{\subseteq}_{2})$
is a Spreen basis. We thus check conditions (\ref{enu: Cond SD }),
(\ref{enu:Cond Dense }), (\ref{enu:Cond X eff open }) and (\ref{enu:Cond Inter })
of Definition \ref{def: Spreen Basis Text}. 

(\ref{enu: Cond SD }) The fact that $\beta_{2}\le\nu_{SD,m}$ follows
directly from $\beta_{2}\equiv\beta_{1}$.

(\ref{enu:Cond Dense }) Computable points are dense: this follows
immediately as well. 

(\ref{enu:Cond X eff open }) We define a function $G_{1}^{2}$ that
witnesses that $X$ is effectively open for the topology generated
by $(\mathfrak{B},\beta_{2})$. We rely on the corresponding function
$G_{1}^{1}$ associated to the basis $(\mathfrak{B},\beta_{1},\mathring{\subseteq}_{1})$. 

Define $G_{1}^{2}$ as follows. Given the name $n$ of a point $x$,
compute $G_{1}^{1}(n)$. This gives a $\beta_{1}$-computable sequence
of basic open sets around $x$. The result of $G_{1}^{2}(n)$ is the
image of this sequence under $f_{12}$. In other words we set $G_{1}^{2}=f_{12}\circ G_{1}^{1}$.
We check conditions $\boxed{\mathring{\sim}}$ and $\circledcirc$. 

The computation of $G_{1}^{2}(n)$ does not depend on the given name
of $x$ up to formal inclusion. Indeed, if $\nu(n')=x$, then $G_{1}^{1}(n)\mathring{\sim}_{1}G_{1}^{1}(n')$.
Because $f_{21}\circ f_{12}$ preserves $\mathring{\subseteq}_{1}$,
we have $f_{21}\circ f_{12}\circ G_{1}^{1}(n)\mathring{\sim}_{1}f_{21}\circ f_{12}\circ G_{1}^{1}(n')$.
But, by definition of $\mathring{\subseteq}_{2}$, this is exactly
$f_{12}\circ G_{1}^{1}(n)\,\mathring{\sim}_{2}\,f_{12}\circ G_{1}^{1}(n')$,
which was what was to be proven. 

We now prove the non-vanishing condition denoted $\circledcirc$.
Given $x$ in $X$, we apply the non-vanishing condition for $G_{1}^{1}$
at $x$. This gives $b\in\text{dom}(\beta_{1})$ such that for any
$y=\nu(n)$ in $\beta_{1}(b)\cap X_{\nu}$, $b\mathring{\subseteq}_{1}G_{1}^{1}(n)$. 

Then $f_{12}(b)$ will show that $\circledcirc$ is satisfied for
$G_{1}^{2}$ at $x$. Indeed, given $y=\nu(n)$ in $\beta_{2}(b)\cap X_{\nu}$,
we have that $G_{1}^{2}(n)=f_{12}\circ G_{1}^{1}(n)$. But then $f_{12}(b)\mathring{\subseteq}_{2}f_{12}\circ G_{1}^{1}(n)$
follows, since this is by definition $f_{21}\circ f_{12}(b)\mathring{\subseteq}_{1}f_{21}\circ f_{12}\circ G_{1}^{1}(n)$,
which holds by $b\mathring{\subseteq}_{1}G_{1}^{1}(n)$ and because
$f_{21}\circ f_{12}$ preserves $\mathring{\subseteq}_{1}$. 

(\ref{enu:Cond Inter }) We construct a function $G_{2}^{2}$ that
computes intersections for $(\mathfrak{B},\beta_{2},\mathring{\subseteq}_{2})$,
based on the corresponding function $G_{2}^{1}$ associated to $(\mathfrak{B},\beta_{1},\mathring{\subseteq}_{1})$.
Define $G_{2}^{2}$ by 
\[
\forall n\in\text{dom}(\nu),\,\forall b_{1},b_{2}\in\text{dom}(\beta_{2}),\,G_{2}^{2}(n,b_{1},b_{2})=f_{12}\circ G_{2}^{1}(n,f_{21}(b_{1}),f_{21}(b_{2})).
\]
The fact that the function $G_{2}^{2}$ thus defined meets the required
criterions follows exactly as (\ref{enu:Cond X eff open }), and we
thus leave it out. 

We have thus shown that $(\mathfrak{B},\beta_{2},\mathring{\subseteq}_{2})$
is indeed a Spreen basis. The fact that $(\mathfrak{B},\beta_{1},\mathring{\subseteq}_{1})$
and $(\mathfrak{B},\beta_{2},\mathring{\subseteq}_{2})$ generate
the same Spreen topology then follows directly from Proposition \ref{prop:Equiv Basis Spreen},
as the functions $f_{21}$ and $f_{12}$ provide an equivalence of
basis between $(\mathfrak{B},\beta_{1},\mathring{\subseteq}_{1})$
and $(\mathfrak{B},\beta_{2},\mathring{\subseteq}_{2})$. 
\end{proof}

\subsection{\label{subsec:Effective-continuity-and preimage of basic open set }Effective
continuity and preimage of basic open sets}

One of the advantages of working with a basis $\mathfrak{B}_{2}$
for a topology $\mathcal{T}_{2}$ on a set $Y$ is the following classical
statement: 
\begin{fact}
For any topological space $(X,\mathcal{T}_{1})$ and function $f:X\rightarrow Y$,
$f$ is continuous if and only if the preimage by $f$ of a basic
open set of $\mathfrak{B}_{2}$ is open in $X$. 
\end{fact}

When does an effective equivalent of this holds for the different
notions of effective bases we have introduced? For Lacombe bases,
always. For Nogina bases, this holds whenever the set $X$ is also
equipped with a Nogina topology (whereas the above classical statement
makes no reference to what the topology on $X$ should be). For Spreen
bases, this holds whenever $X$ is equipped with a Spreen basis \emph{and}
when the function that computes the preimage of a basic open set of
$Y$ respects the formal inclusion relation. We state and prove these
easy facts, as they will be useful later on. 
\begin{prop}
If $(X,\nu,\mathcal{T}_{1},\tau_{1})$ is a computable topological
space and $(Y,\mu)$ is equipped with a Lacombe basis $(\mathfrak{B}_{2},\beta_{2})$,
then a $(\nu,\mu)$-computable map is effectively continuous if and
only if $f^{-1}:\mathfrak{B}_{2}\rightarrow\mathcal{T}_{1}$ is $(\beta_{2},\tau_{1})$
computable. 
\end{prop}

\begin{proof}
We have to check that the hypotheses are sufficient to compute $f^{-1}$,
seen as mapping Lacombe sets to open sets of $\mathcal{T}_{1}$. To
compute the preimage of a Lacombe set $L=\cup_{i\in\mathbb{N}}B_{i}$,
compute $f^{-1}(B_{i})$ for each $i$, then compute the union $\cup f^{-1}(B_{i})$
in $\mathcal{T}_{1}$, this can be done because $(\mathcal{T}_{1},\tau_{1})$
is a computable topological space. 
\end{proof}
\begin{prop}
If $(X,\nu,\mathcal{T}_{1},\tau_{1})$ is a computable topological
space coming from a Nogina basis $(\mathfrak{B}_{1},\beta_{1})$,
and if $(Y,\mu,\mathcal{T}_{2},\tau_{2})$ is a computable topological
space coming from a Nogina basis $(\mathfrak{B}_{2},\beta_{2})$,
then a $(\nu,\mu)$-computable map is effectively continuous if and
only if $f^{-1}:\mathfrak{B}_{2}\rightarrow\mathcal{T}_{1}$ is $(\beta_{2},\tau_{1})$-computable. 
\end{prop}

\begin{proof}
Suppose that $f^{-1}:\mathfrak{B}_{2}\rightarrow\mathcal{T}_{1}$
is $(\beta_{2},\tau_{1})$-computable. Let $O_{2}$ be a Nogina open
of $Y$, given by a pair $(A_{2},F_{2})$. We compute its preimage
$O_{1}$ as a pair $(A_{1},F_{1})$. $A_{1}$ is just $f^{-1}(A_{2})$
given as a semi-decidable set. To compute $F_{1}(x)$, proceed as
follows: compute $f(x)$, then $F_{2}(f(x))$, which gives a basic
open set around $f(x)$, then compute the preimage of this basic open
set by $f^{-1}$. This indeed provides a function $F_{1}$ for $O_{1}$. 
\end{proof}
\begin{prop}
\label{prop:Spreen Preim basic set}If $(X,\nu,\mathcal{T}_{1},\tau_{1},\mathring{\subseteq}_{1})$
is a computable topological space coming from a Spreen basis $(\mathfrak{B}_{1},\beta_{1},\mathring{\subseteq}_{1})$,
and if $(Y,\mu,\mathcal{T}_{2},\tau_{2},\mathring{\subseteq}_{2})$
is defined thanks to a Spreen basis $(\mathfrak{B}_{2},\beta_{2},\mathring{\subseteq}_{2})$,
then a $(\nu,\mu)$-computable map is effectively continuous if $f^{-1}:\mathfrak{B}_{2}\rightarrow\mathcal{T}_{1}$
is $(\beta_{2},\tau_{1})$-computable and if it can be computed by
a map increasing with respect to the formal inclusions. 
\end{prop}

\begin{proof}
Suppose that $f^{-1}:\mathfrak{B}_{2}\rightarrow\mathcal{T}_{1}$
is $(\beta_{2},\tau_{1})$-computable, and that it is computed by
a map increasing for $\mathring{\subseteq}$. Let $O_{2}$ be a Spreen
open of $Y$, given by a pair $(A_{2},F_{2})$. We compute its preimage
$O_{1}$ as a pair $(A_{1},F_{1})$. $A_{1}$ is just $f^{-1}(A_{2})$
given as a semi-decidable set. To compute $F_{1}(x)$, proceed as
follows: compute $f(x)$, then $F_{2}(f(x))$, which gives a sequence
of basic open sets around $f(x)$, then compute the preimage of this
sequence as the sequence of preimages of each basic open set by $f^{-1}$.
This provides a function $F_{1}$ for $O_{1}$, we prove that it satisfies
the required conditions. 

The computation of $F_{1}(x)$ does not depend on a given name of
$x$, up to formal inclusion, this follows directly from the fact
that this holds for $F_{2}$ and from the fact that $f^{-1}:\mathfrak{B}_{2}\rightarrow\mathcal{T}_{1}$
preserves formal inclusion. 

The $\circledcirc$ condition follows as well. Let $x$ be a point
of $A_{1}$. Apply the $\circledcirc$ condition for $F_{2}$ at $f(x)$,
this gives a neighborhood $B_{2}$ of $f(x)$. We then take $B_{1}=f^{-1}(B_{2})$
as a witness of the $\circledcirc$ condition at $x$. For any $y$
in $B_{1}$, $f(y)$ belongs to $B_{2}$, and thus $F_{2}(y)$ formally
contains $B_{2}$. Thus $f^{-1}(F_{2}(y))$ formally contains $B_{1}$,
as $f^{-1}$ respects formal inclusion. 
\end{proof}

\subsection{\label{subsec:Lacombe-bases-define-Spreen-Bases}Lacombe bases define
Spreen bases}

The definition of a Lacombe basis on a subnumbered set $(X,\nu)$
does not have to change between the $\nu$ onto/non-onto cases. In
fact, for a Lacombe basis to define a computable topology, \emph{we
do not even have to suppose that the $\nu$-computable points are
dense}, i.e. meet every non-empty basic set. This allows Grubba and
Weihrauch to talk about ``pointless topology'' in \cite{Weihrauch2009ElementaryCT},
and Bauer to talk about a ``point-free approach'' in \cite{Bauer2000}. 

This is not the case for Spreen topologies, we can reconstruct open
sets only thanks to their intersection with the set of computable
points, it is in this case necessary to suppose that computable points
are dense to prove that a Spreen basis defines a computable topology. 
\begin{lem}
\label{Lem: Lacombe basis is Spreen basis} Let $(X,\nu)$ be a subnumbered
set. Any effective Lacombe basis on $(X,\nu)$ is a Spreen basis as
soon as computable points are dense. 
\end{lem}

\begin{proof}
Let $(\mathfrak{B},\beta)$ be a Lacombe basis. First, we define a
formal inclusion for this basis: put $n\mathring{\subseteq}m\iff n=m$
(and so we also have $n\mathring{\sim}m\iff n=m$). We then follow
the points of Definition \ref{def: Spreen Basis Text}. 

(\ref{enu: Cond SD }) The elements of a Lacombe basis are indeed
semi-decidable. 

(\ref{enu:Cond Dense }) That computable points are dense is an assumption
we have added. 

(\ref{enu:Cond X eff open }) We define a function $G_{1}$ that will
prove that $X$ is effectively open. By assumption, $X$ can be expressed
as a Lacombe set: $X=\bigcup_{b\in A}\beta(b)$, where $A$ is a r.e.
subset of $\mathbb{N}$. $G_{1}$ is computed as follow: given the
name $n$ of a point $x$ of $X$, it lists all the basic open sets
of $(\beta(b))_{b\in A}$ that contain $x$. This is of course computable. 

Condition $\boxed{\mathring{\sim}}$ is easily checked: given two
different names for $x$, the sequences produced by $G_{1}$ on these
names will be identical up to a permutation, and so $\mathring{\sim}$-equivalent. 

Condition $\circledcirc$ is also easy to check. Let $x$ be a point
of $X$. Since $X=\bigcup_{b\in A}\beta(b)$, $x$ belongs to $\beta(b_{1})$,
for some $b_{1}$ in $A$. This number $b_{1}$ shows that $\circledcirc$
is satisfied: for any $y=\nu(m)$ in $\beta(b_{1})$, the list given
by $G_{1}(m)$ will contain $b_{1}$, and thus $b_{1}\mathring{\subseteq}G_{1}(m)$. 

(\ref{enu:Cond Inter }) We now have to define a function $G_{2}$
that computes intersections. Of course we use the function $\mathcal{A}$
that computes intersections for $(\mathfrak{B},\beta)$ as a Lacombe
basis. Given $b_{1}$, $b_{2}$ and $n$ such that $\nu(n)\in\beta(b_{1})\cap\beta(b_{2})$,
$G_{2}$ proceeds as follows: it decomposes $\beta(b_{1})\cap\beta(b_{2})$
as a Lacombe set, using $\mathcal{A}$, it obtains a sequence $(b_{k})_{k\ge3}$,
it then extracts from this sequence all basic open sets that contain
$x$ (using the fact that they are uniformly semi-decidable), outputting
$(b_{k})_{k\ge3\,\&\,\nu(n)\in\beta(b_{k})}$ (in a certain order). 

We now show that $G_{2}$ thus defined satisfies the required conditions
of Spreen bases. 

Condition $\boxed{\mathring{\sim}}$: given another name $m$ of $x=\nu(n)$,
and names $b_{1}$ and $b_{2}$ for $\beta(b_{1})$ and $\beta(b_{2})$,
the list that $G_{2}$ will output for $G_{2}(m,b_{1},b_{2})$ can
only be a permutation of that obtained for $G_{2}(n,b_{1},b_{2})$,
and thus they are $\mathring{\sim}$-equivalent. (Note that in general
we have to check that $G_{2}$ is increasing for $\mathring{\subseteq}$
in its last two variables, but since here $\mathring{\subseteq}$
is the equality relation, $=$, this is automatic.)

Condition $\circledcirc$. Let $b_{1}$ and $b_{2}$ be elements of
$\text{dom}(\beta)$, and $x$ be any element of $\beta(b_{1})\cap\beta(b_{2})$.
When computing the intersection $\beta(b_{1})\cap\beta(b_{2})$, as
a Lacombe set $\bigcup_{b\in A}\beta(b)$, thanks to $\mathcal{A}$,
we know that some $\beta(b_{3})$ with $b_{3}\in A$ must contain
$x$. In this case, $b_{3}$ will witness condition $\circledcirc$.
Indeed, given any computable point $y=\nu(m)$ in $\beta(b_{3})$,
the computation of $G_{2}(m,b_{1},b_{2})$ will yield a sequence that
contains $b_{3}$, and thus we indeed have $b_{3}\mathring{\subseteq}G_{2}(m,b_{1},b_{2})$. 
\end{proof}
We can then describe the Spreen topology defined thanks to a Lacombe
basis: the effective open sets are composed of a semi-decidable set
of computable points, and of a function that, given a point $x$,
outputs a sequence of basic open sets that all contain $x$. This
is a Lacombe set ``centered at $x$''. 

Note that the formal inclusion relation, as extended to these effective
open sets, exactly says the following: a name $n_{1}$ for an open
set $O_{1}$, which encodes a pair $(A_{1},F_{1})$, is formally included
in $n_{2}$, which encodes $O_{2}$ thanks to a pair $(A_{2},F_{2})$,
if for any point $x$ of $A_{1}$, the computation of $F_{1}$ at
$x$ yields a sequence of names of basic open sets which is a subsequence
of the sequence computed by $F_{2}$ at $x$. And thus $n_{1}\mathring{\sim}n_{2}$
if and only if at each point, the sequences computed by $F_{1}$ and
$F_{2}$ are just permutations one of the other. 
\begin{lem}
\label{lem:Lacombe set into Spreen set }Let $(\mathfrak{B},\beta)$
be a Lacombe basis on $(X,\nu)$, assume that the $\nu$-computable
points are dense. Then any Lacombe set for the Lacombe topology generated
by $(\mathfrak{B},\beta)$ is an effective open set for the Spreen
topology generated by $(\mathfrak{B},\beta,=)$ seen as a Spreen basis.
And this statement is uniform: there is a process that transforms
the code of a Lacombe set into the code of this same set as a Spreen
open set.
\end{lem}

The proof is a standard verification. 
\begin{proof}
Recall that $(\mathfrak{B},\beta)$ can be seen as a Spreen basis
by Lemma \ref{Lem: Lacombe basis is Spreen basis}, with the formal
inclusion relation $n\mathring{\subseteq}m\iff n=m$. 

Consider a Lacombe set $L=\bigcup_{i\in\mathbb{N}}\beta(b_{i})$,
where $(b_{i})_{i\in\mathbb{N}}$ is a computable sequence. We transform
it into a Spreen open set. First it is easy to obtain the code of
$L\cap X_{\nu}$ as a semi-decidable set. Then we define a function
$F$ which maps a point $x=\nu(n)$ to a sequence of basic open sets
that contain it and are contained in $L$. Just define $F(n)$ to
be the sequence $\{b_{i}\,\vert\,\nu(n)\in\beta(b_{i})\}$, which
is enumerable because the basic open sets are uniformly semi-decidable. 

Now we ensure conditions denoted $\boxed{\mathring{\sim}}$ and $\circledcirc$. 

If $\nu(n)=\nu(m)$, then the sequences obtained for $F(n)$ and $F(m)$
are just permutation one of the other, and so they are $\mathring{\sim}$-equivalent,
thus $\boxed{\mathring{\sim}}$ is satisfied.

Consider now a point $x$ of $L=\bigcup_{i\in\mathbb{N}}\beta(b_{i})$.
There must be $i$ such that $x$ belongs to $\beta(b_{i})$. Then,
for any $n$ in $\text{dom}(\nu)$ such that $\nu(n)\in\beta(b_{i})$,
the computation of $F(n)$ will yield a sequence which contains $b_{i}$.
This shows that $b_{i}\mathring{\subseteq}F(n)$, and thus that the
condition $\circledcirc$ is satisfied. 
\end{proof}
As before, this topology can have more effective open sets than the
Lacombe topology. 
\begin{prop}
Suppose that a function $f:X\rightarrow Y$ between computable topological
spaces equipped with Lacombe bases is effectively continuous. Suppose
that computable points are dense in both spaces, so that we can equip
$X$ and $Y$ with Spreen topologies, by Proposition \ref{Lem: Lacombe basis is Spreen basis}.
Then the function $f$ is also effectively continuous with respect
to the Spreen topologies on $X$ and $Y$, and it can be computed
by a function that respects the formal inclusion relation. 
\end{prop}

\begin{proof}
Denote by $(\mathfrak{B}_{2},\beta_{2})$ the numbered basis on $Y$.
Denote by $\tau_{1}$ the subnumbering of the Spreen topology on $X$.

We will use Proposition \ref{prop:Spreen Preim basic set}: it suffices
to check that the map $f^{-1}:\mathfrak{B}_{2}\rightarrow\mathcal{T}_{1}$
is $(\beta_{2},\tau_{1})$-computable by a function that respects
formal inclusion to obtain the desired result. 

But given a basic set $\beta_{2}(b)$, $f^{-1}(\beta_{2}(b))$ can
be computed as a Lacombe set, and Lemma \ref{lem:Lacombe set into Spreen set }
shows that this Lacombe set can be turned into a Spreen open. The
fact that the formal inclusion is respected is trivial: here $b_{1}\mathring{\subseteq}b_{2}\iff b_{1}=b_{2}$. 
\end{proof}

\subsection{\label{subsec:Effective-separability-Moschovakis THM}Effective separability
implies equality of Spreen and Lacombe topologies }

Being a Lacombe basis is usually more restrictive than being a Spreen
basis. And, even when a basis is both a Lacombe basis and a Spreen
basis, the corresponding computable topologies can differ. We show
that they are identical whenever the topology has a computable and
dense sequence.

Note that the results that follow cannot be turned into equivalences:
a space can have a Spreen basis that is also a Lacombe basis without
having a computable and dense sequence.
\begin{lem}
\label{lem: dense seq spreen =00003D> lacombe}Let $(\mathfrak{B},\beta,\mathring{\subseteq})$
be a Spreen basis on $(X,\nu)$. If there exists a $\nu$-computable
sequence which is dense for the topology generated by $\mathfrak{B}$,
then $(\mathfrak{B},\beta)$ is also a Lacombe basis. 
\end{lem}

\begin{proof}
There are two points we have to show: (1), that $X$ is a Lacombe
set, and (2), that it is possible, given two basics $\beta(b_{1})$
and $\beta(b_{2})$, to give a computable sequence of basic open sets
that constitutes their intersection. 

Denote by $(u_{k})_{k\in\mathbb{N}}\in\text{dom}(\nu)^{\mathbb{N}}$
a computable sequence of natural numbers such that $(\nu(u_{k}))_{k\in\mathbb{N}}$
meets every non-empty set of $\mathfrak{B}$. 

(1) Because $(\mathfrak{B},\beta)$ is a Spreen basis, there is a
program $G_{1}$ that shows that $X$ is effectively open. For each
$k$ in $\ensuremath{\mathbb{N}},$ $G_{1}(u_{k})$ is a sequence
of names of basic open sets which we denote $(b_{k,p})_{p\in\mathbb{N}}$.
Then the sequence $(b_{k,p})_{\langle k,p\rangle\in\mathbb{N}}$ obtained
by dovetailing testifies for the fact that $X$ is a computable union
of basic sets. Indeed, let $x$ be a point of $X$. By condition $\circledcirc$
for $G_{1}$, there is a basic set $B_{1}=\beta(b_{1})$ for which,
for any computable point $y=\nu(n)$ in $B_{1}$, one of the basic
sets in the sequence $G_{1}(n)$ will contain $x$. Since an element
of $(\nu(u_{k}))_{k\in\mathbb{N}}$ will be contained in $B_{1}$,
we have that indeed 
\[
X=\bigcup_{\langle k,p\rangle\in\mathbb{N}}\beta(b_{k,p}),
\]
and thus $X$ is a Lacombe set. 

(2) The proof is similar as above. Given two basic open sets $B_{1}=\beta(b_{1})$
and $B_{2}=\beta(b_{2})$, consider the double sequence $(G_{2}(u_{k},b_{1},b_{2}))_{\nu(u_{k})\in B_{1}\cap B_{2}}=(\mathfrak{b}_{k,p})_{\nu(u_{k})\in B_{1}\cap B_{2},p\in\mathbb{N}}$.
The condition $\circledcirc$ on $G_{2}$ will then ensure that we
do have the equality: 
\[
B_{1}\cap B_{2}=\bigcup_{\{(k,p)\in\mathbb{N}^{2}\,\vert\,\nu(u_{k})\in B_{1}\cap B_{2}\}}\beta(\mathfrak{b}_{k,p}).
\]

\vspace{-0.3cm}
\end{proof}
\begin{thm}
[Moschovakis' theorem on Lacombe sets is trivial for Spreen bases]\label{thm:Moschovakis-theorem-on Spreen bases}Let
$(\mathfrak{B},\beta,\mathring{\subseteq})$ be a Spreen basis on
$(X,\nu)$. If there exists a $\nu$-computable sequence which is
dense for the topology generated by $\mathfrak{B}$, then the Spreen
and Lacombe topologies generated by $(\mathfrak{B},\beta)$, respectively
seen as a Spreen and as a Lacombe basis, are computably equivalent
(i.e. they are identical). 
\end{thm}

\begin{proof}
The proof is straightforward, and similar to that of the above lemma.
All we need to show is that we can computably transform the name of
a Spreen open set into a Lacombe name for this same set. 

Denote by $(u_{k})_{k\in\mathbb{N}}\in\text{dom}(\nu)^{\mathbb{N}}$
a computable sequence of natural numbers such that $(\nu(u_{k}))_{k\in\mathbb{N}}$
is dense. 

Let $O_{1}$ be given as a Spreen set by a pair $(A,F)$. For each
$n$ in $\nu^{-1}(A)$, $F(n)$ is a sequence, we use below a functional
notation: $F(n)(p)$ is the $p$-th element of this sequence. 

The set $L=\{u_{k},\,\nu(u_{k})\in A\}$ is computably enumerable,
and by the $\circledcirc$ property for $F$, we have 
\[
\underset{u_{k}\in L}{\bigcup}\bigcup_{p\in\mathbb{N}}F(u_{k})(p)=O_{1}.
\]

\vspace{-0.5cm}
\end{proof}

\subsection{\label{subsec:Recursive-metric-spaces II }Computable metric spaces\emph{$^{*}$}
define Spreen bases }

To show that the definition of a Spreen basis is not to restrictive,
we have shown that, under the mild assumption that computable points
be dense, all Lacombe bases define Spreen bases. 

We will now show that in any computable metric space\emph{$^{*}$}
$(X,\nu,d)$, where $X_{\nu}$ is dense in $X$ but not necessarily
equal to it, the open balls centered at computable points also form
a Spreen basis. Recall that these balls do not always form a Lacombe
basis, so this result is not contained in the results of Section \ref{subsec:Lacombe-bases-define-Spreen-Bases}. 

After that, we give a direct description of the Spreen topology on
$(X,\nu,d)$ which is more natural than the one obtained by following
the definition of Spreen bases, in terms of left computable reals. 

\begin{figure}
\includegraphics[scale=0.5]{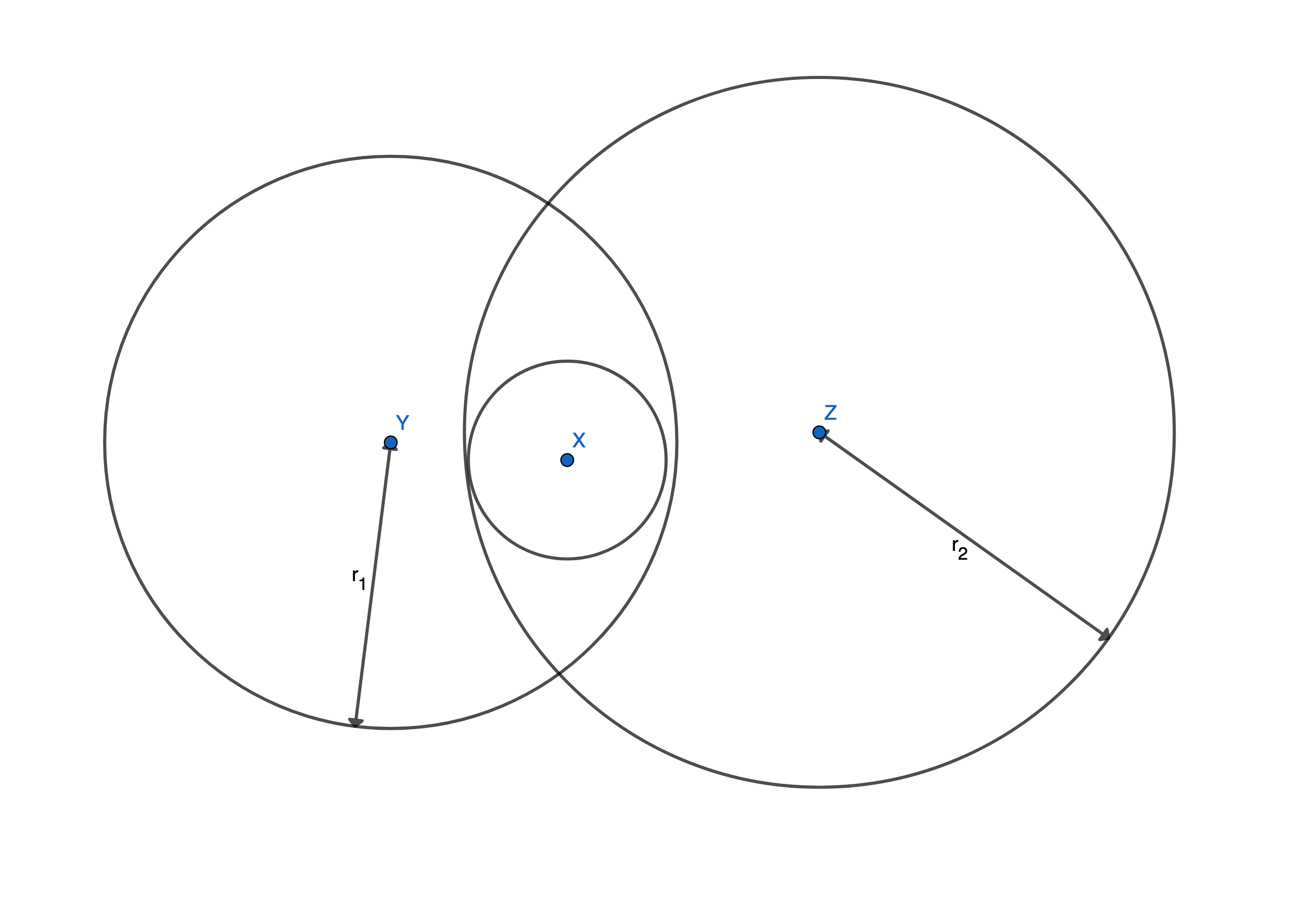}\caption{Computing an intersection in a metric space}
\end{figure}

\subsubsection{Open balls form Spreen bases}
\begin{lem}
\label{lem: Open balls Spreen Basis }Let $(X,\nu,d)$ be a computable
metric space\emph{$^{*}$}. The open balls centered at computable
points, equipped with their natural numbering that comes from $\nu\times c_{\mathbb{R}}$,
form a Spreen basis. 
\end{lem}

\begin{proof}
Denote by $\mathfrak{B}$ the set of open balls of $(X,\nu,d)$, centered
at computable points and with computable radii, and by $\beta$ the
numbering of $\mathfrak{B}$ defined as follows: 
\[
\text{dom}(\beta)=\{\langle n,m\rangle\in\mathbb{N},n\in\text{dom}(\nu)\,\&\,m\in\text{dom}(c_{\mathbb{R}})\};
\]
\[
\forall\langle n,m\rangle\in\text{dom}(\beta),\,\beta(\langle n,m\rangle)=B(\nu(n),c_{\mathbb{R}}(m)).
\]

First we define a formal inclusion relation for $(\mathfrak{B},\beta)$.
This is just the relation that comes from $(x,r_{1})\mathring{\subseteq}(y,r_{2})\iff d(x,y)+r_{1}\le r_{2}$.
Formally, it is defined on $\text{dom}(\beta)\times\text{dom}(\beta)$
by:
\[
\langle n_{1},m_{1}\rangle\mathring{\subseteq}\langle n_{2},m_{2}\rangle\iff d(\nu(n_{1}),\nu(n_{2}))+c_{\mathbb{R}}(m_{1})\le c_{\mathbb{R}}(m_{2}).
\]

We now show each point of Definition \ref{def: Spreen Basis Text}
in turn. 

(\ref{enu: Cond SD }) Of course, the open balls in a $\text{CMS}^{*}$
are uniformly semi-decidable sets. 

(\ref{enu:Cond Dense }) Computable points in a computable metric
space\emph{$^{*}$} are dense. 

(\ref{enu:Cond X eff open }) We want to describe a function $G_{1}:\text{dom}(\nu)\rightarrow\text{dom}(\beta)^{\mathbb{N}}$
that testifies for the fact that $X$ is an effectively open set.
Denote by $1_{\mathbb{R}}$ a $c_{\mathbb{R}}$-name for the real
number $1$. Then, for any $n$ in $\text{dom}(\nu)$, the natural
number $\langle n,1_{\mathbb{R}}\rangle$ is a $\beta$-name of $B(\nu(n),1)$.
Define $G_{1}(n)$ to be the constant sequence $\langle n,1_{\mathbb{R}}\rangle,\,\langle n,1_{\mathbb{R}}\rangle,\,\langle n,1_{\mathbb{R}}\rangle,...$ 

This of course satisfies the condition $\boxed{\mathring{\sim}}$
of independence on a name modulo $\mathring{\sim}$. Similarly, the
$\circledcirc$ condition is trivially satisfied. 

(\ref{enu:Cond Inter }) We finally describe a function $G_{2}:\text{dom}(\nu)\times\text{dom}(\beta)\times\text{dom}(\beta)\rightarrow\text{dom}(\beta)^{\mathbb{N}}$
which will compute intersections. 

Following Definition \ref{def: Spreen Basis Text}, $G_{2}$ should
produce a sequence of balls, it will here in fact produce a single
ball, this ball is in turn seen as a constant sequence.

Given a triple $(n,b_{1},b_{2})$, with $x=\nu(n)$, $B(y,r_{1})=\beta(b_{1})$
and $B(z,r_{2})=\beta(b_{2})$, denote by $t$ the $c_{\mathbb{R}}$-name
of $\text{min}(r_{1}-d(x,y),r_{2}-d(x,z))$, which is computable from
$(n,b_{1},b_{2})$, and we define $G_{2}(n,b_{1},b_{2})$ to be $\langle n,t\rangle$.
That is, $G_{2}(n,b_{1},b_{2})$ gives a $\beta$-name of $B(x,\text{min}(r_{1}-d(x,y),r_{2}-d(x,z)))$. 

We then show that this function satisfies the conditions $\boxed{\mathring{\sim}}$
and $\circledcirc$. 

The $\boxed{\mathring{\sim}}$ condition is easily verified: the given
name of $x$ does not change the outcome of the computation, which
is always $B(x,\text{min}(r_{1}-d(x,y),r_{2}-d(x,z)))$. Suppose $b_{1}$
and $b_{2}$ are replaced by numbers $\hat{b}_{1}$ and $\hat{b}_{2}$
that formally contain them. We have $\beta(\hat{b}_{1})=B(\hat{y},\hat{r}_{1})$
with $\hat{r}_{1}\ge r_{1}+d(y,\hat{y})$, and $\beta(\hat{b}_{2})=B(\hat{z},\hat{r}_{2})$
with $\hat{r}_{2}\ge r_{2}+d(z,\hat{z})$. 

We want to show that $\hat{r}_{1}-d(x,\hat{y})\ge r_{1}-d(x,y)$ and
that $\hat{r}_{2}-d(x,\hat{z})\ge r_{2}-d(x,z)$ both hold. 

This follows easily from the triangular inequality, with: 
\begin{align*}
 & \hat{r}_{1}\ge r_{1}+d(y,\hat{y}) & \text{ (hypothesis on \ensuremath{\hat{r}_{1}})}\\
 & -d(x,\hat{y})\ge-d(x,y)-d(y,\hat{y}) & \text{(triangular inequality)}
\end{align*}

By summing both inequalities we obtain the result, by symmetry $\hat{r}_{2}-d(x,\hat{z})\ge r_{2}-d(x,z)$
also holds. 

Finally, we show that the $\circledcirc$ condition holds: for any
point $x$ (computable or not) and any two balls $B_{1}=\beta(b_{1})$
and $B_{2}=\beta(b_{2})$, there exists $b_{3}$ in $\text{dom}(\beta)$
such that:
\[
\forall y=\nu(m)\in\beta(b_{3}),\,b_{3}\mathring{\subseteq}G_{2}(m,b_{1},b_{2}).
\]

We use the following lemma: 
\begin{lem}
\label{lem: 1/3 works for circledcirc in metric space}Fix a ball
$B=B(y,r)$, and define a map $F$ by $F(z)=r-d(y,z)$ for $z\in B$.
Let $x$ be a point of $B$. Then for any $z\in B(x,F(x)/3)$, we
have $(x,F(x)/3)\mathring{\subseteq}(z,F(z))$, i.e. $d(x,z)+F(x)/3\le F(z)$. 
\end{lem}

\begin{proof}
Suppose $z\in B(x,F(x)/3)$, i.e. $d(x,z)\le\frac{1}{3}F(x)$. We
want to show: 
\[
F(z)\ge d(x,z)+F(x)/3.
\]

We have:
\begin{align*}
F(z) & =r-d(y,z)=r-d(x,y)+d(x,y)-d(y,z)=F(x)+d(x,y)-d(y,z)\\
 & \ge F(x)-d(x,z)
\end{align*}

And 
\[
F(x)-d(x,z)\ge\frac{2}{3}F(x)\ge d(x,z)+F(x)/3.
\]
This proves the Lemma. 
\end{proof}
Fix balls $B_{1}=\beta(b_{1})=B(y,r_{1})$ and $B_{2}=\beta(b_{2})=B(z,r_{2})$,
and let $x$ be any point of $B_{1}\cap B_{2}$. 

Following Lemma \ref{lem: 1/3 works for circledcirc in metric space},
we see that, for any computable point $t=\nu(m)$, if $t\in B(x,\frac{1}{3}\text{min}(r_{1}-d(x,y),r_{2}-d(x,z)))$,
then $G_{2}(m,b_{1},b_{2})$ produces the name of a ball that contains
formally $B(x,\frac{1}{3}\text{min}(r_{1}-d(x,y),r_{2}-d(x,z)))$. 

Choose for $b_{3}$ the name of any ball contained in $B(x,\frac{1}{3}\text{min}(r_{1}-d(x,y),r_{2}-d(x,z)))$,
and which contains $x$. This exists by density of the computable
points of $X$. By Lemma \ref{lem: 1/3 works for circledcirc in metric space},
this ball proves the $\circledcirc$ condition for $G_{2}$ with respect
to $x$, $b_{1}$ and $b_{2}$. 
\end{proof}
(Note that, above, we have no way of choosing a canonical $b_{3}$:
the only canonical choice would be $B(x,\frac{1}{3}\text{min}(r_{1}-d(x,y),r_{2}-d(x,z)))$,
but when $x$ is not a computable point, this ball is not in $\mathfrak{B}$.) 

\subsubsection{Direct description of the Spreen metric topology }

By what precedes, the open balls form a Spreen basis, and thus we
do define a computable topology on $(X,\nu,d)$ by the following:
\begin{itemize}
\item An open set $O$ is described by a pair $(A,F)$, where $A=O\cap X_{\nu}$
is semi-decidable, and $F$ is a program that, given a point $x$
of $O$, produces a sequence of balls that all contain $x$ and are
contained in $O$, and which satisfies the conditions denoted $\boxed{\mathring{\sim}}$
and $\circledcirc$ in Definition \ref{def: Spreen Top text }. 
\end{itemize}
Note however that the program $F$ is not expected, given a point
$x$, to produce balls centered at $x$. We give an equivalent description
of the metric topology where the produced balls are centered at $x$.
What interests us is that we can replace the condition that a sequence
of balls is produced by the fact that a single real number is produced,
\emph{but given as} \emph{a left computable real. }In this case, the
condition $\boxed{\mathring{\sim}}$ of respect of the formal inclusion
relation exactly says that the left computable real computed by $F$
is independent of the name of a point given as input: $F$ is an actual
function, and not a multi-function, as in the case in the Nogina metric
topology. 

We denote by $c_{\nearrow}$ the subnumbering of $\mathbb{R}$ associated
to left computable reals, as introduced in the preliminaries, Section
\ref{subsec:Computable-real-numbers}. 
\begin{defn}
[Direct definition of the metric topology] \label{def:Direct-definition-of-Metric-TOP}Let
$(X,\nu,d)$ be a computable metric space\emph{$^{*}$}. We define
a computable topology $(\mathcal{T},\tau,\mathring{\subseteq})$ on
$(X,\nu)$ thanks to the following. $\mathcal{T}$ is the metric topology
on $X$, and $\tau$ is a subnumbering of $\mathcal{T}$. The description
of an open set $O$ for $\tau$ is an encoded pair $(n,m)$, where
$n$ gives the code of $O\cap X_{\nu}$ as a semi-decidable set, and
where $m$ encodes a $(\nu,c_{\nearrow})$-computable function $F:O\cap X_{\nu}\rightarrow\mathbb{R}$
which satisfies the following:
\[
\forall x\in O\cap X_{\nu},\,B(x,F(x))\subseteq O;
\]
\[
\circledcirc\,\,\forall x\in O,\,\exists r>0,\,\forall y\in B(x,r)\cap X_{\nu},\,F(y)>r.
\]
Finally, define a formal inclusion by the following. Suppose $O_{1}$
and $O_{2}$ are effectively open sets of $X$ given by pairs $(n_{1},m_{1})$
and $(n_{2},m_{2})$ that encode pairs $(O_{1}\cap X_{\nu},F_{1})$
and $(O_{2}\cap X_{\nu},F_{2})$ as above. Then we define $\mathring{\subseteq}$
by:
\[
\langle n_{1},m_{1}\rangle\mathring{\subseteq}\langle n_{2},m_{2}\rangle\iff\forall y\in O_{1}\cap X_{\nu},\,y\in O_{2}\cap X_{\nu}\,\&\,F_{1}(y)\le F_{2}(y).
\]
\end{defn}

We do not give details for the following straightforward proposition: 
\begin{prop}
\label{prop: Spreen top equiv direct description }On a $\text{CMS}^{*}$
$(X,\nu,d)$, the computable topology generated by seeing the open
balls as a Spreen basis is equivalent to the computable topology described
directly in Definition \ref{def:Direct-definition-of-Metric-TOP},
this is witnessed by functions increasing with respect to formal inclusion. 
\end{prop}

\subsubsection{Metric continuity in the general setting }

The definition of a Nogina basis described in Section \ref{part:Two-notions-of-Bases}
is the most general notion of a basis, obtained thanks to the most
natural effectivization of the statement 
\[
(A)\,\,\,\,\forall x\in O,\,\exists B\in\mathfrak{B},\,x\in B\,\&\,B\subseteq O,
\]
for $O$ an open set in a topology defined by a basis $\mathfrak{B}$.

This notion of effective open yields a notion of effective continuity,
and it was shown that, in metric spaces, this notion of effective
continuity is equivalent to the one obtained thanks to the plain effectivization
of the metric continuity statement, for a function $f:X\rightarrow Y$
between metric spaces:
\[
(B)\,\,\,\,\forall x\in X,\,\forall\epsilon>0,\,\exists\delta>0,\,\forall y\in X,\,d(x,y)<\delta\implies d(f(x),f(y))<\epsilon.
\]
In both cases, we can talk about a ``plain effectivization'', because
the definitions considered in Section \ref{part:Two-notions-of-Bases}
were obtained by adding no additional conditions to the effectivization
of the statements (A) and (B) written above.

On the contrary, the notion of effective open set we consider in the
present section, while still obtained as an effectivization of statement
(A), has many more conditions. But all the conditions that are added
to the plain effectivization are conditions that have no classical
equivalent, or which are classically always true. 

Since we have added additional conditions on effective open sets,
and since we have a new definition of the computable topology on a
metric space, we will have a new, non-basic, version of ``effective
metric continuity'', which will include a non vanishing condition
(denoted $\circledcirc$). We describe this now. 

Recall that $c_{\nearrow}$ is the subnumbering of $\mathbb{R}$ associated
to left computable reals. Below, we denote by $\mathbb{R}_{c_{\nearrow}}^{+}$
the set of strictly positive left computable reals. 
\begin{defn}
\label{def:Spreen Metric Continuity }A function $f:X\rightarrow Y$
between computable metric spaces\emph{$^{*}$} $(X,\nu,d)$ and $(Y,\mu,d)$
is \emph{formally effectively metric continuous} if there exists a
$(\nu\times c_{\nearrow},c_{\nearrow})$-computable function $\phi:X\times\mathbb{R}_{c_{\nearrow}}^{+}\rightarrow\mathbb{R}_{c_{\nearrow}}^{+}$
such that 
\[
\forall x\in X_{\nu},\forall y\in X,\forall\epsilon\in\mathbb{R}_{c_{\nearrow}}^{+},\,d(x,y)<\phi(x,\epsilon)\implies d(f(x),f(y))<\epsilon;
\]
and which satisfies the following non-vanishing condition: 
\begin{align*}
\circledcirc\,\,\forall x\in X,\,\forall\epsilon\in\mathbb{R}_{c_{\nearrow}}^{+}, & \,\exists r>0,\,\forall y\in B(x,r)\cap X_{\nu},\,\phi(y,\epsilon)\ge r.
\end{align*}
\end{defn}

Notice that in this definition we use a function for $\phi$, as opposed
to a multi-function which may yield different results for different
names of a same point. 
\begin{lem}
In Definition \ref{def:Spreen Metric Continuity }, we can suppose
WLOG that the function $\phi$ is increasing in its second variable. 
\end{lem}

\begin{proof}
This is a simple application of the following general principle: if
$f:\mathbb{R}_{c_{\nearrow}}^{+}\rightarrow\mathbb{R}_{c_{\nearrow}}^{+}$
is $(c_{\nearrow},c_{\nearrow})$-computable, then so is the function
$\hat{f}$ defined by $x\mapsto\underset{0<y\le x}{\text{sup}}f(x)$.
And this statement is uniform: it is possible to obtain a code for
$\hat{f}$ from a code for $f$. This follows from the following facts,
stated in Section \ref{part: Preliminaries}: 
\begin{itemize}
\item the set of left computable reals between $0$ and $x$ is $c_{\nearrow}$-computably
enumerable, 
\item the function $\text{sup}$ is computable for computable sequences
of left computable reals. 
\end{itemize}
\vspace{-0,5cm}
\end{proof}
We will thus always use increasing moduli of continuity in what follows.

We will now prove a theorem of equivalence between formal effective
metric continuity and effective continuity with respect to a Spreen
topology generated by open balls. 
\begin{thm}
\label{thm: Spreen metric continuity <=00003D> Spreen continuity }Let
$f:X\rightarrow Y$ be a function between computable metric spaces\emph{$^{*}$}
$(X,\nu,d)$ and $(Y,\mu,d)$. Then, $f$ is formally effectively
metric continuous if and only if it is effectively continuous with
respect to the Spreen topologies induced by the metrics on $X$ and
$Y$. 
\end{thm}

In what follows, we use the improper notation $B(x,r)\mathring{\subseteq}B(y,r')$
to mean $d(x,y)+r\le r'$. This is legitimate when there is no ambiguity
on the given $x$, $r,$ $y$ and $r'$. 
\begin{proof}
\textbf{First implication. }

We first suppose that $f$ is formally metric continuous, and show
it is effectively continuous for the Spreen topologies induced by
the metrics on $X$ and $Y$. 

By hypothesis, there is a modulus of continuity $\phi:X\times\mathbb{R}_{c_{\nearrow}}\rightarrow\mathbb{R}_{c_{\nearrow}}$,
which is increasing in its second variable. 

We use the direct description of the Spreen topology on $X$: an open
set is given by a pair $(A,F)$, where $A$ is a semi-decidable set,
and $F:A\rightarrow\mathbb{R}_{c_{\nearrow}}$ is a computable function
that, given a point of $A$, produces a radius around it. It should
satisfy the appropriate non-vanishing condition. 

Let $O_{Y}$ be a Spreen open in $Y$, given by a pair $(A_{Y},F_{Y})$.
We compute $O_{X}=f^{-1}(O_{Y})$ as a Spreen open. 

A name for $A_{X}=f^{-1}(A_{Y})$ as a semi-decidable set can be obtain
directly. 

We now describe a function $F_{X}$ associated to $O_{X}$. Given
a point $x\in A_{X}$, we put
\[
F_{X}(x)=\phi(x,F_{Y}(f(x))).
\]

This is indeed a map from $A_{X}$ to $\mathbb{R}_{c_{\nearrow}}$,
and we always have $B(x,F_{X}(x))\subseteq O_{X}$. This second fact
comes from the property of $\phi$: for any $z$ in $B(x,F_{X}(x))$,
$d(f(x),f(z))<F_{Y}(f(x))$, and so $f(z)\in O_{Y}$. 

We check condition $\circledcirc$ for $F_{X}$, which comes directly
from the $\circledcirc$ conditions for $F_{Y}$ and $\phi$. Let
$x$ be any point of $O_{X}$. Apply condition $\circledcirc$ for
$O_{Y}$ at $f(x)$: there is a radius $r_{1}$ such that for any
computable point $y$ in $B(f(x),r_{1})$, we have $B(f(x),r_{1})\mathring{\subseteq}B(y,F_{Y}(y))$.
Apply now the $\circledcirc$ condition for $\phi$ at $x$ and $r_{1}$:
this gives $r_{2}$ such that for any computable $z$ in $B(x,r_{2})$,
$B(x,r_{2})\mathring{\subseteq}B(z,\phi(z,r_{1}))$. 

Up to reducing $r_{2}$, we can also assume that $B(x,r_{2})\subseteq f^{-1}(B(f(x),r_{1}))$
holds. 

We now show that $B(x,r_{2})$ testifies for the non-vanishing condition
$\circledcirc$ for $F_{X}$ at $x$.

Consider any point $z$ in $B(x,r_{2})$. Then $f(z)$ belongs to
$B(f(x),r_{1})$, and so: 
\[
B(f(x),r_{1})\mathring{\subseteq}B(f(z),F_{Y}(f(z))).
\]
This implies that $F_{Y}(f(z))\ge r_{1}$. We can thus apply the hypothesis
on $\phi$, and obtain, because $\phi$ is increasing in its second
variable:
\[
B(x,r_{2})\mathring{\subseteq}B(z,\phi(z,r_{1}))\mathring{\subseteq}B(z,\phi(z,F_{Y}(f(z)))).
\]
And thus we obtain $B(z,r_{2})\mathring{\subseteq}B(z,\phi(z,F_{Y}(f(z))))$,
as was required. 

We thus have proven that $f^{-1}$ is $(\tau_{2},\tau_{1})$-computable,
witnessed by the map $F_{X}$ defined by $F_{X}(x)=\phi(x,F_{Y}(f(x)))$.
It remains to show that this map is increasing with respect to formal
inclusion, but this is straightforward: if $F_{Y}$ is replaced by
$F_{Y}'$, with $F_{Y}(y)\le F_{Y}'(y)$ for any $y$, and if $F'_{X}$
designates the map $x\mapsto\phi(x,F_{Y}'(f(x)))$, then $\forall x,\,F_{X}(x)\le F'_{X}(x)$,
because $\phi$ is increasing in its second variable. 

\textbf{Converse. }

Suppose now that $f$ is effectively continuous between $(X,\nu,d)$
and $(Y,\mu,d)$. The corresponding Spreen topologies are denoted
$(\mathcal{T}_{1},\tau_{1},\mathring{\subseteq}_{1})$ and $(\mathcal{T}_{2},\tau_{2},\mathring{\subseteq}_{2})$.
We build a modulus of continuity $\phi$ for it. 

Given $x$ in $X_{\nu}$ and $\epsilon$ in $\mathbb{R}_{c_{\nearrow}}$,
we show how to compute $\phi(x,\epsilon)$. First compute a $\tau_{2}$-name
for $B(f(x),\epsilon)$. Then compute a $\tau_{1}$-name of $f^{-1}(B(f(x),\epsilon))$,
which is of the form $(A_{1},F_{1})$, finally compute a $c_{\nearrow}$-name
of $F_{1}(x)$. 

We show that this indeed defines a modulus of continuity for $f$. 

Suppose that $d(x,z)<F_{1}(x)$. Then $z$ belongs to $f^{-1}(B(f(x),\epsilon)$,
thus $d(f(x),f(z))<\epsilon$, as should be. 

Note that the function $\phi$ thus defined is automatically increasing
in its second variable, because $f^{-1}$ is computed by a map that
should respect formal inclusion. 

Finally, we prove the $\circledcirc$ condition for $\phi$: 
\begin{align*}
\circledcirc\,\,\forall x\in X,\,\forall\epsilon\in\mathbb{R}_{c_{\nearrow}}, & \,\exists r>0,\,\forall y\in B(x,r)\cap X_{\nu},\,B(x,r)\mathring{\subseteq}B(y,\phi(y,\epsilon)).
\end{align*}

Let $x$ be a point of $X$ and $\epsilon$ in $\mathbb{R}_{c_{\nearrow}}$.
Consider the ball $B(f(x),\epsilon/6)$, which may not be centered
at a computable point. Consider a computable point $z$ close enough
of $f(x)$ so that $B(z,\epsilon/3)$ contains formally $B(f(x),\epsilon/6)$,
i.e. such that $d(z,f(x))<\epsilon/3$. 

Consider now the preimage of $B(z,\epsilon/3)$ by $f$ as a Spreen
open. It is given by a pair $(A_{z},F_{z})$. Since it contains $x$,
we can apply condition $\circledcirc$ at $x$ to it: there exists
$r$ such that for any $y$ at distance at most $r$ of $x$, $B(x,r)\mathring{\subseteq}B(y,F_{z}(y))$.

This radius $r$ shows that the condition $\circledcirc$ is satisfied
for $\phi$ at $(x,\epsilon)$. 

Indeed, let $y$ be any point of $B(x,r)$. Then $f(y)$ belongs to
$B(z,\epsilon/3)$, and thus $B(z,\epsilon/3)\mathring{\subseteq}B(f(y),\epsilon)$.
And thus $f^{-1}(B(z,\epsilon/3))\mathring{\subseteq}f^{-1}(B(f(y),\epsilon))$,
as we suppose that $f^{-1}$ is computed by a function that respects
formal inclusion. And thus if $f^{-1}(B(f(y),\epsilon))$ is described
by a pair $(A,F)$, we have $F(y)\ge F_{z}(y)$. And $B(x,r)\mathring{\subseteq}B(y,F_{z}(y))\mathring{\subseteq}B(y,F(y))$. 

But we have defined $\phi(y,\epsilon)$ to be $F(y)$, so we indeed
have that for any $y$ in $B(x,r)$, $B(x,r)\mathring{\subseteq}B(y,\phi(y,\epsilon))$. 
\end{proof}

\subsection{Spreen bases and Nogina bases with onto numberings\label{subsec:Spreen-bases-and-Nogina bases}}

We have seen that as soon as the ambient space has a computable and
dense sequence, any Spreen basis is a Lacombe basis, and the Spreen
and Lacombe topologies they generate are identical. Notice that the
proof of this result was trivial. However, the result of Moschovakis
which states the same for Nogina and Lacombe bases is non trivial
(a proof of Moschovakis' Theorem was given in Theorem \ref{thm:Moschovakis,-,-Theorem-1}). 

Where did the non-trivial part of this theorem move to? In the passage
from a Nogina basis to a Spreen basis.

The notion of Spreen basis was introduced here to define properly
a computable topology on a subnumbered set $(X,\nu)$ without having
to suppose that $\nu$ is onto, while still providing more effective
open sets than Lacombe bases, and properly generalizing the concept
of a computable metric space\emph{$^{*}$}. 

But when $\nu$ is onto, the definition of a Spreen basis is still
valid, and the topology it generates is still a computable topology. 

It is obvious that the conditions of a Spreen basis are more restrictive
than those of Nogina bases, and we can thus ask: 
\begin{itemize}
\item When is it true that a Nogina basis is automatically a Spreen basis? 
\item When is it true that the Spreen and Nogina topologies agree? 
\end{itemize}
On a computable metric space\emph{$^{*}$}, the open balls centered
at computable points, with computable radii form both a Nogina basis
and a Spreen basis. We can then interpret the theorem of Moschovakis
as follows.
\begin{thm}
Suppose that $(X,\nu,d)$ is a computable metric space\emph{$^{*}$}
with a computable and dense sequence, and an algorithm that computes
limits (for sequences that converge at exponential speed). If $\nu$
is actually onto, then the Nogina topology generated by open balls
is effectively equivalent (i.e. identical) to the Spreen topology. 
\end{thm}

\begin{proof}
Moschovakis proves directly that the Nogina topology agrees with the
Lacombe topology, but since it is always the case that a Lacombe open
can be transformed into a Spreen open, by Lemma \ref{lem:Lacombe set into Spreen set },
the result follows. 
\end{proof}

\bibliographystyle{alpha}
\bibliography{TopBiblio}

\begin{thebibliography}{GWX08}

\bibitem[Bau00]{Bauer2000}
Andrej Bauer.
\newblock {\em The realizability approach to computable analysis and topology}.
\newblock PhD thesis, School of Computer Science, Carnegie Mellon University,
  2000.

\bibitem[Bau23]{Bauer2023}
Andrej Bauer.
\newblock Spreen spaces and the synthetic kreisel-lacombe-shoenfield-tseitin
  theorem.
\newblock July 2023.

\bibitem[Ber97]{Berge1997}
Claude Berge.
\newblock {\em Topological spaces}.
\newblock Dover books on mathematics. Dover Publications, Mineola, N.Y., 1997.

\bibitem[BL12]{Bauer2012}
Andrej Bauer and Davorin Le{\v s}nik.
\newblock Metric spaces in synthetic topology.
\newblock {\em Annals of Pure and Applied Logic}, 163(2):87--100, February
  2012.

\bibitem[BP03]{Brattka2003}
Vasco Brattka and Gero Presser.
\newblock Computability on subsets of metric spaces.
\newblock {\em Theoretical Computer Science}, 305(1-3):43--76, aug 2003.

\bibitem[dB13]{Brecht2013}
Matthew de~Brecht.
\newblock Quasi-polish spaces.
\newblock {\em Annals of Pure and Applied Logic}, 164(3):356--381, March 2013.

\bibitem[dB20]{Brecht2020a}
Matthew de~Brecht.
\newblock {\em Some Notes on Spaces of Ideals and Computable Topology}, pages
  26--37.
\newblock Springer International Publishing, 2020.

\bibitem[Esc04]{Escardo2004}
Mart{\'i}n Escard{\'o}.
\newblock Synthetic topology.
\newblock {\em Electronic Notes in Theoretical Computer Science}, 87:21--156,
  November 2004.

\bibitem[FR10]{Freer2010}
Cameron Freer and Daniel Roy.
\newblock Posterior distributions are computable from predictive distributions.
\newblock {\em Journal of Machine Learning Research - Proceedings Track},
  9:233--240, 01 2010.

\bibitem[Fri58]{Friedberg1958}
Richard Friedberg.
\newblock Un contre-exemple relatif aux fonctionnelles r{\'e}cursives.
\newblock {\em Comptes rendus hebdomadaires des s{\'e}ances de l'Acad{\'e}mie
  des Sciences (Paris)}, vol. 24, 1958.

\bibitem[GKP16]{Gregoriades2016}
Vassilios Gregoriades, Tamás Kispéter, and Arno Pauly.
\newblock A comparison of concepts from computable analysis and effective
  descriptive set theory.
\newblock {\em Mathematical Structures in Computer Science}, 27(8):1414--1436,
  June 2016.

\bibitem[Grz57]{Grzegorczyk1957}
Andrzej Grzegorczyk.
\newblock On the definitions of computable real continuous functions.
\newblock {\em Fundamenta Mathematicae}, 44(1):61--71, 1957.

\bibitem[GWX08]{Wei08}
Tanja Grubba, Klaus Weihrauch, and Yatao Xu.
\newblock Effectivity on continuous functions in topological spaces.
\newblock {\em Electronic Notes in Theoretical Computer Science}, 202:237--254,
  03 2008.

\bibitem[Her96]{Hertling1996}
Peter Hertling.
\newblock Computable real functions: Type 1 computability versus type 2
  computability.
\newblock In Ker{-}I Ko, Norbert~Th. M{\"{u}}ller, and Klaus Weihrauch,
  editors, {\em Second Workshop on Computability and Complexity in Analysis,
  {CCA} 1996, August 22-23, 1996, Trier, Germany}, volume {TR} 96-44 of {\em
  Technical Report}. Unjiversity of Trier, 1996.

\bibitem[HR16]{Hoyrup2016}
Mathieu Hoyrup and Crist{\'{o}}bal Rojas.
\newblock On the information carried by programs about the objects they
  compute.
\newblock {\em Theory of Computing Systems}, 61(4):1214--1236, dec 2016.

\bibitem[IK21]{Iljazovic2021}
Zvonko Iljazovi{\'{c}} and Takayuki Kihara.
\newblock Computability of subsets of metric spaces.
\newblock In Vasco Brattka and Peter Hertling, editors, {\em Handbook of
  Computability and Complexity in Analysis}, pages 29--69. Springer
  International Publishing, 2021.

\bibitem[KK16]{KOROVINA2016}
Margarita Korovina and Oleg Kudinov.
\newblock Computable elements and functions in effectively enumerable
  topological spaces.
\newblock {\em Mathematical Structures in Computer Science}, 27(8):1466--1494,
  jun 2016.

\bibitem[Kle52]{Kleene1952}
Stephen~Cole Kleene.
\newblock {\em Introduction to Metamathematics}.
\newblock Princeton, NJ, USA: North Holland, 1952.

\bibitem[KLS57]{Kreisel1957}
Georg Kreisel, Daniel Lacombe, and Joseph~R. Shoenfield.
\newblock Partial recursive functionals and effective operations.
\newblock Constructivity in mathematics, Proceedings of the colloquium held at
  Amsterdam:pp. 290--297, 1957.

\bibitem[Kus84]{Kushner1984}
Boris~A. Kushner.
\newblock {\em Lectures on Constructive Mathematical Analysis}.
\newblock American Mathematical Society, 1984.

\bibitem[KW85]{Kreitz1985}
Christoph Kreitz and Klaus Weihrauch.
\newblock Theory of representations.
\newblock {\em Theoretical Computer Science}, 38:35--53, 1985.

\bibitem[KW08]{Kalantari2008}
Iraj Kalantari and Larry Welch.
\newblock On turing degrees of points in computable topology.
\newblock {\em {MLQ}}, 54(5):470--482, sep 2008.

\bibitem[Lac57a]{Lacombe1957}
Daniel Lacombe.
\newblock Les ensembles r{\'e}cursivement ouverts ou ferm{\'e}s, et leurs
  applications {\`a} l'analyse r{\'e}cursive.
\newblock {\em Comptes rendus hebdomadaires des s{\'e}ances de l'Acad{\'e}mie
  des Sciences (Paris)}, 245:1040--1043, September 1957.

\bibitem[Lac57b]{Lacombe1957b}
Daniel Lacombe.
\newblock Quelques proc{\'e}d{\'e}s de d{\'e}finition en topologie
  r{\'e}cursive.
\newblock In Arend Heyting, editor, {\em Constructivity in mathematics,
  Proceedings of the colloquium held at Amsterdam}, pages 129--158.
  North-Holland Publishing Company, 1957.

\bibitem[Lac64]{Lachlan1964}
Alistair~H. Lachlan.
\newblock Effective operations in a general setting.
\newblock {\em Journal of Symbolic Logic}, 29(4):163--178, dec 1964.

\bibitem[LMN23]{Lupini2023}
Martino Lupini, Alexander Melnikov, and Andre Nies.
\newblock Computable topological abelian groups.
\newblock {\em Journal of Algebra}, 615:278--327, feb 2023.

\bibitem[Mal61]{Maltsev1961}
Anatolii~I. Maltsev.
\newblock Constructive algebras {I}.
\newblock {\em Uspehi Mat. Nauk}, 1961.

\bibitem[Mal71]{Malcev1971}
Anatolii~I. Maltsev.
\newblock {\em The metamathematics of algebraic systems, collected papers:
  1936-1967}.
\newblock North-Holland Pub. Co, Amsterdam, 1971.

\bibitem[Mar63]{Markov1963}
Andrei~Andreevich Markov.
\newblock {\em On constructive functions}, chapter {i}n Twelve papers on logic
  and differential equations, pages 163--195.
\newblock American Mathematical Society Translations, 1963.

\bibitem[MN23]{Melnikov2023}
Alexander~G. Melnikov and Keng~Meng Ng.
\newblock Separating notions in effective topology.
\newblock {\em International Journal of Algebra and Computation},
  33(08):1687--1711, October 2023.

\bibitem[Mor04]{Morozov2004}
Andrei~S. Morozov.
\newblock On homeomorphisms of effective topological spaces.
\newblock {\em Siberian Mathematical Journal}, 45(5):956--968, sep 2004.

\bibitem[Mos64]{Moschovakis1964}
Yiannis Moschovakis.
\newblock Recursive metric spaces.
\newblock {\em Fundamenta Mathematicae}, 55(3):215--238, 1964.

\bibitem[MS55]{Myhill1955}
J.~Myhill and J.~C. Shepherdson.
\newblock Effective operations on partial recursive functions.
\newblock {\em Zeitschrift für Mathematische Logik und Grundlagen der
  Mathematik}, 1(4):310--317, 1955.

\bibitem[Nog66]{Nogina1966}
Elena~Yu. Nogina.
\newblock Effectively topological spaces.
\newblock {\em Doklady Akademii Nauk SSSR}, 169:28--31, 1966.

\bibitem[Nog69]{Nogina_1969}
Elena~Yu. Nogina.
\newblock Relations between certain classes of effectively topological spaces.
\newblock {\em Mathematical Notes of the Academy of Sciences of the {USSR}},
  5(4):288--294, apr 1969.

\bibitem[Pau16]{Pauly2016}
Arno Pauly.
\newblock On the topological aspects of the theory of represented spaces.
\newblock {\em Computability}, 5(2):159--180, may 2016.

\bibitem[Rau21]{RauzyV}
Emmanuel Rauzy.
\newblock Computable analysis on the space of marked groups.
\newblock {\em arXiv:2111.01179}, 2021.

\bibitem[Rau23]{Rauzy2023RPZ}
Emmanuel Rauzy.
\newblock Multi-representation associated to the numbering of a subbasis and
  formal inclusion relation.
\newblock {\em arXiv:2311.15861}, 2023.

\bibitem[Rog87]{Rogers1987}
Hartley Rogers.
\newblock {\em Theory of Recursive Functions and Effective Computability}.
\newblock MIT Press, April 1987.

\bibitem[RZ21]{Rettinger_2021}
Robert Rettinger and Xizhong Zheng.
\newblock Computability of real numbers.
\newblock In Vasco Brattka and Peter Hertling, editors, {\em Handbook of
  Computability and Complexity in Analysis}, pages 3--28. Springer
  International Publishing, 2021.

\bibitem[Sch98]{Schroeder1998}
Matthias Schr{\"o}der.
\newblock Effective metrization of regular spaces.
\newblock volume 235, pages 63--80. FernUniversit{\"a}t Hagen, 1998.

\bibitem[Sch01]{Schroeder2001}
Matthias Schr{\"o}der.
\newblock Admissible representations of limit spaces.
\newblock In {\em Computability and Complexity in Analysis}, pages 273--295.
  Springer Berlin Heidelberg, 2001.

\bibitem[Sch03]{Schroeder2003}
Matthias Schr{\"o}der.
\newblock Admissible representations for continuous computations.
\newblock 2003.

\bibitem[Sch21]{Schroeder2021}
Matthias Schr{\"o}der.
\newblock Admissibly represented spaces and qcb-spaces.
\newblock In Vasco Brattka and Peter Hertling, editors, {\em Handbook of
  Computability and Complexity in Analysis}, pages 305--346. Springer
  International Publishing, Cham, 2021.

\bibitem[Sel08]{Selivanov2008}
Victor Selivanov.
\newblock On the difference hierarchy in countably based {T}0-spaces.
\newblock {\em Electronic Notes in Theoretical Computer Science}, 221:257--269,
  dec 2008.

\bibitem[SHT08]{StoltenbergHansen2008}
Viggo Stoltenberg-Hansen and John~V. Tucker.
\newblock Computability on topological spaces via domain representations.
\newblock In {\em New Computational Paradigms}, pages 153--194. Springer New
  York, 2008.

\bibitem[Spr98]{Spr98}
Dieter Spreen.
\newblock On effective topological spaces.
\newblock {\em The Journal of Symbolic Logic}, 63(1):185--221, 1998.

\bibitem[Spr10]{Spreen_2010}
Dieter Spreen.
\newblock Effectivity and effective continuity of multifunctions.
\newblock {\em The Journal of Symbolic Logic}, 75(2):602--640, jun 2010.

\bibitem[Spr16]{SPREEN_2016}
Dieter Spreen.
\newblock Some results related to the continuity problem.
\newblock {\em Mathematical Structures in Computer Science}, 27(8):1601--1624,
  jun 2016.

\bibitem[Tay11]{Taylor2011}
Paul Taylor.
\newblock Foundations for computable topology.
\newblock In Giovanni Sommaruga, editor, {\em Foundational Theories of
  Classical and Constructive Mathematics}, number~76 in Western Ontario Series
  in Philosophy of Science. Springer-Verlag, January 2011.

\bibitem[Tur37]{Turing1937}
Alan~M. Turing.
\newblock On computable numbers, with an application to the
  {E}ntscheidungsproblem.
\newblock {\em Proceedings of the London Mathematical Society},
  s2-42(1):230--265, 1937.

\bibitem[Tur38]{Turing1938}
Alan~M. Turing.
\newblock On computable numbers, with an application to the
  {E}ntscheidungsproblem. {A} correction.
\newblock {\em Proceedings of the London Mathematical Society},
  s2-43(1):544--546, 1938.

\bibitem[Wei87]{Weihrauch1987}
Klaus Weihrauch.
\newblock {\em Computability}.
\newblock Springer Berlin Heidelberg, 1987.

\bibitem[Wei08]{Weihrauch2008}
Klaus Weihrauch.
\newblock The computable multi-functions on multi-represented sets are closed
  under programming.
\newblock {\em Journal of Universal Computer Science}, 14(6):801--844, 2008.

\bibitem[Wei13]{Weihrauch2013}
Klaus Weihrauch.
\newblock Computably regular topological spaces.
\newblock {\em Logical Methods in Computer Science}, Volume 9, Issue 3, August
  2013.

\bibitem[WG09]{Weihrauch2009ElementaryCT}
Klaus Weihrauch and Tanja Grubba.
\newblock Elementary computable topology.
\newblock {\em J. Univers. Comput. Sci.}, 15:1381--1422, 2009.

\bibitem[WS83]{Weihrauch1983}
Klaus Weihrauch and Gisela Sch{\"a}fer.
\newblock Admissible representations of effective cpo's.
\newblock {\em Theoretical Computer Science}, 26(1-2):131--147, September 1983.

\end{thebibliography}

\end{document}